\documentclass[11pt]{article}
\usepackage{amsmath,amssymb,amsthm,amsfonts,amstext,amsbsy,amscd}
\usepackage[utf8]{inputenc}
\usepackage{bbm,dsfont}
\usepackage{color}
\usepackage{comment}
\usepackage[a4paper, total={6in, 8in}]{geometry}

\usepackage[colorlinks,citecolor=blue,urlcolor=blue,linkcolor = blue,hypertexnames=false]{hyperref}

\numberwithin{equation}{section}
\theoremstyle{plain}
\newtheorem{thm}{Theorem}
\newtheorem{proposition}{Proposition}
\newtheorem{corollary}{Corollary}
\newtheorem{lemma}{Lemma}

\theoremstyle{definition}
\newtheorem{definition}{Definition}
\newtheorem{setting}{Setting}

\theoremstyle{remark}
\newtheorem{remark}{Remark}
\newtheorem{example}{Example}

\newcommand{\R}{\mathbbm{R}}
\newcommand{\N}{\mathbbm{N}}

\newcommand{\E}{\mathbb{E}}

\renewcommand{\P}{\mathbb{P}}
\newcommand{\rr}{\mathbf{r}}

\allowdisplaybreaks[4]

\begin{document}

\title{Relative perturbation bounds with applications to empirical covariance operators}
\author{Moritz Jirak\thanks{Universit\"{a}t Wien, Austria. E-mail: moritz.jirak@univie.ac.at} \qquad
Martin Wahl\thanks{Humboldt-Universit\"{a}t zu Berlin, Germany. E-mail: martin.wahl@math.hu-berlin.de
\newline
\textit{2010 Mathematics Subject Classifcation.} 60B20, 60F05, 15A42, 47A55, 62H25\newline
}}
\date{}
\maketitle
\begin{abstract}
The goal of this paper is to establish relative perturbation bounds, tailored for empirical covariance operators. Our main results are expansions for empirical eigenvalues and spectral projectors, leading to concentration inequalities and limit theorems. One of the key ingredients is a specific separation measure for population eigenvalues, which we call the relative rank, giving rise to a sharp invariance principle in terms of limit theorems, concentration inequalities and inconsistency results. Our framework is very general, requiring only $p > 4$ moments and allows for a huge variety of dependence structures.

\end{abstract}

\tableofcontents

\section{Introduction}
The empirical covariance operator is a central object in high-dimensional probability. An important question studied in this context is the behaviour of empirical eigenvalues and eigenvectors. Using the empirical covariance operator $\hat\Sigma$ as estimator for the population version $\Sigma$, one wants to ensure that empirical eigenvalues $(\hat \lambda_j)_{j\geq 1}$ and corresponding empirical eigenvectors $(\hat u_j)_{j\geq 1}$ do not deviate too much from their population counterparts $(\lambda_j)_{j\geq 1}$ and $(u_j)_{j\geq 1}$. There is, by now, quite an extensive literature in this area regarding stochastic fluctuations and perturbation bounds. A classical result along these lines is Anderson's central limit theorem. In case of eigenvalues, it states that
\begin{align}\label{intro:clt:anderson}
\sqrt{\frac{n}{\sigma^2}}\frac{\hat{\lambda}_j - \lambda_j}{\lambda_j} \xrightarrow{d} \mathcal{N}(0,1),
\end{align} where $\sigma^2>0$ is the variance of the squared $j$-th Karhunen-Loève coefficient, see e.g.~\cite{And,Dauxois_1982} and \cite{cai:2020:aos:spiked:limit,wang2017,MR4210724} for some more recent results. Another line of substantial research is dealing with high-dimensional phenomena, when the number of observations $n$ is comparable to the dimension $d$, ranging from eigenvector inconsistency to eigenvalue (upward) bias and more. A prominent example is the following: if $d/n\rightarrow \gamma>0$ and there is only a fixed number of spiked eigenvalues, the leading empirical eigenvalues and eigenvectors undergo a phase transition. For instance, if $\lambda_1$ is below a certain threshold, then $\hat{u}_1$ may even be asymptotically orthogonal to $u_1$, that is
\begin{align}\label{intro:inconsistency:eigenvectors}
\langle \hat{u}_1, u_1 \rangle^2 \xrightarrow{\P} 0,
\end{align}
see~\cite{debashis_2007} and also~\cite{baik:aop:2005,boaz:aos:2008,Benaych-Georges:advances:2011} for related contributions. While such high-dimensional phenomena are well understood in the (Gaussian) spiked covariance model, extensions to other probabilistic settings and spectral decays appear to be largely unexplored, remaining an active research area.

One of our main contribution is to demonstrate that the phenomenon of phase transition can also be observed - and characterised - in completely different scenarios. For instance, eigenvector inconsistency such as in \eqref{intro:inconsistency:eigenvectors} may also happen in cases where $d$ is significantly different from $n$. A key quantity in this context turns out to be the map \begin{align}\label{EqDefRelRank}
j \mapsto \rr_j(\Sigma)=\sum_{k\neq j} \frac{\lambda_k}{|\lambda_j - \lambda_k|}+ \frac{\lambda_j}{g_j},
\end{align}
which we refer to as the \textit{relative rank} of $\Sigma$ (we actually consider a generalisation with multiplicities). In \eqref{EqDefRelRank}, $g_j$ denotes the $j$-th spectral gap, defined as the distance of $\lambda_j$ to the rest of the spectrum of $\Sigma$. The relative rank allows us to formulate an interesting invariance principle, which, roughly speaking, goes as follows: if $\rr_j(\Sigma)$ is below a certain critical barrier, then classical results continue to hold. This result is (up to mild moment conditions) invariant with respect to the underlying probability measures. However, if $\rr_j(\Sigma)$ is above the critical barrier, things can break down, and one may even observe the aforementioned eigenvector inconsistency in \eqref{intro:inconsistency:eigenvectors}. To give a flavour of this type of invariance, consider the central limit theorem given in \eqref{intro:clt:anderson}, based on a triangular array $X_1^{(n)},\ldots, X_n^{(n)}$ of independent copies of a random variable $X^{(n)}$ with covariance operator $\Sigma^{(n)}$, $n\geq 1$. Subject to mild moment conditions, we show that \eqref{intro:clt:anderson} remains valid as long as
\begin{align}\label{intro:clt:condition}
\frac{1}{\sqrt{n}}\rr_j(\Sigma^{(n)}) \to 0,
\end{align}
and a related statement holds true for empirical eigenvectors. On the other hand, it is possible to construct a specific sequence of random variables $X^{(n)}$ (with covariance operator $\Sigma^{(n)}$), where the above implication becomes an equivalence: if \eqref{intro:clt:condition} does not hold, then the left-hand side of \eqref{intro:clt:anderson} is not tight, and even a weaker relative consistency does not hold anymore. Moreover, if \eqref{intro:clt:condition} is no longer valid, then the (leading) empirical eigenvector is not consistent anymore, and one can even observe the asymptotic orthogonality \eqref{intro:inconsistency:eigenvectors}. A similar phenomenon is true concerning high probability bounds.

The key to our limit theorems and concentration inequalities are tight relative perturbation bounds, in which the relative rank is the main characteristic. This is achieved by exploiting a contraction property for empirical spectral projectors. We require two ingredients. First, we assume that certain relative coefficients (resp.~certain relative sub-blocks) of the perturbation $\hat\Sigma-\Sigma$ are bounded by some (usually, small) value $x$. Then, if the relative rank satisfies the bound $\rr_j(\Sigma)\leq 1/(3x)$, we establish perturbation expansions for empirical eigenvalues, eigenvectors and spectral projectors. An important aspect of our first order expansions is the fact that the remainder terms are typically of smaller order than the linear perturbation terms and thus scale correctly, a property that plays an important role for our invariance principles.

The study of general perturbation bounds has a long tradition in matrix analysis, functional analysis, and operator theory. Classical perturbation bounds for eigenvalues and eigenspaces include the Weyl inequality and the Davis-Kahan $\sin \Theta$ inequality, see e.g.~\cite{MR1477662,MR2978290}. These bounds have been extended in many directions. A basic tool in perturbation theory for linear operators is the holomorphic functional calculus \cite{MR1009162,K95,C83,MR878974}. Key ingredients such as Cauchy's integral formula and the resolvent equations have been successfully applied to various stochastic perturbation problems, see e.g~\cite{KG00,MR2033885,MR2102509,MR3099123,MR2738536,KL17b} to mention a few. In classical perturbation bounds, deviations of spectral characteristics of $\hat\Sigma$ from their accompanying spectral characteristics of $\Sigma$ are usually controlled in terms of the operator norm (or other relevant norms) of the perturbation $\hat\Sigma-\Sigma$.


Regarding random matrices, a fundamental question is to find precise estimates of corresponding norms. A number of more recent results established tight bounds for the operator norm of (possibly structured) random matrices, see for instance \cite{MR3531673,MR2111932,MR3878726}. However, all those and related results do not seem to directly apply to empirical covariance operators.  Using the method of generic chaining (cf.~\cite{MR4381414}), it has been recently shown in~\cite{KL16b} that for sub-Gaussian i.i.d.~observations the size of $\|\hat\Sigma-\Sigma\|_\infty$ is characterised by $\|\Sigma\|_\infty$ and the effective rank $\rr(\Sigma)=\operatorname{tr}(\Sigma)/\|\Sigma\|_\infty$. Alternative approaches are also offered in \cite{MR3407216,bandeira2021matrix,brailovskaya2022universality}, see also \cite{V12,MR353419} for earlier, related results. Moving to a more special setup, a precise characterisation of the operator norm is possible in terms of the Tracy-Widom law, see for instance \cite{Johnstone_2000, MR2308592, TV12}.

Relative bounds already appeared in other branches of mathematics, see e.g.~the review papers \cite{I98,I00}. For instance, there are relative versions of the Weyl inequality and the Davis-Kahan $\sin \Theta$ theorem, benefitting from considering relative errors and relative spectral gaps. However, these bounds are often (substantially) sub-optimal from a probabilistic perspective and typically involve the very quantities we actually wish to control (e.g.~empirical eigenvalues). On the other hand, despite their usefulness, relative bounds appear to be a rarely studied in the context of empirical covariance operators. Only more recently, there appears to be some interest in this topic. For instance, it has been observed for problems related to empirical covariance operators that relative techniques may lead to substantial improvements over absolute ones, see \cite{M16,mas_complex_2014,RW17}. These works already noticed the relevance of the relative rank, but didn't exploit it to its full potential. Similarly,~\cite{MR3739989,VVU2011} focus on specific structures like low ranks to give improved bounds.

\subsection*{Notation}\label{SecNotation}
Let $(\mathcal{H},\langle\cdot,\cdot\rangle)$ be a separable Hilbert space of dimension $d\in\N\cup\{+\infty\}$ and let $\|\cdot\|$ denote the norm on $\mathcal{H}$, defined by $\|u\|=\sqrt{\langle u,u\rangle}$. Let $\Sigma$ be a self-adjoint and positive trace class operator on $\mathcal{H}$. By the spectral theorem, there exists a sequence $\lambda_1\ge \lambda_2\ge\dots>0$ of  positive eigenvalues (which is either finite or converges to zero), together with an orthonormal system of eigenvectors $u_1,u_2,\dots$ such that $\Sigma$ has the spectral representation $\Sigma =\sum_{j\ge 1}\lambda_j P_j$ with rank-one projectors $P_j=u_j\otimes u_j$. Here, for $u,v\in\mathcal{H}$ we denote by $u\otimes v$ the rank-one operator defined by $(u\otimes v)x=\langle v,x\rangle u$, $x\in \mathcal{H}$. We denote by $\operatorname{tr}(\cdot)$, $\|\cdot\|_1$ and $\|\cdot\|_2$ the trace, the trace norm and the Hilbert-Schmidt norm, respectively. By assumption, we have $\|\Sigma\|_1=\operatorname{tr}(\Sigma)=\sum_{j\geq 1}\lambda_j<\infty$.  Finally, for $j\geq 1$, let $g_j$ be the $j$-th spectral gap defined by $g_j=\min(\lambda_{j-1}-\lambda_j,\lambda_j-\lambda_{j+1})$ for $j\geq 2$ and $g_1=\lambda_1-\lambda_2$.

Let $\mu_1>\mu_2>\dots>0$ be the sequence of positive and distinct eigenvalues of $\Sigma$. For $r\geq 1$, let $\mathcal{I}_r=\{j\geq 1:\lambda_j=\mu_r\}$ and $m_r=|\mathcal{I}_r|$. Let $Q_r$ be the orthogonal projection onto the eigenspace corresponding to $\mu_r$, that is,
\begin{equation}\label{SpectralProj}
 Q_r=\sum_{j\in \mathcal{I}_r}P_j.
\end{equation}
Then the spectral theorem leads to $\Sigma=\sum_{r\geq 1}\mu_rQ_r$,
with convergence in trace norm (and thus also in Hilbert-Schmidt norm). Without loss of generality, we shall assume that the eigenvectors $u_1,u_2,\dots$ form an orthonormal basis of $\mathcal{H}$ such that $\sum_{r\ge 1}Q_r=I$. For $r\geq 1$, define the resolvent
\[
R_r=\sum_{s\neq r}\frac{1}{\mu_s-\mu_r}Q_s.
\]
Finally, for $r\ge 1$, we make use of the abbreviation $\operatorname{tr}_{\geq r}(\Sigma)$  for $\sum_{s\geq r}m_s\mu_s$.

Let $\hat \Sigma$ be another self-adjoint and positive trace class operator on $\mathcal{H}$. We consider $\hat\Sigma$ as a perturbed version of $\Sigma$ and write $E=\hat\Sigma-\Sigma$ for the (additive) perturbation. By the spectral theorem, there exists a sequence $\hat{\lambda}_1\ge \hat{\lambda}_2\ge\dots\ge 0$ of eigenvalues together with an orthonormal system of eigenvectors $\hat{u}_1,\hat{u}_2,\dots$ such that we can write $\hat{\Sigma}=\sum_{j\ge 1}\hat{\lambda}_j \hat P_j$ with with rank-one projectors $\hat P_j=\hat u_j\otimes\hat u_j$. For $r\geq 1$, let
\[
\hat{Q}_r=\sum_{j\in \mathcal{I}_r}\hat{P}_j.
\]
Finally, for $j,k\geq 1$, let
\[
\bar\eta_{jk}=\frac{\langle u_j, Eu_k\rangle}{\sqrt{\lambda_j\lambda_k}}.
\]
If $d$ is finite, then the $\bar\eta_{jk}$ are the coefficients of the relative perturbation $\Sigma^{-1/2}E\Sigma^{-1/2}$ with respect to the orthonormal basis given by the eigenvectors of $\Sigma$.

Throughout the paper, we use the letters $c,C$ for constants that may change from line to line (by a numerical value). If no further dependencies are mentioned, then these constants are absolute.

\section{Relative perturbation bounds}\label{SecPerturb}
In this section, we present our main relative perturbation bounds, which are the key to our invariance principles. The proofs are deferred to Section \ref{SecProofs} below. We begin with recalling the notion of the relative rank.
\begin{definition}\label{DefRelRank} For $j\in \mathbb{N}$, we define
\begin{equation*}
\rr_j(\Sigma)=\sum_{k\neq j} \frac{\lambda_k}{|\lambda_j - \lambda_k|}+\frac{\lambda_j}{g_j}.
\end{equation*}
\end{definition}
The relative rank $\rr_j(\Sigma)$ measures in a weighted way, how well $\lambda_j$ is separated from the rest of the spectrum. 
\begin{thm}\label{ThmExpSimpleEV}
Let $j\geq 1$. Suppose that $\lambda_j$ is a simple eigenvalue, meaning that $\lambda_j\neq \lambda_k$ for all $k\neq j$. Let $x>0$ be such that $|\bar\eta_{kl}|\leq x$ for all $k,l\geq 1$. Suppose that
\begin{equation}\label{EqCCond}
\rr_j(\Sigma)\leq 1/(3x).
\end{equation}
Then we have
\begin{equation}\label{EqEVE}
|\hat{\lambda}_j-\lambda_j-\lambda_j\bar{\eta}_{jj}|\leq Cx^2\rr_j(\Sigma)\lambda_j.
\end{equation}
\end{thm}

The boundedness condition on the coefficients $\bar\eta_{kl}$ measures how close $\hat \Sigma$ is to $\Sigma$. In particular, if these coefficients are small enough (in absolute value) in the sense that the relative rank condition \eqref{EqCCond} is satisfied, then we obtain a first order perturbation expansion of $\hat\lambda_j$ around $\lambda_j$.

\begin{thm}\label{ThmExpEVe}
Let $j\geq 1$. Suppose that $\lambda_j$ is a simple eigenvalue.  Let $x>0$ be such that $|\bar\eta_{kl}|\leq x$ for all $k,l\geq 1$. Suppose that \eqref{EqCCond} holds. Then we have
\begin{equation}\label{EqEVeE}
\bigg\|\hat{u}_j-u_j-\sum_{k\neq j}\frac{\sqrt{\lambda_j\lambda_k}}{\lambda_j-\lambda_k}\bar\eta_{jk}u_k\bigg\|\leq Cx^2\rr_j(\Sigma) \sqrt{\sum_{k\neq j}\frac{\lambda_j\lambda_k}{(\lambda_j-\lambda_k)^2}}
\end{equation}
In \eqref{EqEVeE}, the sign of $\hat u_j$ is chosen such that $\langle \hat u_j,u_j\rangle>0$.
\end{thm}
Theorems \ref{ThmExpSimpleEV} and \ref{ThmExpEVe} give relative remainder terms for first order perturbation expansions of $\hat\lambda_j$ and $\hat u_j$, scaling with $\lambda_j$ and $(\sum_{k\neq j}\lambda_j\lambda_k/(\lambda_j-\lambda_k)^2)^{1/2}$, respectively. A striking property is that these remainder terms are typically of smaller order than the leading perturbation terms (in fact they scale in the size of the linear perturbation term), provided that the bounds $|\bar\eta_{kl}|\leq x$, $k,l\geq 1$, are sufficiently tight. 
This is in contrast to classical perturbation bounds (see e.g.~\cite{K95}), where the remainder terms are usually controlled in terms of the operator norm (or other relevant norms) of the perturbation $\hat\Sigma-\Sigma$.

Applying the boundedness of the $\bar\eta_{jk}$ to the linear terms, we get the following corollary:

\begin{corollary} \label{CorSimple}
Let $j\geq 1$. Suppose that $\lambda_j$ is a simple eigenvalue. Let $x>0$ be such that $|\bar\eta_{kl}|\leq x$ for all $k,l\geq 1$. Suppose that Condition \eqref{EqCCond} holds. Then we have
\begin{equation*}
|\hat{\lambda}_j-\lambda_j|\leq Cx\lambda_j
\end{equation*}
and
\begin{equation*}
\|\hat{u}_j-u_j\|\leq Cx\sqrt{\sum_{k\neq j}\frac{\lambda_j\lambda_k}{(\lambda_j-\lambda_k)^2}}.
\end{equation*}
\end{corollary}

To discuss sharpness of the bounds (in the deterministic case), consider the rank-one perturbation $\hat\Sigma=\Sigma+x(v\otimes v)$ with $v=\sum_{k\geq 1}\sqrt{\lambda_k}u_k$. Then $\bar\eta_{kl}=x$ for all $k,l\geq 1$, and we see from Theorem \ref{ThmExpSimpleEV} that
\[
x-Cx^2\rr_j(\Sigma)\leq |\hat{\lambda}_j-\lambda_j|/\lambda_j\leq x+Cx^2\rr_j(\Sigma)
\]
and
\[
x-Cx^2\rr_j(\Sigma)\leq \|\hat{u}_j-u_j\|\Big/\sqrt{\sum_{k\neq j}\frac{\lambda_j\lambda_k}{(\lambda_j-\lambda_k)^2}}\leq x+Cx^2\rr_j(\Sigma)
\]
for all $x$ satisfying \eqref{EqCCond}. For $\rr_j(\Sigma)\leq c/x$ with $c>0$ small enough, the lower and upper bound reduce to $x/C$ and $Cx$, respectively, and thus coincide up to a constant factor.

In the probabilistic setting considered in Section \ref{SecApp}, the main remaining point is to show that, with high probability, the bound $\sup_{k,l\geq 1}|\bar\eta_{kl}|\leq x$ holds for some appropriate $x$. This is achieved by applying standard concentration inequalities in combination with the union bound. In particular, using this strategy, Theorems \ref{ThmExpSimpleEV} and \ref{ThmExpEVe} are not applicable if $\dim \mathcal{H}=\infty$. This drawback can be eliminated by modifying the assumption on the relative coefficients as follows. Let $j\leq d$ be such that $\lambda_j$ is a simple eigenvalue. Consider $x>0$ such that Condition \eqref{EqCCond} holds. Let $j_0\geq 1$ be such that $\lambda_{j_0}\leq \lambda_j/2$. Then inequalities \eqref{EqEVE} and \eqref{EqEVeE} hold provided that for all $k,l< j_0$,
\begin{equation}\label{DefEventCoeff}
|\bar{\eta}_{kl}|,\sqrt{\frac{\sum_{m\geq j_0}\lambda_k\lambda_m\bar{\eta}_{km}^2}{\sum_{m\geq j_0}\lambda_k\lambda_m}},\sqrt{\frac{\sum_{m\geq j_0}\sum_{n\geq j_0}\lambda_m\lambda_n\bar{\eta}_{mn}^2}{\sum_{m\geq j_0}\sum_{n\geq j_0}\lambda_m\lambda_n}}\leq x.
\end{equation}
We give a more general construction, which accounts also for multiplicities of the eigenvalues of $\Sigma$. With a minor abuse of notation, we define in the case of multiple eigenvalues.

\begin{definition}
For $r\geq 1$, we define
\begin{align*}
\rr_r(\Sigma)=\sum_{s\neq r}\frac{m_s\mu_s}{|\mu_r-\mu_s|}+\frac{m_r\mu_r}{\min(\mu_{r-1}-\mu_{r},\mu_r-\mu_{r+1})}.
\end{align*}
\end{definition}

\begin{thm}\label{ThmExpEvMult}
Let $r\geq 1$. Consider $r_0\geq 1$ such that $\mu_{r_0}\leq \mu_r/2$. Let $x>0$ be such that for all $s,t< r_0$,
\begin{equation}\label{EqBoundRelCoeff}
\frac{\|Q_sEQ_t\|_2}{\sqrt{m_s\mu_sm_t\mu_t}},
\frac{\|Q_sEQ_{\geq r_0}\|_2}{\sqrt{m_s\mu_s\operatorname{tr}_{\geq r_0}(\Sigma)}},
\frac{\|Q_{\geq r_0}EQ_{\geq r_0}\|_2}{\operatorname{tr}_{\geq r_0}(\Sigma)}\leq x.
\end{equation}
Suppose that
\begin{equation}\label{EqCCondMult}
\rr_r(\Sigma)\leq 1/(6x).
\end{equation}
Then we have
\begin{equation}\label{EqEVEMult}
\sum_{k=1}^{m_r}|\lambda_k(\hat Q_r(\hat \Sigma-\mu_rI)\hat Q_r)-\lambda_k(Q_rEQ_r)|\leq Cx^2\rr_r(\Sigma)m_r\mu_r,
\end{equation}
where $\lambda_k(\cdot)$ denotes the $k$-th largest eigenvalue.
\end{thm}
The key point is that the linear approximation of $(\hat\lambda_j/\mu_r-1)_{j\in\mathcal{I}_r}$ is given by the first $m_r$ eigenvalues of $(1/\mu_r)Q_rEQ_r$. For instance, we have the following corollary.

\begin{corollary}
Let $r\geq 1$. Consider $r_0\geq 1$ such that $\mu_{r_0}\leq \mu_r/2$. Let $x$ be such that \eqref{EqBoundRelCoeff} holds. Moreover, suppose that Condition \eqref{EqCCondMult} holds. If $j$ is the smallest integer in $\mathcal{I}_r$, then
\[
|\hat \lambda_j-\mu_r-\lambda_1(Q_rEQ_r)|\leq Cx^2\rr_r(\Sigma)m_r\mu_r.
\]
\end{corollary}

If $|\bar\eta_{kl}|\leq x$ for all $k,l\geq 1$, then \eqref{EqBoundRelCoeff} holds, as can be seen from inserting \eqref{SpectralProj} and squaring out the Hilbert-Schmidt norms. Moreover, \eqref{DefEventCoeff} implies \eqref{EqBoundRelCoeff} if $j_0$ and $r_0$ are related by $\lambda_{j_0}\leq \mu_{r_0}$ and they coincide  if all eigenvalues are simple and $r_0=j_0$ holds. Theorem \ref{ThmExpEvMult} generalises Theorem~\ref{ThmExpSimpleEV}. In fact, in the case when $m_r=1$ with $\mathcal{I}_r=\{j\}$, then  \eqref{EqCCondMult} is satisfied if $\rr_j(\Sigma)\leq 1/(6x)$ and \eqref{EqEVEMult} boils down to \eqref{EqEVE}.

\begin{thm}\label{ThmProjExp}
Let $r\geq 1$. Consider $r_0\geq 1$ such that $\mu_{r_0}\leq \mu_r/2$. Let $x$ be such that \eqref{EqBoundRelCoeff} holds. Moreover, suppose that Condition \eqref{EqCCondMult} holds. Then we have
\begin{equation}\label{EqProjE}
\|\hat{Q}_r-Q_r+R_rEQ_r+Q_rER_r\|_2\leq Cx^2\rr_r(\Sigma) \sqrt{\sum_{s\neq r}\frac{m_r\mu_rm_s\mu_s}{(\mu_r-\mu_s)^2}}
\end{equation}
\end{thm}

Applying \eqref{EqBoundRelCoeff} to the linear terms, we get the following corollary:

\begin{corollary}\label{CorConcentrationMulti}
Let $r\geq 1$. Consider $r_0\geq 1$ such that $\mu_{r_0}\leq \mu_r/2$. Let $x$ be such that \eqref{EqBoundRelCoeff} holds. Moreover, suppose that Condition \eqref{EqCCondMult} holds. Then we have
\begin{equation*}
\sqrt{\sum_{j\in\mathcal{I}_r}(\hat\lambda_j-\mu_r)^2}\leq Cxm_r\mu_r
\end{equation*}
and
\begin{equation*}
\|\hat{Q}_r-Q_r\|_2\leq Cx\sqrt{\sum_{s\neq r}\frac{m_r\mu_rm_s\mu_s}{(\mu_r-\mu_s)^2}}.
\end{equation*}
\end{corollary}

Finally, note that our main results in Theorems \ref{ThmExpSimpleEV}-\ref{ThmProjExp} give linear expansions, which are sufficient for our probabilistic applications. Higher order expansions can be derived by similar, but more tedious considerations. Moreover, note that in \cite{JW20}, Corollary \ref{CorConcentrationMulti} has been extended to eigenspaces (e.g.~principal subspaces) using an extended notion of relative rank. Yet, \cite{JW20} does not establish the more involved relative linear expansions, which are the key to our (central) limit theorems and phase transitions presented in Section \ref{SecApp}.

\section{Applications to covariance operators: invariance principle \\ and phase transition}\label{SecApp}

In this section, we apply our relative perturbation bounds to the eigenstructure of empirical covariance operators. For simplicity, we only consider i.i.d. sequences. 

\begin{setting}[i.i.d. sequence]\label{iidSetting}
Let $X$ be a random variable taking values in $\mathcal{H}$. Suppose that $X$ is centered and strongly square-integrable, meaning that $\E X =0$ and $\E\|X\|^2<\infty$. Let $\Sigma =\E X\otimes X$ be the covariance operator of $X$, which is a self-adjoint and positive trace class operator on $\mathcal{H}$ (see e.g. \cite[Theorem 7.2.5]{MR3379106}). For $j\geq 1$, let $\eta_j=\lambda_j^{-1/2}\langle u_j, X\rangle$ be the Karhunen-Lo\`{e}ve coefficients of $X$. Suppose that for some $p\geq 4$ and a constant $C_\eta>0$,
\begin{equation}\label{EqMomentAssIID}
\sup_{j\geq 1}\E|\eta_j|^p\leq C_\eta.
\end{equation}
Let $X_1,\dots,X_n$ be $n$ independent copies of $X$ and let
\begin{equation*}
\hat{\Sigma}=\hat{\Sigma}_n = \frac{1}{n}\sum_{i=1}^nX_i\otimes X_i
\end{equation*}
be the sample covariance operator, which is self-adjoint, positive and of finite rank.
\end{setting}

Observe that Setting \ref{iidSetting} is very general. Our only condition regarding the distribution of $X$ is the moment condition \eqref{EqMomentAssIID}, no further assumptions are necessary. Therefore, up to this condition, all of our results are only determined by the relative rank $\rr_j(\Sigma)$ and are thus \textit{invariant} with respect to the distribution of $X$ within this framework. This applies in particular to the limit theorems and concentration inequalities discussed below.

In order to apply our perturbation results to Setting \ref{iidSetting}, the main remaining point is to control the probability of the `good event', defined by the boundedness assumptions in our main results. Recall that, for $x > 0$, this either means
\begin{align}\label{Def:event:simple}
|\bar\eta_{kl}|\leq x \quad \text{for all $k,l\geq 1$,}
\end{align}
or, for $r_0\geq 1$, that
\begin{align}\label{Def:event:general}
\frac{\|Q_sEQ_t\|_2}{\sqrt{m_s\mu_sm_t\mu_t}},
\frac{\|Q_sEQ_{\geq r_0}\|_2}{\sqrt{m_s\mu_s\operatorname{tr}_{\geq r_0}(\Sigma)}},
\frac{\|Q_{\geq r_0}EQ_{\geq r_0}\|_2}{\operatorname{tr}_{\geq r_0}(\Sigma)}\leq x\quad \text{for all $s,t< r_0$.}
\end{align}
The relative coefficients satisfy
\begin{equation}\label{EqKLCoeff}
\bar\eta_{kl}=\frac{\langle u_k, Eu_l\rangle}{\sqrt{\lambda_k\lambda_l}}=\frac{1}{n}\sum_{i=1}^n\frac{\langle u_k, X_i\rangle}{\sqrt{\lambda_k}}\frac{\langle u_l, X_i\rangle}{\sqrt{\lambda_l}}-\delta_{k,l}, \qquad k,l\geq 1,
\end{equation}
and are thus sums of (independent) and centered random variables, where the $i$-th summand involves the product of two Karhunen-Lo\`{e}ve coefficients. To control these coefficients, many powerful concentration inequalities have been developed in the literature. This can be easily done in Setting \ref{iidSetting}, but one may also consider several other models (e.g.~weakly dependent sequences). We refer to Propositions \ref{prop_control_events1} and \ref{prop_control_events2} for a demonstration of such type of results, and to Remark \ref{rem:weakdependence} on how to go beyond the i.i.d.~case.

Finally, we also discuss a second, specific distribution contained in Setting \ref{iidSetting} to explicitly describe the phase transition around the critical barrier provided by the relative rank.

\begin{setting}[Inconsistency model]\label{setting:inconsistency:model}
Let $\Sigma$ be a self-adjoint, positive trace class operator on $\mathcal{H}$ with spectral representation $\Sigma =\sum_{j\ge 1}\lambda_j (u_j\otimes u_j)$. Let $r\geq 1$ and $F=\sum_{j\leq r}\sqrt{\lambda_j}u_j$. Let $\epsilon$ be a Gaussian random variable in $\mathcal{H}$ with expectation $0$ and covariance operator $\Sigma-1/(2r)(F\otimes F)$, and let $f$ be a real random variable defined by $\P(f=0)=1-1/(2r^2)$ and $\P(f=\pm\sqrt{r})=1/(4r^2)$, independent of $\epsilon$. Finally, let $X=f\cdot F+\epsilon$ be the convolution with covariance operator $\Sigma$, and let $X_1,\dots,X_n$ be $n$ independent copies of $X$.
\end{setting}

Constructions of these type are often referred to as factor models in the literature, and typically imply a specific covariance structure. Our construction is different, since we essentially allow for any covariance operator $\Sigma$.

The characteristic feature of such a construction is that despite the uncorrelatedness of the random variables $\eta_k$, that is, $\E \eta_k \eta_l  = 0$ for $k \neq l$, they are highly dependent. This dependence, however, is only manifested in the higher order cumulants.  Although we only consider Setting \ref{setting:inconsistency:model} to discuss phase transitions, our method of proof shows that this phenomenon can be applied to other probability measures exhibiting such a strong higher order dependence. The crucial point here is that such dependencies imply that with high probability, one may extract a (deterministic) rank-one perturbation, which in turn governs all phase transitions. We refer to Section \ref{section:proofs:inconsistency} for more details.

\subsection{Limit theorems}\label{sec:CLT}

This section is dedicated to law of large numbers and central limit theorems. We consider a triangular array $X_1^{(n)},\ldots, X_n^{(n)}$ of independent copies of a random variable $X^{(n)}$ in $\mathcal{H}^{(n)}$ with covariance operator $\Sigma^{(n)}$, satisfying the assumptions in Setting \ref{iidSetting} for all $n\geq 1$.
We use the notation of Section \ref{SecNotation} with an additional superscript $^{(n)}$.
We establish two types of results. First, in Setting \ref{iidSetting}, empirical eigenvalues and eigenvectors are consistent and asymptotically normal as long as the relative rank condition \eqref{intro:clt:condition} holds. Second, if \eqref{intro:clt:condition} fails, then eigenvector inconsistency and eigenvalue (upward) bias hold in Setting \ref{setting:inconsistency:model}. For ease of exposition, we focus on simple eigenvalues.

\begin{thm}\label{CorRootNCons}
In the above triangular array, suppose that $\lambda_j^{(n)}$ is simple for all $n\geq 1$ and that \eqref{EqMomentAssIID} is satisfied with $p\geq 4$ and  $C_\eta$ independent of $n$. Suppose that
\begin{equation}\label{EqCCondAsymp}
 \frac{1}{\sqrt{n}}\rr_j(\Sigma^{(n)})\rightarrow 0\quad\text{as}\quad n\rightarrow \infty.
\end{equation}
Then the sequence $(\sqrt{n}(\hat\lambda_j^{(n)}-\lambda_j^{(n)})/\lambda_j^{(n)})$ is tight, i.e.
\begin{equation}\label{EqRootNCons}
\lim_{R\rightarrow\infty}\sup_{n\geq 1}\P(\sqrt{n}|\hat\lambda_j^{(n)}-\lambda_j^{(n)}|/\lambda_j^{(n)}>R)=0.
\end{equation}
In particular, we have the weighted law of large numbers
\begin{align}\label{EqGapWeightedCons}
    (\hat\lambda_j^{(n)}-\lambda_j^{(n)})/g_j^{(n)}\xrightarrow{\mathbb{P}} 0.
\end{align}
\end{thm}

Roughly speaking, the last assertions in \eqref{EqGapWeightedCons} says that given \eqref{EqCCondAsymp}, it is possible to assign empirical and population eigenvalues correctly, at least asymptotically. The following result shows that in Setting \ref{setting:inconsistency:model} for $j=1$ (see Theorem \ref{thm:inconsistency:j>=2:projector} below for extensions to $j>1$), both properties are indeed equivalent, revealing an interesting phase transition.

\begin{thm}\label{thm:inconsistency:j=1:CLT}
For a sequence of covariance operators $\Sigma^{(n)}$ on Hilbert spaces $\mathcal{H}^{(n)}$ of dimension $d_n=\dim \mathcal{H}^{(n)}<\infty$ satisfying $d_n/\sqrt{n}\rightarrow 0$ as $n\rightarrow\infty$, consider the sequence of models given in Setting \ref{setting:inconsistency:model} with $r_n=d_n$. Then the following three assertions are equivalent:
\begin{itemize}
\item[(a)] $
\frac{\rr_1(\Sigma^{(n)})}{\sqrt{n}}\rightarrow 0\quad\text{as}\quad n\rightarrow \infty.
$
\item[(b)] The sequence $(\sqrt{n}(\hat\lambda_1^{(n)}-\lambda_1^{(n)})/\lambda_1^{(n)})$ is tight and $\frac{1}{\sqrt{n}}\frac{\lambda_1^{(n)}}{\lambda_1^{(n)}-\lambda_2^{(n)}}\rightarrow 0$  as $n\rightarrow \infty$.
\item[(c)] $\frac{\hat\lambda_1^{(n)}-\lambda_1^{(n)}}{\lambda_1^{(n)}-\lambda_2^{(n)}}\xrightarrow{\mathbb{P}} 0$.
\end{itemize}
\end{thm}

If the relative rank condition \eqref{EqCCondAsymp} fails, then (c) of Theorem \ref{thm:inconsistency:j=1:CLT} says that the leading empirical eigenvalue is not (relatively) consistent anymore. This, of course, also affects the behaviour of the empirical eigenvectors and spectral projectors. More precisely, the following two results provide the analogs of Theorems \ref{CorRootNCons} and \ref{thm:inconsistency:j=1:CLT} in the case of spectral projectors.

\begin{thm}\label{ThmTightProjector}
In the above triangular array, assume that $\lambda_j^{(n)}$ is simple for all $n\geq 1$ and that \eqref{EqMomentAssIID} is satisfied with $p\geq 4$ and  $C_\eta$ independent of $n$. Suppose that \eqref{EqCCondAsymp} holds. Then
\begin{align}\label{EqSpectralProjCons}
    \|\hat P_j^{(n)}-P_j^{(n)}\|_2\xrightarrow{\mathbb{P}}0.
\end{align}
Moreover, if additionally $\lambda_{j_0}^{(n)} \leq \lambda_j^{(n)}/2$ for all $n\geq 1$ with $j_0 > j$ independent of $n$, then the sequence
\begin{align*}
    \sqrt{n} \|\hat P_j^{(n)}-P_j^{(n)}\|_2\Big/\sqrt{\sum_{k\neq j}\frac{\lambda_j^{(n)}\lambda^{(n)}_k}{(\lambda^{(n)}_j-\lambda^{(n)}_k)^2}},\quad n\geq 1,
\end{align*}
is tight.
\end{thm}

\begin{remark}
Observe that condition $\lambda_{j_0}^{(n)} \leq \lambda_j^{(n)}/2$ for all $n\geq 1$ with $j_0 > j$ independent of $n$, may be weakened at the cost of higher moment conditions and by strengthening \eqref{EqCCondAsymp}. See, for instance, Corollary \ref{CorIntro}. This also applies to Theorem \ref{Cor:clt:eigenvalues:anderson} below. 
\end{remark}

\begin{thm}\label{thm:inconsistency:j=1:projectors}
For a sequence of covariance operators $\Sigma^{(n)}$ on Hilbert spaces $\mathcal{H}^{(n)}$ of dimension $d_n=\dim \mathcal{H}^{(n)}<\infty$ satisfying $d_n/\sqrt{n}\rightarrow 0$ as $n\rightarrow\infty$, consider the sequence of models given in Setting \ref{setting:inconsistency:model} with $r_n=d_n$. Suppose that $n^{-1/2}\lambda_1^{(n)}/(\lambda_1^{(n)}-\lambda_2^{(n)})\rightarrow 0$ as $n\rightarrow \infty$. Then the following statements are equivalent:
\begin{itemize}
\item[(a)] $
\frac{\rr_1(\Sigma^{(n)})}{\sqrt{n}}\rightarrow 0\quad\text{as}\quad n\rightarrow \infty.
$
\item[(b)] $\|\hat P_1^{(n)}-P_1^{(n)}\|_2\xrightarrow{\mathbb{P}}0$.
\end{itemize}
\end{thm}

It remains to address the issue of phase transitions for $j > 1$. Due to more eigenvalue interactions, this turns out to be a more difficult problem to solve. In contrast to the case $j=1$, there is not a `single upward force', pushing $\hat\lambda_1^{(n)}$ outside the spectrum. Still, under slightly stronger assumptions, we can provide the following result in Setting \ref{setting:inconsistency:model}.

\begin{thm}\label{thm:inconsistency:j>=2:eigenvalue}
For a sequence of covariance operators $\Sigma^{(n)}$ on Hilbert spaces $\mathcal{H}^{(n)}$, consider the sequence of models given in Setting \ref{setting:inconsistency:model} with $r_n\rightarrow\infty$ such that  $r_n/\sqrt{n}\rightarrow 0$. Let $j \geq 2$ and suppose that
\begin{align}
    &\liminf_{n\rightarrow\infty}\frac{1}{\sqrt{n}}\sum_{k= j}^{r_n}\frac{\lambda_k^{(n)}}{\lambda_{j-1}^{(n)}-\lambda_k^{(n)}}>0,\nonumber\\
    &\lim_{n\rightarrow\infty}\frac{1}{\sqrt{n}}\frac{\lambda_{j-1}^{(n)}}{\lambda_{j-1}^{(n)}-\lambda_j^{(n)}}=0,\qquad\lim_{n\rightarrow\infty}\frac{1}{\sqrt{n}}\sum_{k\leq j-2}\frac{\lambda_{k}^{(n)}}{\lambda_k^{(n)}-\lambda_{j-1}^{(n)}}=0.\label{eq:thm:inconsistency:j>=2:Ev}
\end{align}
Then we have
\begin{align}\label{eq:thm:inconsistency:j>=2:Ev:statement}
   \lim_{\delta\rightarrow 0} \liminf_{n\rightarrow\infty}\P(|\hat\lambda_j^{(n)}-\lambda_{j}^{(n)}|/g_{j}^{(n)}\geq 1-\delta)>0.
\end{align}
\end{thm}

The (one-sided) relative rank condition in \eqref{eq:thm:inconsistency:j>=2:Ev} ensures that the $j$-th empirical eigenvalue is pushed upwards towards the $(j-1)$-th population eigenvalue, and hence the relative consistency from \eqref{EqGapWeightedCons} does not hold anymore. At the same time, the $j$-th empirical spectral projector contains more information on the $(j-1)$-th population spectral projector, and the following result shows that empirical eigenvectors become asymptotically orthogonal to their population analogs. More background can be found in Section \ref{section:proofs:inconsistency}.

\begin{thm}\label{thm:inconsistency:j>=2:projector}
For a sequence of covariance operators $\Sigma^{(n)}$ on Hilbert spaces $\mathcal{H}^{(n)}$, consider the sequence of models defined in Setting \ref{setting:inconsistency:model} with $r_n\rightarrow\infty$ such that $r_n/\sqrt{n}\rightarrow 0$. For $j\geq 3$ suppose that
\begin{align}
    &\liminf_{n\rightarrow\infty}\frac{1}{\sqrt{n}}\sum_{k=1,k\neq j^*}^{r_n}\frac{\lambda_k^{(n)}}{\lambda_{j^*}^{(n)}-\lambda_k^{(n)}}>0\quad\text{for }j^*\in\{j-1,j-2\},\nonumber\\
    &
    \lim_{n\rightarrow\infty}\frac{1}{\sqrt{n}}\frac{\lambda_{j-1}^{(n)}}{g_{j-1}^{(n)}}=0,\qquad \lim_{n\rightarrow\infty}\frac{1}{n}\frac{\lambda_{j-1}^{(n)}}{g_{j-1}^{(n)}}\rr_{j-1}(\Sigma^{(n)})=0.\label{eq:thm:inconsistency:j>=2:Proj}
\end{align}
Then we have
\begin{align*}
   \lim_{\delta\rightarrow 0} \liminf_{n\rightarrow\infty}\P(\|\hat P_j^{(n)}-P_j^{(n)}\|_2^2\geq 2-\delta)>0.
\end{align*}
\end{thm}

\begin{remark}
In Theorems~\ref{thm:inconsistency:j=1:CLT} and \ref{thm:inconsistency:j=1:projectors},
Conditions $r_n=d_n<\infty$ and $d_n/\sqrt{n}\rightarrow 0$ can be replaced by $d_n=\infty$ and $r_n/\sqrt{n}\rightarrow 0$, provided that additionally
\begin{align*}
    \frac{1}{\sqrt{n}}\sum_{k>r_n}\frac{\lambda_k^{(n)}}{\lambda_1^{(n)}-\lambda_k^{(n)}}\rightarrow 0\quad\text{as}\quad n\rightarrow \infty.
\end{align*}
In this case, it is possible to characterise the behaviour of $n^{-1/2}\rr_1(\Sigma^{(n)})$ by restricting the sum to indices $k\leq r_n$ and the general inconsistency results from Section \ref{sec:General:inconsistency:results} are still applicable.

In contrast, Theorems \ref{thm:inconsistency:j>=2:eigenvalue} and \ref{thm:inconsistency:j>=2:projector} do not provide an equivalence. In this case, it is enough to formulate the relative rank conditions only for eigenvalues with indices $k\leq r_n$, allowing us to consider also Hilbert spaces $\mathcal{H}^{(n)}$ having infinite dimension.
\end{remark}

Finally, we state asymptotic normality. By Prohorov's theorem, any subsequence of $(\sqrt{n}(\hat\lambda_1^{(n)}-\lambda_1^{(n)})/\lambda_1^{(n)})$ has a subsubsequence converging in distribution. Slightly more can be said by applying e.g.~Theorem \ref{ThmExpSimpleEV}. In fact, we have the following generalisation of Anderson's central limit theorem (see \cite{And}):

\begin{thm}\label{Cor:clt:eigenvalues:anderson}
In the above triangular array, suppose that $\lambda_j^{(n)}$ is simple for all $n\geq 1$ and that \eqref{EqMomentAssIID} is satisfied with $p>4$ and $C_\eta$ independent of $n$. Moreover, suppose that $\lambda_{j_0}^{(n)} \leq \lambda_j^{(n)}/2$ for all $n\geq 1$ with $j_0 > j$ independent of $n$. Suppose that \eqref{EqCCondAsymp} holds. Then we have
\[
\sqrt{\frac{n}{\operatorname{Var}((\eta_j^{(n)})^2)}}\frac{\hat\lambda_j^{(n)}-\lambda_j^{(n)}}{\lambda_j^{(n)}}\xrightarrow{d}\mathcal{N}(0,1).
\]
Moreover, for $k\neq j$, with $\lambda_k^{(n)}>\lambda_{j_0}^{(n)}$ for all $n\geq 1$, we have
\begin{align*}
\frac{\lambda_j^{(n)} - \lambda_k^{(n)}}{\sqrt{\lambda_j^{(n)} \lambda_k^{(n)}}}\frac{\sqrt{n} \langle \hat{u}_j^{(n)},u_k^{(n)} \rangle}{\sqrt{\operatorname{Var}(\eta_{k}^{(n)} \eta_{j}^{(n)})}} \xrightarrow{d} \mathcal{N}\big(0,1\big),
\end{align*}
provided that the sign of $\hat u_j^{(n)}$ is chosen such that $\langle \hat u_j^{(n)},u_j^{(n)}\rangle>0$.
\end{thm}

\begin{remark}\label{RemFactorLindebergFeller}
Assumption $p > 4$ can be replaced by the weaker assumption $p = 4$, together with a Lindeberg-Feller condition on $(\eta_j^{(n)})^2$ and $\eta_{k}^{(n)} \eta_{j}^{(n)}$, respectively. 
\end{remark}

\begin{remark}
The case of $\langle \hat{u}_j^{(n)},u_j^{(n)} \rangle$ can be treated using $2\langle \hat{u}_j^{(n)},u_j^{(n)} \rangle = 2-\|\hat{u}_j^{(n)} - u_j^{(n)}\|^2$ in combination with Theorem \ref{ThmExpEVe}, we omit the details.
\end{remark}

\begin{remark}
The central limit theorems in Theorem \ref{Cor:clt:eigenvalues:anderson} are based on linear perturbation expansions. In principle, higher order expansion can lead to weaker assumptions compared to \eqref{EqCCondAsymp}. In general, though, it appears to be a very difficult question to determine the behaviour in case of $n^{-1/2} \rr_j(\Sigma^{(n)} ) \to \infty$, even if $n^{-1} \rr_j(\Sigma^{(n)}) \to 0$, as this strongly depends on the underlying probability distributions. A general invariance result would certainly be of high interest, but would need to include higher order cumulants, possibly of infinite order.
\end{remark}

\subsection{High probability bounds}\label{sec_concentration_ineq}

Concentration results for norms
of (covariance) operators are a well studied problem in the literature. Despite their importance, corresponding results for eigenvalues (resp.~eigenvectors) are less known, a reason certainly being the fact that an application of Weyl's inequality immediately directs the problem to bounding $\|E\|_{\infty}$, and corresponding results are readily available in the literature (cf.~\cite{MR3407216,MR2946459}). Below, we leave this common path and establish high-probability bounds based on our results obtained in Section \ref{SecPerturb}. It turns out that, in general, the relative rank $\rr_j(\Sigma)$ yields a sharp transition when concentration with high probability is possible for empirical eigenvalues, and when this is not the case. As in Section \ref{sec:CLT}, we focus on simple eigenvalues and corresponding eigenvectors.

\begin{corollary}\label{CorIntro} In Setting \ref{iidSetting}, suppose that \eqref{EqMomentAssIID} is satisfied with $p>4$ and constant $C_\eta$. Then there are constants $c_1,C_1>0$, depending only on $C_\eta$ and $p$, such that, with probability at least $1-d^2(\log n)^{-p/4} n^{1-p/4}$, the inequalities
\[
|\hat\lambda_j-\lambda_j|/\lambda_j\leq C_1\sqrt{\frac{\log n}{n}},\qquad\|\hat P_j-P_j\|_2\Big/\sqrt{\sum_{k\neq j}\frac{\lambda_j\lambda_k}{(\lambda_j-\lambda_k)^2}}\leq C_1\sqrt{\frac{\log n}{n}}
\]
hold uniformly for all $j\geq 1$ satisfying
\begin{equation}\label{EqRelEffRank}
\sqrt{\frac{\log n}{n}}\rr_j(\Sigma)\leq c_1.
\end{equation}
Moreover, for a single $j\geq 1$ and $j_0>j$ satisfying $\lambda_{j_0} \leq \lambda_j/2$ both inequalities remain valid with probability at least $1-j_0^2(\log n)^{-p/4}n^{1-p/4}$.
\end{corollary}

Let us briefly discuss the above result. A striking aspect is its relative nature, both eigenvalues and spectral projectors scale with the correct (first order) variance, which is important for applications. The underlying assumptions are simple and relatively weak. Moreover, the following theorem shows that Condition \eqref{EqRelEffRank} is both sufficient and necessary when considering Setting \ref{setting:inconsistency:model}. 

\begin{thm}\label{thm:inconsistency:j=1:concentration}
For a sequence of covariance operators $\Sigma^{(n)}$ on Hilbert spaces $\mathcal{H}^{(n)}$ of dimension $d_n=\dim \mathcal{H}^{(n)}<\infty$ satisfying $d_n \leq n^{1/2 - \delta}$, some $\delta > 0$, consider the sequence of models given in Setting \ref{setting:inconsistency:model} with $r_n=d_n$. Suppose that $\limsup_{n\rightarrow\infty}((\log n)/n)^{1/2}\lambda_1^{(n)}/(\lambda_1^{(n)}-\lambda_2^{(n)})<\infty$. Then the following statements are equivalent:

\begin{itemize}
    \item[(a)] $\limsup_{n\rightarrow\infty}\sqrt{\frac{\log n}{n}}\rr_1(\Sigma^{(n)})=\infty$.
    \item[(b)] $-\log\big(\P\big(\sqrt{n}(\hat\lambda_1^{(n)}-\lambda_1^{(n)})/\lambda_1^{(n)}\geq  \sqrt{\log n}\big)\big)=o(\log n).$
\end{itemize}
\end{thm}

\subsection{Examples}\label{sec:examples}


\begin{example}[Spiked covariance and factor models]
Among different structures of covariances, the spiked covariance model is of great interest. The signature feature is that several eigenvalues are larger than the remaining, and typically one is interested in recovering these leading eigenvalues and their associated eigenvectors. The spiked part is of importance, as we are usually interested in the directions that explain the most variations of the data. The model has been extensively studied in the literature, see for instance ~\cite{baik:aop:2005,Benaych-Georges:advances:2011,Cai2015_ptrf,debashis_2007,wang2017} and the many references therein.

One way to define the model is as follows, where we suppose for simplicity that $\mathcal{H}=\R^d$. Let $f_1,\dots,f_d$ be orthogonal vectors and $\Gamma$ be a covariance matrix such that
\begin{align}\label{ExFactorAcondi}
C_{\Gamma}^{-1} \leq \lambda_d(\Gamma) \leq \lambda_1(\Gamma) \leq C_{\Gamma}.
\end{align}
For a sequence of weights $\omega_1,\dots,\omega_d$, consider the spiked covariance model
\begin{align}
\Sigma =\sum_{k =1}^d \omega_k^2 f_kf_k^{\top} + \Gamma,
\end{align}
We now equip $\Sigma$ with a probabilistic structure by constructing a factor model generating $\Sigma$. Given a filtration, let $F_1,\dots,F_d$ be a martingale difference sequence with $\E F_{k}^2 = 1$, which serve as the factor loadings. Similarly, let $Y =(Y_{1},\dots,Y_d)^{\top}$ be a random vector, where $Y_1,\dots,Y_d$ form a martingale difference sequence with $\E Y_{k}^2 = 1$. In both cases, the underlying filtration is of no particular relevance. In addition, we assume that $F$ and $Y$ are mutually uncorrelated, that is, all cross correlations are zero. The idiosyncratic error $\epsilon$ and the canonical factor model are then defined as
\begin{align}\label{FactorX}
X = \sum_{k =1}^d \omega_k F_{k} f_k + \epsilon, \quad \epsilon  = \Gamma^{1/2}Y.
\end{align}
Obviously, $X$ has covariance matrix $\Sigma$. In order to apply our results, we need to verify the assumptions made in Settings \ref{iidSetting} regarding the coefficients $\eta_j$. The following proposition provides the connection between the underlying moments of $\eta_j$ and $F_1,\dots,F_d$, $Y$.

\begin{proposition}\label{prop:factor:model}
For $p \geq 2$, suppose that
\begin{align}\label{prop:factor:model:condi}
\E|F_{k}|^p \leq C_F, \quad \E|Y_{k}|^p \leq C_Y
\end{align}
for all $k=1,\dots,d$. Then the conditions above imply that
\begin{equation*}
\max_{j \geq 1} \E|\eta_j|^p \leq C_{\eta},
\end{equation*}
where $C_{\eta}$ only depends on $C_{F}$, $C_{Y}$ and $C_{\Gamma}$. In particular, if $X_1,\dots,X_n$ are independent copies of $X$, then Setting \ref{iidSetting} applies.
\end{proposition}

As an immediate consequence, all the results of Sections \ref{sec:CLT} and \ref{sec_concentration_ineq} apply. Observe that due to the Gaussianity in Setting \ref{setting:inconsistency:model}, we can always find a martingale structure with corresponding $(F_{k})$ by independence. Consequently, also all inconsistency results of Sections \ref{sec:CLT} and \ref{sec_concentration_ineq} apply.
\end{example}


\begin{example}[Functional data: trace class operators]
In the context of high-dimensional data, functional principal component analysis (FPCA) is becoming more and more important. The characteristic feature of FPCA is that the underlying Hilbert space $\mathcal{H}$ has (possibly) infinite dimension, while the covariance operator of the corresponding data is assumed to be of trace class. A comprehensive overview and some leading examples can be found in ~\cite{jolliffe_book_2002,ramsay_silverman_2005,hoermann_2010}. Seeking the optimal subspace, prediction or approximation, statisticians are therefore facing the problem of model selection with respect to some risk function (cf.~\cite{blanchard_2007,cai_yuan_2012,MR2332269,mas_hilgert_2013,meister_2011}), as the actual decay rate of $\lambda_j$ is usually unknown. For optimal results, precise deviation bounds for $\hat{\lambda}_j$ and $\hat{P}_j$ are essential for these kind of problems. In the literature, rather strong, explicit assumptions like polynomial or exponential decay of eigenvalues are typically imposed in this context, see for instance~\cite{MR2332269,hall_hosseini_2009} and the references above. Using our results from Section \ref{SecApp}, relatively general results can be obtained. To this end, let us assume that
\begin{align}\label{condi_convex}
\text{there is a convex function $\lambda:\R_{\geq 0}\rightarrow\R_{\geq 0}$, such that $ \lambda(j)=\lambda_j $ and $\lim_{j\rightarrow\infty}\lambda(j) = 0$.}
\end{align}
We then have the bounds (cf.~\cite{cardot_mas_sarda_2007})
\begin{align}\label{eq_convex_condi}
\sum_{k\neq j} \frac{\lambda_k }{|\lambda_j - \lambda_k|} \leq C j \log j \quad \text{and} \quad \sum_{k\neq j} \frac{\lambda_k \lambda_j }{(\lambda_j - \lambda_k)^2} \leq C j^2,
\end{align}
where $C$ is a constant which only depends on $\operatorname{tr}(\Sigma)$.
Condition \eqref{condi_convex} is quite general, and is valid in particular for polynomially and exponentially decaying eigenvalues. Using \eqref{condi_convex}, it is very easy to validate the relative error bounds and conditions of Section \ref{SecApp}. For example, we have the following result in expectation
\begin{corollary}\label{CorrFunctionalDataProj}
Suppose we are in Setting \ref{iidSetting}
with $p \geq 16$. If \eqref{condi_convex} holds,
then
\begin{align*}
\E\|\hat{P}_j - P_j\|_{\infty}^2 \leq  \E\|\hat{P}_j - P_j\|_2^2 \leq Cj^2/n,  \quad 1 \leq j \leq C \sqrt{n} (\log n)^{-5/2}.
\end{align*}
\end{corollary}
Corollary \ref{CorrFunctionalDataProj} is (up to log terms) optimal in the case where $\lambda_j =C j^{-\alpha-1}$, $\alpha > 0$. For such a decay, given the  Setting \ref{iidSetting} with $\sup_{j\geq 1}\E|\eta_j|^{2p} \leq p!C^{p}$ for all $p\geq 1$, \cite{mas_complex_2014} shows that for any $j \geq 1$
\begin{align}\label{eq_ryumgaart_2}
\E\bigl\|\hat{P}_j - P_j\|_{\infty}^2 \geq c (j^2/n)\wedge 1 .
\end{align}
Hence we obtain the optimal bound for almost the whole range (up to the factor $(\log n)^{-5/2}$), where the trivial bound $2$ does not apply. Moreover, we only require the mild conditions of Setting \ref{iidSetting}. Given the stronger assumption that all moments of $\eta_j$ exist for all $j \geq 1$,~\cite{M16} also established a matching upper bound for the region $j \leq n^{1/2-b}$, $b > 0$. It is interesting to note that the stochastic behaviour of the $\eta_j$ - in terms of their dependence structure - is irrelevant for the optimal algebraic structure conditions in this case.

\end{example}

\section{Proofs of the perturbation bounds}\label{SecProofs}
The purpose of this section is to prove the results from Section \ref{SecPerturb}.

\subsection{Relative Weyl and Davis-Kahan inequalities}\label{SecRelWeylDK}
We derive some auxiliary perturbation bounds for eigenvalues and spectral projectors. We start with the following result established in \cite{JW20,RW17}.
\begin{lemma}[Proposition 1 of \cite{JW20}]\label{LemEvRC}
For all $j\geq 1$ and $y>0$, we have the implications
\begin{align*}
\|T_{\geq j}(y)ET_{\geq j}(y)\|_\infty&\leq 1\Rightarrow
\hat \lambda_j-\lambda_j\leq y,\\
\|T_{\leq j}(y)ET_{\leq j}(y)\|_\infty&\leq 1\Rightarrow
 \hat\lambda_j-\lambda_j\geq -y,
\end{align*}
with operators $T_{\geq j}(y)$ and $T_{\leq j}(y)$ defined by
\begin{align*}
T_{\geq j}(y)=\sum_{k\geq  j}(\lambda_j+y-\lambda_k)^{-1/2}P_k,\qquad T_{\leq j}(y)=\sum_{k\leq  j}(\lambda_k+y-\lambda_j)^{-1/2}P_k.
\end{align*}
\end{lemma}
The following corollary of Lemma \ref{LemEvRC} follows from bounding the operator norm by the Hilbert-Schmidt norm. It serves as a first eigenvalue separation step in the proofs of our linear expansions.

\begin{corollary}\label{PropEvRC}
For all $j\geq 1$ and $y>0$, we have the implications
\begin{equation}\label{EqEvRD}
\sum_{k\geq j}\sum_{l\geq j}\frac{\lambda_k}{\lambda_j+y-\lambda_k}\frac{\lambda_l}{\lambda_j+y-\lambda_l}\bar{\eta}_{kl}^2\leq 1\Rightarrow
\hat \lambda_j-\lambda_j\leq y
\end{equation}
and
\begin{equation}\label{EqEvLD}
\sum_{k\leq j}\sum_{l\leq j}\frac{\lambda_k}{\lambda_k+y-\lambda_j}\frac{\lambda_l}{\lambda_l+y-\lambda_j}\bar{\eta}_{kl}^2\leq 1\Rightarrow
 \hat\lambda_j-\lambda_j\geq -y.
\end{equation}
\end{corollary}
For simple eigenvalues, Lemma \ref{LemEvRC} implies a relative Weyl bound. For this purpose, let
\begin{align}\label{Def:delta:j}
   \delta_j=\| T_j(\hat\Sigma-\Sigma) T_j\Vert_\infty,\qquad T_j=| R_j|^{1/2}+ g_j^{-1/2}P_j,
\end{align}
where $| R_j|^{1/2}=\sum_{k\neq j}|\lambda_k-\lambda_j|^{-1/2} P_k$ is the square-root of the reduced resolvent of $ \Sigma$ at $\lambda_j$.

\begin{corollary}\label{lem:relative:Weyl}
If $\delta_j\leq 1$, then we have $|\hat\lambda_j-\lambda_j|\leq g_j\delta_j$.
\end{corollary}

\begin{proof}
For $y=\delta_jg_j$, we have
\begin{align*}
\|T_{\geq j}(\delta_jg_j)ET_{\geq j}(\delta_jg_j)\|_\infty
&\leq \Vert(|R_j|^{1/2}+(\delta_jg_j)^{-1/2}P_j)E(|R_j|^{1/2}+(\delta_jg_j)^{-1/2}P_j)\Vert_\infty\\
&\leq \delta_j^{-1}\Vert(|R_j|^{1/2}+g_j^{-1/2}P_j)E(|R_j|^{1/2}+g_j^{-1/2}P_j)\Vert_\infty\leq 1,
\end{align*}
as can be seen by simple properties of the operator norm, using that $\sqrt{\lambda_j+\delta_jg_j-\lambda_k}\geq \sqrt{\lambda_j-\lambda_k}$ for every $k>j$. Similarly, we have $\|T_{\leq j}(\delta_jg_j)ET_{\leq j}(\delta_jg_j)\|_\infty\leq 1$. The claim now follows from Lemma \ref{LemEvRC}.
\end{proof}

\begin{lemma}\label{lemma:relative:DK}
We have $\|\hat P_j- P_j\|_2\leq C\delta_j$ for some absolute constant $C>1$.
\end{lemma}

\begin{proof}
We first show that if $\delta_j<1/2$, then
\begin{equation}\label{EqEqTaylorSPp1}
\Vert|R_j|^{-1/2}\hat P_j\Vert_2\leq \frac{\Vert|R_j|^{1/2}E P_j\Vert_2}{1-2\delta_j},
\end{equation}
where $| R_j|^{-1/2}=\sum_{k\neq j}|\lambda_k-\lambda_j|^{1/2} P_k$. Combining Corollary \ref{lem:relative:Weyl} with $(\hat \lambda_j-\lambda_k)\hat P_j P_k=\hat P_jE P_k$, $k\neq j$, leads to (see e.g.~the proof of Lemma 4 in \cite{JW20})
\begin{equation}\label{EqNormalizedBasicIneq}
\Vert|R_j|^{-1/2}\hat P_j\Vert_2\leq \frac{\Vert|R_j|^{1/2}E \hat P_j\Vert_2}{1-\delta_j}.
\end{equation}
Applying the triangle inequality and the identities $I=P_j+(I-P_j)$ and $I-P_j=|R_j|^{1/2}|R_j|^{-1/2}$, we get
\begin{align}
\Vert|R_j|^{1/2}E \hat P_j\Vert_2&\leq \Vert|R_j|^{1/2}E P_j\hat P_j\Vert_2+\Vert|R_j|^{1/2}E (I-P_i)\hat P_j\Vert_2\nonumber\\
&=\Vert|R_j|^{1/2}E P_j\hat P_j\Vert_2+\Vert|R_j|^{1/2}E |R_j|^{1/2}|R_j|^{-1/2}\hat P_j\Vert_2\nonumber\\
&\leq \Vert|R_j|^{1/2}E P_j\hat P_j\Vert_2+\Vert|R_j|^{1/2}E |R_j|^{1/2}\Vert_\infty\Vert|R_j|^{-1/2}\hat P_j\Vert_2\nonumber\\
&\leq \Vert|R_j|^{1/2}E P_j\Vert_2+\delta_j\Vert|R_j|^{-1/2}\hat P_j\Vert_2\label{EqRecIneq}.
\end{align}
Inserting \eqref{EqRecIneq} into \eqref{EqNormalizedBasicIneq}, we obtain \eqref{EqEqTaylorSPp1}. Now, if $\delta_j\geq 1/4$, then the lemma is trivially true since $\|\hat P_j- P_j\|_2$ is always bounded above by $\sqrt{2}$. Thus, assume $\delta_j\leq 1/4$. We have
\begin{align*}
    \|\hat P_j- P_j\|_2=\sqrt{2}\|(I-P_j)\hat P_j\|_2\leq \sqrt{2}g_j^{-1/2} \Vert|R_j|^{-1/2} \hat P_j\Vert_2.
\end{align*}
Inserting \eqref{EqEqTaylorSPp1}, we get
\begin{align*}
    \|\hat P_j- P_j\|_2\leq \sqrt{8}g_j^{-1/2}\Vert|R_j|^{1/2}E P_j\Vert_2\leq \sqrt{8}\delta_j.
\end{align*}
This completes the proof.
\end{proof}

\subsection{Separation of eigenvalues}

\begin{lemma}\label{LemBFE}
Let $j\geq 1$. Suppose that $\lambda_j$ is a simple eigenvalue. Let $x>0$ be such that $|\bar\eta_{kl}|\leq x$ for all $k,l\geq 1$ and Condition \eqref{EqCCond} holds.
Then we have
\begin{equation}\label{EqEVC}
|\hat{\lambda}_j-\lambda_j|\leq  3x\lambda_j/2.
\end{equation}
In particular, the inequality  $|\hat\lambda_j-\lambda_k|\geq |\lambda_j-\lambda_k|/2$ holds for all $k\neq j$.
\end{lemma}
\begin{proof}
For the first claim, it suffices to show that the assumptions in \eqref{EqEvRD} and \eqref{EqEvLD} in Corollary \ref{PropEvRC} are satisfied with $y=3x\lambda_j/2$. Indeed, for this choice, we have
\begin{align*}
&\bigg(\sum_{k\geq j}\sum_{l\geq j}\frac{\lambda_k}{\lambda_j+y-\lambda_k}\frac{\lambda_l}{\lambda_j+y-\lambda_l}\bar{\eta}_{kl}^2\bigg)^{1/2}\\
&\leq x\sum_{k\geq j}\frac{\lambda_k}{\lambda_j+y-\lambda_k}\leq 2/3+x\sum_{k>j}\frac{\lambda_k}{\lambda_j-\lambda_k}\leq 1,
\end{align*}
as can be seen from the fact that $|\bar\eta_{kl}|\leq x$ for all $k,l\geq 1$ and  Condition \eqref{EqCCond}. The proof of the assumption in \eqref{EqEvLD} follows the same line of arguments.

For the second claim, note that $|\hat{\lambda}_j-\lambda_k|=|\lambda_j-\lambda_k+\hat{\lambda}_j-\lambda_j|\geq |\lambda_j-\lambda_k|-|\hat{\lambda}_j-\lambda_j|$. Inserting the inequality
\[
|\lambda_j-\hat{\lambda}_j|\leq 3x\lambda_j/2\leq 3x\rr_j(\Sigma)|\lambda_j-\lambda_k|/2\leq |\lambda_j-\lambda_k|/2,
\]
which follows from the first claim and Condition \eqref{EqCCond}, the second claim follows.
\end{proof}

\subsection{Contraction argument for eigenvectors}

\begin{lemma}\label{LemBasicBBEV} Let $j\geq 1$. Suppose that $\lambda_j$ is a simple eigenvalue. Let $x>0$ be such that $|\bar\eta_{kl}|\leq x$ for all $k,l\geq 1$ and Condition \eqref{EqCCond} holds. Then the inequality
\begin{equation*}
|\langle \hat u_j,u_k\rangle|\leq Cx \frac{\sqrt{\lambda_j\lambda_k}}{|\lambda_j-\lambda_k|}
\end{equation*}
holds for all $k\neq j$, with $C=6$.
\end{lemma}

\begin{proof}
By Parseval's identity and the definition of the coefficients $\bar \eta_{kl}$, we have
\begin{align*}
\langle \hat u_j,Eu_k\rangle&=\sum_{l\geq 1}\langle \hat u_j,u_l\rangle\langle  u_l,Eu_k\rangle\\
&=\langle \hat u_j,u_j\rangle\langle  u_j,Eu_k\rangle+\sum_{l\neq j}\langle \hat u_j,u_l\rangle\langle  u_l,Eu_k\rangle\\
&=\langle\hat u_j,u_j\rangle\sqrt{\lambda_k\lambda_j}\bar\eta_{kj}+\sum_{l\neq j}\langle\hat u_j,u_l\rangle\sqrt{\lambda_k\lambda_l}\bar\eta_{kl}
\end{align*}
for all $k\neq j$ and thus
\begin{align}
\sqrt{\lambda_k}\langle \hat u_j,u_k\rangle&=\frac{\sqrt{\lambda_k}}{\hat \lambda_j-\lambda_k}\langle \hat u_j,Eu_k\rangle\nonumber\\
&=\frac{\lambda_k}{\hat \lambda_j-\lambda_k}\big(\bar\eta_{kj}\sqrt{\lambda_j}\langle\hat u_j,u_j\rangle+\sum_{l\neq j}\bar\eta_{kl}\sqrt{\lambda_l}\langle\hat u_j,u_l\rangle\big)\label{EqEvEq}
\end{align}
(note that by Lemma \ref{LemBFE}, we have $\hat\lambda_j\neq \lambda_k$). Setting
\[
\alpha_{l}=\sqrt{\lambda_l}|\langle\hat u_j,u_l\rangle|\qquad\forall l\neq j,
\]
we get from \eqref{EqEvEq}, the triangle inequality, Lemma \ref{LemBFE}, and the boundedness assumption on the $\bar\eta_{kl}$ that
\begin{align}\label{EqRecEq}
\alpha_k&\leq \frac{\lambda_k}{|\hat \lambda_j-\lambda_k|}\big(|\bar\eta_{kj}|\sqrt{\lambda_j}+\sum_{l\neq j}|\bar\eta_{kl}|\alpha_l\big)\leq 2x\frac{\lambda_k}{|\lambda_j-\lambda_k|}\big(\sqrt{\lambda_j}+\sum_{l\neq j}\alpha_l\big)
\end{align}
for all $k\neq j$. Summing over $k\neq j$ and using \eqref{EqCCond}, we get the contraction inequality
\begin{equation}
3\sum_{k\neq j}\alpha_k\leq 2\sqrt{\lambda_j}+2\sum_{l\neq j}\alpha_l,
\end{equation}
and thus
\begin{equation*}
\sum_{k\neq j}\alpha_k\leq 2\sqrt{\lambda_j}.
\end{equation*}
Plugging this into \eqref{EqRecEq} we have proven the inequality
\begin{equation}
\alpha_k\leq 6x\frac{\lambda_k\sqrt{\lambda_j}}{|\lambda_j-\lambda_k|}
\end{equation}
for all $k\neq j$.  Dividing through by $\sqrt{\lambda}_k$, the claim follows.
\end{proof}

\begin{proposition}\label{PropCEVe}  Let $j\geq 1$. Suppose that $\lambda_j$ is a simple eigenvalue. Let $x>0$ be such that $|\bar\eta_{kl}|\leq x$ for all $k,l\geq 1$ and Condition \eqref{EqCCond} holds. Then the inequality
\begin{equation*}
\|\hat u_j-u_j\|\leq Cx \sqrt{\sum_{k\neq j}\frac{\lambda_j\lambda_k}{(\lambda_j-\lambda_k)^2}}\leq Cx\rr_j(\Sigma)
\end{equation*}
holds for all $k\neq j$. Here, the sign of $\hat u_j$ is chosen such that $\langle \hat u_j,u_j\rangle>0$.
\end{proposition}
\begin{proof}
By Parseval's identity, we have
\begin{equation}\label{EqNormPars}
\|\hat u_j-u_j\|^2=\sum_{k\geq 1}\langle \hat u_j-u_j,u_k\rangle^2=\sum_{k\neq j}\langle \hat u_j,u_k\rangle^2+(1-\langle \hat u_j,u_j\rangle)^2.
\end{equation}
On the other hand, we have $\|\hat u_j-u_j\|^2=2(1-\langle \hat u_j,u_j\rangle)$. Since $\langle \hat u_j,u_j\rangle>0$, we get $\|\hat u_j-u_j\|^2\leq 2$ and thus $(1-\langle \hat u_j,u_j\rangle)^2=\|\hat u_j-u_j\|^4/4\leq \|\hat u_j-u_j\|^2/2$. Inserting this into \eqref{EqNormPars} and using Lemma \ref{LemBasicBBEV}, we get
\[
\|\hat u_j-u_j\|^2\leq 2\sum_{k \neq j}\langle \hat u_j,u_k\rangle^2\leq Cx^2 \sum_{k\neq j}\frac{\lambda_j\lambda_k}{(\lambda_j-\lambda_k)^2},
\]
which
gives the first inequality. The second inequality follows from Condition \eqref{EqCCond}.
\end{proof}

\subsection{Proof of Theorems \ref{ThmExpSimpleEV} and \ref{ThmExpEVe}}

Apart from the improved constant in condition \eqref{EqCCond}, Theorem \ref{ThmExpSimpleEV} can also be deduced from Theorem \ref{ThmExpEvMult}. For the sake of completeness, we give the direct proof.

\begin{proof}[Proof of Theorem \ref{ThmExpSimpleEV}] We have
\begin{align*}
\hat{\lambda}_j-\lambda_j-\lambda_j\bar{\eta}_{jj}&=\langle \hat u_j,\hat\Sigma\hat u_j\rangle-\lambda_j\langle  \hat u_j,\hat u_j\rangle-\langle u_j,Eu_j\rangle\\
&=\langle \hat u_j,E\hat u_j\rangle-\langle  u_j,E u_j\rangle+\langle \hat u_j,(\Sigma-\lambda_jI)\hat u_j\rangle
\end{align*}
and thus
\begin{equation}\label{EqEVDec}
|\hat{\lambda}_j-\lambda_j-\lambda_j\bar{\eta}_{jj}|\leq |\langle \hat u_j,E\hat u_j\rangle-\langle  u_j,E u_j\rangle|+|\langle \hat u_j,(\Sigma-\lambda_jI)\hat u_j\rangle|.
\end{equation}
We begin with the second term on the right-hand side of \eqref{EqEVDec}. By Parseval's identity, we have
\begin{align*}
\langle \hat u_j,(E-\lambda_jI)\hat u_j\rangle=\sum_{k\neq j}\langle \hat u_j,u_k\rangle\langle (\Sigma-\lambda_jI)\hat u_j,u_k\rangle=\sum_{k\neq j}(\lambda_k-\lambda_j)\langle \hat u_j,u_k\rangle^2.
\end{align*}
From this, the triangle inequality, Lemma \ref{LemBasicBBEV}, we conclude that
\begin{align*}
|\langle \hat u_j,(E-\lambda_jI)\hat u_j\rangle|&
\leq\sum_{k\neq j}|\lambda_j-\lambda_k|\langle \hat u_j,u_k\rangle^2\leq Cx^2\sum_{k\neq j}\frac{\lambda_j\lambda_k}{|\lambda_j-\lambda_k|}\leq Cx^2 \lambda_j\rr_j(\Sigma).\nonumber
\end{align*}
Similarly, the first term can be written as
\begin{align*}
&\langle \hat u_j,E\hat u_j\rangle-\langle  u_j,E u_j\rangle\\
&=\sum_{k,l\geq 1}\langle \hat u_j,u_k\rangle\langle \hat u_j,u_l\rangle\langle u_k,E u_l\rangle-\langle  u_j,E u_j\rangle\\
&=\sum_{k,l\neq j}\langle \hat u_j,u_k\rangle\langle \hat u_j,u_l\rangle\langle u_k,E u_l\rangle+2\sum_{k\neq j}\langle \hat u_j,u_k\rangle\langle \hat u_j,u_j\rangle\langle u_k,E u_j\rangle+(\langle \hat u_j,u_j\rangle^2-1)\langle u_j,E u_j\rangle.
\end{align*}
From this and the triangle inequality, we obtain
\begin{align*}
|\langle \hat u_j,E\hat u_j\rangle-\langle  u_j,E u_j\rangle|
&\leq \sum_{k,l\neq j}|\sqrt{\lambda_k}\langle \hat u_j,u_k\rangle\sqrt{\lambda_l}\langle \hat u_j,u_l\rangle\bar\eta_{kl}|\\
&+2\sum_{k\neq j}|\sqrt{\lambda_k}\langle \hat u_j,u_k\rangle\sqrt{\lambda_j}\langle \hat u_j,u_j\rangle\bar\eta_{kj}|+\lambda_j\bar\eta_{jj}\sum_{k\neq j}\langle \hat u_j,u_k\rangle^2.
\end{align*}
Using Lemma \ref{LemBasicBBEV}, the boundedness assumption on the $\bar\eta_{kl}$, and Condition \eqref{EqCCond}, we conclude that
\begin{align*}
&|\langle \hat u_j,E\hat u_j\rangle-\langle  u_j,E u_j\rangle|\\
&\leq Cx^3\lambda_j\sum_{k,l\neq j}\frac{ \lambda_k}{|\lambda_j-\lambda_k|}\frac{ \lambda_l}{|\lambda_j-\lambda_l|}+Cx^2 \lambda_j\sum_{k\neq j}\frac{\lambda_k}{|\lambda_j-\lambda_k|}+Cx^3\lambda_j\sum_{k\neq j}\frac{\lambda_j\lambda_k}{(\lambda_j-\lambda_k)^2}\\
&\leq Cx^3\lambda_j\rr_j^2(\Sigma)+Cx^2\lambda_j\rr_j(\Sigma)+Cx^3\lambda_j\rr_j^2(\Sigma)\leq Cx^2\lambda_j\rr_j(\Sigma).
\end{align*}
This completes the proof.
\end{proof}
\begin{proof}[Proof of Theorem \ref{ThmExpEVe}] For each $k\neq j$, we have
\[
\langle \hat u_j, u_k\rangle=\frac{\langle  u_j, Eu_k\rangle}{\lambda_j-\lambda_k}+\frac{\langle \hat u_j-u_j, Eu_k\rangle}{\lambda_j-\lambda_k}+\frac{\lambda_j-\hat\lambda_j}{\lambda_j-\lambda_k}\langle \hat u_j, u_k\rangle,
\]
as can be seen from inserting the equality $(\hat\lambda_j-\lambda_k)\langle \hat u_j, u_k\rangle=\langle \hat u_j, Eu_k\rangle$.
Thus
\begin{align*}
&\hat u_j-u_j-\sum_{k\neq j}\frac{\langle  u_j, Eu_k\rangle}{\lambda_j-\lambda_k}u_k\\
&=\sum_{k\neq j}\langle \hat u_j, u_k\rangle u_k+(\langle \hat u_j, u_j\rangle-1)u_j-\sum_{k\neq j}\frac{\langle  u_j, Eu_k\rangle}{\lambda_j-\lambda_k}u_k\\
&=\sum_{k\neq j}\frac{\langle \hat u_j-u_j, Eu_k\rangle}{\lambda_j-\lambda_k}u_k+\sum_{k\neq j}\frac{\lambda_j-\hat\lambda_j}{\lambda_j-\lambda_k}\langle \hat u_j, u_k\rangle u_k+(\langle \hat u_j, u_j\rangle-1)u_j.
\end{align*}
From this, the triangle inequality, and Parseval's identity, we get
\begin{align*}&\bigg\|\hat{u}_j-u_j-\sum_{k\neq j}\frac{\langle  u_j, Eu_k\rangle}{\lambda_j-\lambda_k} u_k\bigg\|\\
&\leq\sqrt{\sum_{k\neq j}\frac{\langle \hat u_j-u_j, Eu_k\rangle^2}{(\lambda_j-\lambda_k)^2}}+\sqrt{\sum_{k\neq j}\frac{(\lambda_j-\hat\lambda_j)^2}{(\lambda_j-\lambda_k)^2}\langle \hat u_j, u_k\rangle^2}+1-\langle\hat u_j, u_j\rangle.
\end{align*}
By Proposition \ref{PropCEVe}, and Condition \eqref{EqCCond}, the third term is bounded as follows:
\begin{align*}
1-\langle\hat u_j, u_j\rangle=\|\hat u_j-u_j\|^2/2&\leq Cx^2\sum_{k\neq j}\frac{\lambda_j\lambda_k}{(\lambda_j-\lambda_k)^2}\leq Cx^2\rr_j(\Sigma)\sqrt{\sum_{k\neq j}\frac{\lambda_j\lambda_k}{(\lambda_j-\lambda_k)^2}}.
\end{align*}
By Lemma \ref{LemBFE} and Lemma \ref{LemBasicBBEV}, the second term is bounded as follows:
\begin{align*}
\sqrt{\sum_{k\neq j}\frac{(\hat\lambda_j-\lambda_j)^2}{(\lambda_j-\lambda_k)^2}\langle \hat u_j, u_k\rangle^2}
& \leq Cx^2\sqrt{\sum_{k\neq j}\frac{\lambda_j^2}{(\lambda_j-\lambda_k)^2}\frac{\lambda_j\lambda_k}{(\lambda_j-\lambda_k)^2}}\leq Cx^2\rr_j(\Sigma)\sqrt{\sum_{k\neq j}\frac{\lambda_j\lambda_k}{(\lambda_j-\lambda_k)^2}}.
\end{align*}
It remains to bound the first term. By Parseval's identity, we have
\begin{align*}
\langle \hat u_j-u_j, Eu_k\rangle&=\sum_{l\geq 1}\langle \hat u_j-u_j,u_l\rangle\langle u_l, Eu_k\rangle\\
&=\sqrt{\lambda_j\lambda_k}\langle \hat u_j-u_j,u_j\rangle\bar\eta_{jk} +\sum_{l\neq j}\sqrt{\lambda_k\lambda_l}\langle \hat u_j,u_l\rangle\bar\eta_{kl}
\end{align*}
Applying the triangle inequality, Proposition \ref{PropCEVe}, Lemma \ref{LemBasicBBEV}, and Condition \eqref{EqCCond}, we get
\begin{align*}
|\langle \hat u_j-u_j, Eu_k\rangle|&\leq x\sqrt{\lambda_j\lambda_k}\|\hat u_j-u_j\|^2/2+Cx^2\sqrt{\lambda_j\lambda_k}\sum_{l\neq j }\frac{\lambda_l}{|\lambda_j-\lambda_l|}\\
&\leq Cx^3\sqrt{\lambda_j\lambda_k}\rr_j^2(\Sigma)+x^2\sqrt{\lambda_j\lambda_k}\rr_j(\Sigma)\leq Cx^2\sqrt{\lambda_j\lambda_k}\rr_j(\Sigma)
\end{align*}
for all $k\neq j$. Thus
\begin{equation*}
\sqrt{\sum_{k\neq j}\frac{\langle \hat u_j-u_j, Eu_k\rangle^2}{(\lambda_j-\lambda_k)^2}}\leq Cx^2\rr_j(\Sigma)\sqrt{\sum_{k\neq j}\frac{\lambda_j\lambda_k}{(\lambda_j-\lambda_k)^2}}.
\end{equation*}
This completes the proof.
\end{proof}

\subsection{Separation of eigenvalues in the case of multiplicities}
\begin{lemma}\label{LemEVConcMult}
Let $r\geq 1$. Let $r_0\geq 1$ be such that $\mu_{r_0}\leq \mu_r/2$. Let $x>0$ be such that \eqref{EqBoundRelCoeff} holds. Moreover, suppose that Condition \eqref{EqCCondMult} holds. Then we have
\[
|\hat \lambda_j-\mu_r|\leq 6x m_r\mu_r/2\qquad \forall j\in\mathcal{I}_r.
\]
In particular, we have the following separation of eigenvalues
\[
|\hat{\lambda}_j-\mu_r|\leq  |\mu_r-\mu_s|/2\qquad \forall j\in\mathcal{I}_r,\forall s\neq r.
\]
\end{lemma}

\begin{proof}
First, note that the second claim follows from the first one by inserting Condition \eqref{EqCCondMult}. For the first claim, it suffices to show that the assumptions in \eqref{EqEvRD} and \eqref{EqEvLD} in Corollary \ref{PropEvRC} are satisfied with $y=6xm_r\mu_r/2$. We only verify the assumption in \eqref{EqEvRD}, \eqref{EqEvLD} follows from the same line of arguments. First, from \eqref{EqCCondMult} and $y=6xm_r\mu_r/2$, we get
\[
2x\sum_{s\geq r}\frac{m_s\mu_s}{\mu_r+y-\mu_s}\leq 2/3+2x\sum_{s\geq r}\frac{m_s\mu_s}{\mu_r-\mu_s}\leq 1.
\]
Thus the assumption in \eqref{EqEvRD} follows if we can show that
\begin{equation}\label{EqShit}
\sum_{k\geq j}\sum_{l\geq j}\frac{\lambda_k}{\lambda_j+y-\lambda_k}\frac{\lambda_l}{\lambda_j+y-\lambda_l}\bar{\eta}_{kl}^2\leq
4x^2\bigg(\sum_{s\geq r}\frac{m_s\mu_s}{\mu_r+y-\mu_s}\bigg)^2
\end{equation}
for all $j\in\mathcal{I}_r$. First, note that
\[
\sum_{k\geq j}\sum_{l\geq j}\frac{\lambda_k}{\lambda_j+y-\lambda_k}\frac{\lambda_l}{\lambda_j+y-\lambda_l}\bar{\eta}_{kl}^2
\leq\sum_{s\geq r}\sum_{t\geq r}\frac{\|Q_sEQ_t\|_2^2}{(\mu_r+y-\mu_s)(\mu_r+y-\mu_t)}
\]
and the right hand side is equal to
\begin{align}
&\sum_{r\leq s< r_0}\sum_{r\leq t< r_0}\frac{\|Q_sEQ_t\|_2^2}{(\mu_r+y-\mu_s)(\mu_r+y-\mu_t)}\nonumber\\
+&\sum_{r\leq s< r_0}\sum_{t\geq r_0}\frac{2\|Q_sEQ_t\|_2^2}{(\mu_r+y-\mu_s)(\mu_r+y-\mu_t)}\nonumber\\
+&\sum_{s\geq r_0}\sum_{t\geq r_0}\frac{\|Q_sEQ_t\|_2^2}{(\mu_r+y-\mu_s)(\mu_r+y-\mu_t)}.\label{EqShit1}
\end{align}
By the property of $r_0$, we have for all $s\geq r_0$
\begin{equation}\label{EqDefR0y}
\frac{1}{\mu_r+y-\mu_s}\leq\frac{1}{\mu_r+y-\mu_{r_0}}\leq \frac{2}{\mu_r+y}\leq \frac{2}{\mu_r+y-\mu_s}.
\end{equation}
Using the second inequality in \eqref{EqDefR0y}, we can bound \eqref{EqShit1} by
\begin{align*}
&\sum_{r\leq s< r_0}\sum_{r\leq t< r_0}\frac{\|Q_sEQ_t\|_2^2}{(\mu_r+y-\mu_s)(\mu_r+y-\mu_t)}
\\&+\sum_{r\leq s< r_0}\frac{4\|Q_sEQ_{\geq r_0}\|_2^2}{(\mu_r+y-\mu_s)(\mu_r+y)}+\frac{4\|Q_{\geq r_0}EQ_{\geq r_0}\|_2^2}{(\mu_r+y)(\mu_r+y)}.
\end{align*}
Inserting \eqref{EqBoundRelCoeff}, this is bounded by
\begin{align*}
&x^2\sum_{r\leq s< r_0}\sum_{r\leq t< r_0}\frac{m_s\mu_s}{\mu_r+y-\mu_s}\frac{m_t\mu_t}{\mu_r+y-\mu_t}\\&+8x^2\sum_{r\leq s< r_0}\frac{m_s\mu_s}{\mu_r+y-\mu_s}\frac{\operatorname{tr}_{\geq r_0}(\Sigma)}{\mu_r+y}+4x^2\frac{\operatorname{tr}_{\geq r_0}(\Sigma)}{\mu_r+y}\frac{\operatorname{tr}_{\geq r_0}(\Sigma)}{\mu_r+y}
\end{align*}
which, by the last inequality in \eqref{EqDefR0y}, is bounded by
\begin{align*}
&4x^2\sum_{r\leq s< r_0}\sum_{r\leq t< r_0}\frac{m_s\mu_s}{\mu_r+y-\mu_s}\frac{m_t\mu_t}{\mu_r+y-\mu_t}\\
+&8x^2\sum_{r\leq s< r_0}\sum_{t\geq r_0}\frac{m_s\mu_s}{\mu_r+y-\mu_s}\frac{m_t\mu_t}{\mu_r+y-\mu_t}\\
+&4x^2\sum_{s\geq r_0}\sum_{t\geq r_0}\frac{m_s\mu_s}{\mu_r+y-\mu_s}\frac{m_t\mu_t}{\mu_r+y-\mu_t}
=4x^2\bigg(\sum_{s\geq r}\frac{m_s\mu_s}{\mu_r+y-\mu_s}\bigg)^2.
\end{align*}
From these inequalities \eqref{EqShit} follows.
\end{proof}

\subsection{Contraction argument for spectral projectors}

\begin{lemma}\label{lemProjCalc1}
Let $r\geq 1$. Let $r_0\geq 1$ be such that $\mu_{r_0}\leq \mu_r/2$. Let $x>0$ be such that \eqref{EqBoundRelCoeff} holds. Moreover, suppose that Condition \eqref{EqCCondMult} holds. Then we have
\[
\|\hat Q_rQ_s\|_2\leq \frac{2\|\hat Q_rE Q_s\|_2}{|\mu_r-\mu_s|}
\]
for all $s<r_0$, $s\neq r$ and
\[
\|\hat Q_rQ_{\geq r_0}\|_2\leq \frac{2\|\hat Q_rEQ_{\geq r_0}\|_2}{|\mu_r-\mu_{r_0}|}.
\]
\end{lemma}
\begin{proof} For all $s\neq r$, we have
\begin{align}
4\|\hat Q_rE Q_s\|_2^2&=\sum_{j\in\mathcal{I}_r}\sum_{k\in\mathcal{I}_s}4(\hat\lambda_j-\mu_s)^2\langle \hat u_j,u_k\rangle^2\nonumber\\
&=\sum_{j\in\mathcal{I}_r}\sum_{k\in\mathcal{I}_s}4(\hat \lambda_j-\mu_r+\mu_r-\mu_s)^2\langle \hat u_j,u_k\rangle^2\nonumber\\
&\geq\sum_{j\in\mathcal{I}_r}\sum_{k\in\mathcal{I}_s} (\mu_r-\mu_s)^2\langle \hat u_j,u_k\rangle^2=(\mu_r-\mu_s)^2\|\hat Q_rQ_s\|_2^2,\label{EqLem6Proof}
\end{align}
where we used the second part of Lemma \ref{LemEVConcMult} in the inequality. Taking square roots on both sides gives the first claim.
Summing the above inequality over $s\geq r_0$, we get
\begin{align*}
\|\hat Q_rQ_{\geq r_0}\|_2^2=\sum_{s\geq r_0}\|\hat Q_rQ_s\|_2^2&\leq \sum_{s\geq r_0}\frac{4\|\hat Q_rE Q_s\|_2^2}{(\mu_r-\mu_s)^2}\leq \sum_{s\geq r_0}\frac{4\|\hat Q_rE Q_s\|_2^2}{(\mu_r-\mu_{r_0})^2}=\frac{4\|\hat Q_rEQ_{\geq r_0}\|_2^2}{(\mu_r-\mu_{r_0})^2},
\end{align*}
which gives the second claim.
\end{proof}

Our next result is the contraction property for spectral projectors.

\begin{lemma}\label{LemBasIneq} Let $r\geq 1$. Let $r_0\geq 1$ be such that $\mu_{r_0}\leq \mu_r/2$. Let $x>0$ be such that \eqref{EqBoundRelCoeff} holds. Moreover, suppose that Condition \eqref{EqCCondMult} holds. Then we have
\[
\|\hat Q_rQ_s\|_2\leq Cx \frac{\sqrt{m_r\mu_rm_s\mu_s}}{|\mu_r-\mu_s|}
\]
for all $s<r_0$, $s\neq r$ and
\[
\|\hat Q_rQ_{\geq r_0}\|_2\leq Cx\frac{\sqrt{m_r\mu_r\operatorname{tr}_{\geq r_0}(\Sigma)}}{|\mu_r-\mu_{r_0}|}\leq Cx\sqrt{\sum_{s\geq r_0}\frac{m_r\mu_rm_s\mu_s}{(\mu_r-\mu_s)^2}}.
\]
\end{lemma}
\begin{proof}
By the identity $I=\sum_{t\geq 1}Q_t$, the triangle inequality, and the fact that the Hilbert-Schmidt norm is sub-multiplicative, we have \begin{align*}
\|\hat Q_rE Q_s\|_2
&\leq \sum_{t<r_0}\|\hat Q_rQ_tE Q_s\|_2+\|\hat Q_rQ_{\geq r_0}E Q_s\|_2\\
&\leq \|Q_rEQ_s\|_2+\sum_{\substack{t<r_0:\\t\neq r}}\|\hat Q_rQ_t\|_2\|Q_tE Q_s\|_2+\|\hat Q_rQ_{\geq r_0}\|_2\|Q_{\geq r_0}E Q_s\|_2
\end{align*}
for all $s<r_0$, $s\neq r$.
From \eqref{EqBoundRelCoeff}, we get
\begin{align*}
\|\hat Q_rE Q_s\|_2\leq x\sqrt{m_r\mu_r m_s\mu_s}+\sum_{\substack{t< r_0:\\t\neq r}}x\sqrt{ m_t\mu_tm_s\mu_s}\|\hat Q_rQ_t\|_2+x\sqrt{\operatorname{tr}_{\geq r_0}(\Sigma)m_s\mu_s}\|\hat Q_rQ_{\geq r_0}\|_2
\end{align*}
for all $s<r_0$, $s\neq r$. Similarly, we have
\begin{align*}
\|\hat Q_rEQ_{\geq r_0}\|_2\leq x\sqrt{m_r\mu_r \operatorname{tr}_{\geq r_0}(\Sigma)}+\sum_{\substack{t< r_0:\\t\neq r}}x\sqrt{m_t\mu_t\operatorname{tr}_{\geq r_0}(\Sigma)}\|\hat Q_rQ_t\|_2+x\operatorname{tr}_{\geq r_0}(\Sigma)\|\hat Q_rQ_{\geq r_0}\|_2.
\end{align*}
Inserting Lemma \ref{lemProjCalc1}, we get
\begin{align*}
\|\hat Q_rQ_s\|_2\leq 2x\frac{\sqrt{m_s\mu_s}}{|\mu_r-\mu_s|}\Big(\sqrt{m_r\mu_r}+\sum_{\substack{t<r_0:\\t\neq r}}\sqrt{m_t\mu_t}\|\hat Q_rQ_t\|_2+\sqrt{\operatorname{tr}_{\geq r_0}(\Sigma)}\|\hat Q_rQ_{\geq r_0}\|_2\Big)
\end{align*}
for all $s<r_0$, $s\neq r$ and
\begin{align*}
\|\hat Q_rQ_{\geq r_0}\|_2\leq 2x\frac{\sqrt{\operatorname{tr}_{\geq r_0}(\Sigma)}}{|\mu_r-\mu_{r_0}|}\Big(\sqrt{m_r\mu_r}+\sum_{\substack{t<r_0:\\t\neq r}}\sqrt{m_t\mu_t}\|\hat Q_rQ_t\|_2+\sqrt{\operatorname{tr}_{\geq r_0}(\Sigma)}\|\hat Q_rQ_{\geq r_0}\|_2\Big).
\end{align*}
Setting
\[
\alpha_t= \sqrt{m_t\mu_t}\|\hat Q_rQ_t\|_2\quad\forall t<r_0,t\neq r,\quad\alpha_{r_0}=\sqrt{\operatorname{tr}_{\geq r_0}(\Sigma)}\|\hat Q_rQ_{\geq r_0}\|_2,
\]
these inequalities can be written as
\begin{equation}\label{EqRecEq1}
\alpha_s\leq2x\frac{ m_s\mu_s}{|\mu_r-\mu_s|}(\sqrt{m_r\mu_r}+ \sum_{\substack{t\leq  r_0:\\t\neq r}}\alpha_t)
\end{equation}
for all $s<r_0$, $s\neq r$ and
\begin{equation}\label{EqRecEq2}
\alpha_{r_0}\leq2 x\frac{\operatorname{tr}_{\geq r_0}(\Sigma)}{|\mu_r-\mu_{r_0}|}(\sqrt{m_r\mu_r}+ \sum_{\substack{t\leq r_0:\\t\neq r}}\alpha_t).
\end{equation}
By the properties of $r_0$, we have for all  $s\geq r_0$
\begin{equation}\label{EqDefR0}
\frac{1}{\mu_r-\mu_s}\leq\frac{1}{\mu_r-\mu_{r_0}}\leq \frac{2}{\mu_r}\leq \frac{2}{\mu_r-\mu_s}.
\end{equation}
In particular, combining \eqref{EqDefR0} with \eqref{EqCCondMult}, we have
\[
2x\sum_{\substack{s<r_0:\\s\neq r}}\frac{m_s\mu_s}{|\mu_r-\mu_s|}+2x\frac{ \operatorname{tr}_{\geq r_0}(\Sigma)}{|\mu_r-\mu_{r_0}|}\leq 4x\sum_{s\neq r}\frac{ m_s\mu_s}{|\mu_r-\mu_s|}\leq 2/3.
\]
Summing the inequalities in \eqref{EqRecEq1} and \eqref{EqRecEq2}, again arrive at the contraction inequality
\begin{equation}
3\sum_{\substack{s\leq r_0:\\s\neq r}}\alpha_s\leq 2\sqrt{m_r\mu_r}+2\sum_{\substack{t\leq r_0:\\t\neq r}}\alpha_t,
\end{equation}
and thus
\begin{equation}\label{EqSumIneq}
\sum_{\substack{s\leq r_0:\\s\neq r}}\alpha_s\leq 2\sqrt{m_r\mu_r}.
\end{equation}
Plugging \eqref{EqSumIneq} into \eqref{EqRecEq1}, we conclude that
\begin{equation*}
\sqrt{m_s\mu_s}\|\hat Q_rQ_s\|_2\leq 6x\frac{ m_s\mu_s\sqrt{m_r\mu_r}}{|\mu_r-\mu_s|}
\end{equation*}
for all $s<r_0$, $s\neq r$, which gives the first claim. Similarly, plugging \eqref{EqSumIneq} into \eqref{EqRecEq2}, we get
\begin{align*}
\sqrt{\operatorname{tr}_{\geq r_0}(\Sigma)}\|\hat Q_rQ_{\geq r_0}\|_2&\leq 6x\frac{\operatorname{tr}_{\geq r_0}(\Sigma)\sqrt{m_r\mu_r}}{|\mu_r-\mu_{r_0}|}
\end{align*}
which gives the first inequality of the second claim. The second inequality follows from invoking \eqref{EqDefR0}. This completes the proof.
\end{proof}

\begin{proposition}\label{PropProjConcMult}
Let $r\geq 1$ and $r_0\geq 1$ be such that $\mu_{r_0}\leq \mu_r/2$. Let $x>0$ be such that \eqref{EqBoundRelCoeff} holds. Moreover, suppose that Condition \eqref{EqCCondMult} holds. Then we have
\begin{equation*}
\|\hat{Q}_r-Q_r\|_2\leq Cx\sqrt{\sum_{s\neq r}\frac{m_r\mu_rm_s\mu_s}{(\mu_r-\mu_s)^2}}\leq Cx\rr_r(\Sigma).
\end{equation*}
\end{proposition}
\begin{proof}
Using that orthogonal projectors are idempotent and self-adjoint, we have
\begin{align*}
\|\hat{Q}_r-Q_r\|_2^2=2\langle\hat Q_r,I-Q_r\rangle
& =\sum_{\substack{s< r_0:\\s\neq r}}2\langle\hat Q_r,Q_s\rangle+2\langle\hat Q_r,Q_{\geq r_0}\rangle\\
&=\sum_{\substack{s< r_0:\\s\neq r}}2\|\hat{Q}_rQ_s\|_2^2+2\|\hat{Q}_rQ_{\geq r_0}\|_2^2.
\end{align*}
Inserting Lemma \ref{LemBasIneq} gives the first inequality. The second inequality follows from
\[
\sum_{s\neq r}\frac{m_r\mu_rm_s\mu_s}{(\mu_r-\mu_s)^2}\leq \rr_r(\Sigma)\sum_{s\neq r}\frac{m_s\mu_s}{|\mu_r-\mu_s|}\leq   \rr^2_r(\Sigma),
\]
where we applied Condition \eqref{EqCCondMult} twice.
\end{proof}

\subsection{Proof of Theorems \ref{ThmExpEvMult} and \ref{ThmProjExp}}
\begin{proof}[Proof of Theorem \ref{ThmExpEvMult}]
By the Hoffman-Wielandt inequality with the $1$-norm, we have
\begin{align*}
\sum_{k=1}^{m_r}|\lambda_k(\hat Q_r(\hat \Sigma-\mu_rI)\hat Q_r)-\lambda_k(Q_rEQ_r)|\leq \|\hat Q_r(\hat \Sigma-\mu_rI)\hat Q_r-Q_rEQ_r\|_1.
\end{align*}
Indeed, we can apply the infinite-dimensional version, see e.g. \cite[Theorem 5.1]{M64}, or the finite-dimensional version in \cite[Equation (1.64)]{T12} or \cite{B07}, since both $Q_r$ and $\hat Q_r$ have an $m_r$-dimensional range. Thus it suffices to show that
\begin{equation*}
\|\hat Q_r(\hat \Sigma-\mu_rI)\hat Q_r-Q_rEQ_r\|_1\leq Cx^2m_r\mu_r\rr_r(\Sigma).
\end{equation*}
We decompose
\begin{align}
\hat Q_r(\hat \Sigma-\mu_rI)\hat Q_r-Q_rEQ_r=(\hat{Q}_rE\hat{Q}_r-Q_rEQ_r)+\hat{Q}_r( \Sigma-\lambda_r I)\hat{Q}_r\label{EqNuclearDec}.
\end{align}
Thus it suffices to show that  the trace norm of the right-hand side is bounded by $Cx^2m_r\mu_r\rr_r(\Sigma)$.
Using the triangle inequality and \eqref{EqDefR0}, the trace norm of the last term on the right-hand side of \eqref{EqNuclearDec} can be bounded as follows:
\begin{align*}
\|\hat{Q}_r(\Sigma-\lambda_r I)\hat{Q}_r\|_1
&\leq \sum_{s\geq 1}\|\hat{Q}_rQ_s(\Sigma-\mu_rI)\hat{Q}_r\|_1\\
&=\sum_{s\neq r}\|(\mu_s-\mu_r)\hat{Q}_rQ_s\hat{Q}_r\|_1\\
&=\sum_{s\neq r}|\mu_r-\mu_s|\|\hat{Q}_rQ_s\|_2^2\\
&\leq \sum_{\substack{s<r_0:\\s\neq r}}|\mu_r-\mu_s|\|\hat{Q}_rQ_s\|_2^2+2|\mu_r-\mu_{r_0}|\|\hat{Q}_rQ_{\geq r_0}\|_2^2.
\end{align*}
Using Lemma \ref{LemBasIneq} and \eqref{EqDefR0}, we conclude that
\begin{align*}
\|\hat{Q}_r(\Sigma-\lambda_r I)\hat{Q}_r\|_1&\leq Cx^2\bigg(\sum_{\substack{s<r_0:\\s\neq r}}\frac{m_r\mu_rm_s\mu_s}{|\mu_r-\mu_s|}+\frac{\operatorname{tr}_{\geq r_0}(\Sigma)}{\mu_r-\mu_{r_0}}\bigg)\\
&\leq Cx^2\sum_{s\neq r}\frac{m_r\mu_rm_s\mu_s}{|\mu_r-\mu_s|}\leq Cx^2m_r\mu_r\rr_r(\Sigma).
\end{align*}
Similarly, the trace norm of the first term on the right-hand side of \eqref{EqNuclearDec} can be bounded as follows:
\begin{align}
\|\hat{Q}_rE\hat{Q}_r-Q_rEQ_r\|_1
&\leq \|\hat{Q}_r\sum_{s\neq r}Q_sE\sum_{t\neq r}Q_t\hat{Q}_r\|_1+2\|\hat{Q}_r\sum_{s\neq r}Q_sEQ_r\hat{Q}_r\|_1\nonumber\\
&+\|\hat{Q}_rQ_rEQ_r\hat{Q}_r-Q_rEQ_r\|_1\label{EqDecSecTerm}.
\end{align}
We start with the second term on the right-hand side of \eqref{EqDecSecTerm}. From the triangle inequality and the fact that the Hilbert-Schmidt norm is sub-multiplicative, we have
\begin{align*}
&\|\hat{Q}_r\sum_{s\neq r}Q_sEQ_r\hat{Q}_r\|_1\\
&\leq \sum_{\substack{s<r_0:\\s\neq r}}\|\hat{Q}_rQ_sEQ_r\hat{Q}_r\|_1+\|\hat{Q}_rQ_{\geq r_0}EQ_r\hat{Q}_r\|_1\\
&\leq \sum_{\substack{s<r_0:\\s\neq r}}\|\hat{Q}_rQ_s\|_2\|Q_sEQ_r\|_2+\|\hat{Q}_rQ_{\geq r_0}\|_2\|Q_{\geq r_0}EQ_r\|_2.
\end{align*}
Applying Lemma \ref{LemBasIneq}, \eqref{EqBoundRelCoeff}, and \eqref{EqDefR0}, we get
\begin{align*}
\|\hat{Q}_r\sum_{s\neq r}Q_sEQ_r\hat{Q}_r\|_1&\leq Cx^2\bigg(\sum_{\substack{s<r_0:\\s\neq r}}\frac{m_r\mu_rm_s\mu_s}{|\mu_r-\mu_s|}+\frac{m_r\mu_r\operatorname{tr}_{\geq r_0}(C)}{\mu_r-\mu_{r_0}}\bigg)\\
&\leq Cx^2\sum_{s\neq r}\frac{m_r\mu_rm_s\mu_s}{|\mu_r-\mu_s|}\leq Cx^2m_r\mu_r\rr_r(\Sigma).
\end{align*}
Similarly, the first term on the right-hand side of \eqref{EqDecSecTerm} can be bounded as follows:
\begin{align*}
\|\hat{Q}_r\sum_{s\neq r}Q_sE\sum_{t\neq r}Q_t\hat{Q}_r\|_1
&\leq \sum_{\substack{s<r_0:\\s\neq r}}\sum_{\substack{t<r_0:\\t\neq r}}\|\hat{Q}_rQ_sEQ_t\hat{Q}_r\|_1\\
&+2\sum_{\substack{s<r_0:\\s\neq r}}\|\hat{Q}_rQ_sEQ_{\geq r_0}\hat{Q}_r\|_1+\|\hat{Q}_rQ_{\geq r_0}EQ_{\geq r_0}\hat{Q}_r\|_1\\
&\leq \sum_{\substack{s<r_0:\\s\neq r}}\sum_{\substack{t<r_0:\\t\neq r}}\|\hat{Q}_rQ_s\|_2\|Q_sEQ_t\|_2\|Q_t\hat{Q}_r\|_2\\
&+2\sum_{\substack{s<r_0:\\s\neq r}}\|\hat{Q}_rQ_s\|_2\|Q_sEQ_{\geq r_0}\|_2\|Q_{\geq r_0}\hat{Q}_r\|_2\\
&+\|\hat{Q}_rQ_{\geq r_0}\|_2\|Q_{\geq r_0}EQ_{\geq r_0}\|_2\|Q_{\geq r_0}\hat{Q}_r\|_2.
\end{align*}
Applying Lemma \ref{LemBasIneq}, \eqref{EqBoundRelCoeff}, \eqref{EqDefR0}, and Condition \eqref{EqCCondMult}, we conclude that
\begin{align*}
\|\hat{Q}_r\sum_{s\neq r}Q_sE\sum_{t\neq r}Q_t\hat{Q}_r\|_1&\leq Cx^3m_r\mu_r\bigg(\sum_{\substack{s<r_0:\\s\neq r}}\frac{m_s\mu_s}{|\mu_r-\mu_s|}+\frac{\operatorname{tr}_{\geq r_0}(\Sigma)}{\mu_r-\mu_{r_0}}\bigg)^2\\
&\leq Cx^3m_r\mu_r\rr_r^2(\Sigma)\leq Cx^2m_r\mu_r\rr_r(\Sigma).
\end{align*}
Finally, using Proposition \ref{PropProjConcMult} and \eqref{EqBoundRelCoeff}, the last term on the right-hand side of \eqref{EqDecSecTerm} can be bounded as follows:
\begin{align*}
\|\hat{Q}_rQ_rEQ_r\hat{Q}_r-Q_rEQ_r\|_1
&=\|(\hat{Q}_r-Q_r)Q_rEQ_r\hat{Q}_r+Q_rEQ_r(\hat{Q}_r-Q_r)\|_1\\
&\leq  2\|\hat{Q}_r-Q_r\|_2\|Q_rEQ_r\|_2 \leq Cx^2 m_r\mu_r\rr_r(\Sigma).
\end{align*}
Hence, all summands on the right-hand side of \eqref{EqDecSecTerm} are bounded by $Cx^2\rr_r(\Sigma) m_r\mu_r$, and the claim follows.
\end{proof}

\begin{proof}[Proof of Theorem \ref{ThmProjExp}] Expanding the right-hand side of the identity $Q_r(\hat Q_r-Q_r)Q_r=(Q_r-I+I)(\hat Q_r-Q_r)(Q_r-I+I)$, we get
\begin{equation}\label{EqBF2}
\hat Q_r-Q_r=\hat Q_r(I-Q_r)+(I-Q_r)\hat Q_r+Q_r(\hat Q_r-Q_r)Q_r-(I-Q_r)\hat Q_r(I-Q_r).
\end{equation}
Using that $R_r(\Sigma-\mu_rI)=(\Sigma-\mu_rI)R_r=I-Q_r$, we have
\begin{align}
\hat{Q}_r(I-Q_r)
&=\hat{Q}_r(\hat \Sigma-\mu_rI+ \Sigma-\hat\Sigma)R_r\nonumber\\
&=\hat Q_r(\hat\Sigma-\mu_rI)R_r-Q_rER_r-(\hat{Q}_r-Q_r)
ER_r\label{EqBexp1}
\end{align}
and
\begin{align}
(I-Q_r)\hat{Q}_r
&=R_r(\hat \Sigma-\mu_rI+ \Sigma-\hat\Sigma)\hat{Q}_r\nonumber\\
&=R_r(\hat\Sigma-\mu_rI)\hat{Q}_r-R_rEQ_r-R_rE(\hat{Q}_r-Q_r)\label{EqBexp2}.
\end{align}
Hence, inserting \eqref{EqBexp1} and \eqref{EqBexp2} into \eqref{EqBF2} and using the triangle inequality, we get
\begin{align}
&\|\hat Q_r-Q_r+R_rEQ_r+Q_rER_r\|_2\nonumber\\
&\leq 2\|\hat Q_r(\hat\Sigma-\mu_rI)R_r\|_2+2\|(\hat{Q}_r-Q_r)ER_r\|_2\nonumber\\
&+\|Q_r(\hat Q_r-Q_r)Q_r\|_2+\|(I-Q_r)\hat Q_r(I-Q_r)\|_2\label{EqDecExpSP}.
\end{align}
We now bound successively the four terms on the right-hand side of \eqref{EqDecExpSP}. First, by Lemma \ref{LemEVConcMult}, we have
\begin{align*}
\|\hat Q_r(\hat\Sigma-\mu_rI)R_r\|_2&=\sqrt{\sum_{j\in\mathcal{I}_r}(\hat\lambda_j-\mu_r)^2\|(\hat u_j\otimes\hat u_j) R_r\|_2^2}\\
&\leq Cxm_r\mu_r\sqrt{\sum_{j\in\mathcal{I}_r}\|(\hat u_j\otimes\hat u_j) R_r\|_2^2}= Cxm_r\mu_r\|\hat Q_r R_r\|_2.
\end{align*}
Combined with
\begin{align*}
\|\hat Q_rR_r\|_2&=\sqrt{\sum_{s\neq r}\frac{\|\hat Q_rQ_s\|_2^2}{(\mu_r-\mu_s)^2}}\leq \sqrt{\sum_{\substack{s<r_0:\\s\neq r}}\frac{\|\hat Q_rQ_s\|_2^2}{(\mu_r-\mu_s)^2}+\frac{\|\hat Q_rQ_{\geq r_0}\|_2^2}{(\mu_r-\mu_{r_0})^2}}\\
&\leq Cx\sqrt{\sum_{\substack{s<r_0:\\s\neq r}}\frac{m_r\mu_rm_s\mu_s}{(\mu_r-\mu_s)^4}+\frac{m_r\mu_r\operatorname{tr}_{\geq r_0}(\Sigma)}{(\mu_r-\mu_{r_0})^4}}\leq Cx\sqrt{\sum_{s\neq r}\frac{m_r\mu_rm_s\mu_s}{(\mu_r-\mu_s)^4}}
\end{align*}
(which follows from Lemma \ref{LemBasIneq} and \eqref{EqDefR0}), we get
\begin{align*}
\|\hat Q_r(\hat\Sigma-\mu_rI)\hat{Q}_rR_r\|_2&\leq Cx^2\sqrt{\sum_{s\neq r}\frac{m_r^2\mu_r^2}{(\mu_r-\mu_s)^2}\frac{m_r\mu_rm_s\mu_s}{(\mu_r-\mu_s)^2}}\leq Cx^2\rr_r(\Sigma)\sqrt{\sum_{s\neq r}\frac{m_r\mu_rm_s\mu_s}{(\mu_r-\mu_s)^2}}.
\end{align*}
For the second term on the right-hand side of \eqref{EqDecExpSP}, note that
\begin{align*}
\|(\hat{Q}_r-Q_r)ER_r\|_2&=\sqrt{\sum_{s\neq r}\frac{\|\hat{Q}_r-Q_r)EQ_s\|_2^2}{(\mu_r-\mu_s)^2}}\\
&\leq \sqrt{\sum_{\substack{s<r_0:\\s\neq r}}\frac{\|\hat{Q}_r-Q_r)EQ_s\|_2^2}{(\mu_r-\mu_s)^2}+\frac{\|\hat{Q}_r-Q_r)EQ_{\geq r_0}\|_2^2}{(\mu_r-\mu_{r_0})^2}}.
\end{align*}
By the identity $I=\sum_{t\geq 1}Q_t$, the triangle inequality, and the fact that the Hilbert-Schmidt norm is sub-multiplicative, we have
\begin{align*}
&\frac{\|(\hat{Q}_r-Q_r)EQ_s\|_2}{|\mu_r-\mu_s|}\\
&\leq \sum_{t< r_0}\frac{\|(\hat{Q}_r-Q_r)Q_tEQ_s\|_2}{|\mu_r-\mu_s|}+\frac{\|(\hat{Q}_r-Q_r)Q_{\geq r_0}EQ_s\|_2}{|\mu_r-\mu_s|}\\
&\leq \frac{\|\hat{Q}_r-Q_r\|_2\|Q_rEQ_s\|_2}{|\mu_r-\mu_s|}+\sum_{t< r_0:t\neq r}\frac{\|\hat{Q}_rQ_t\|\|Q_tEQ_s\|_2}{|\mu_r-\mu_s|}+\frac{\|\hat{Q}_rQ_{\geq r_0}\|\|Q_{\geq r_0}EQ_s\|_2}{|\mu_r-\mu_s|}
\end{align*}
and similarly
\begin{align*}
&\frac{\|(\hat{Q}_r-Q_r)EQ_{\geq r_0}\|_2}{|\mu_r-\mu_{r_0}|}\\
&\leq \frac{\|\hat{Q}_r-Q_r\|_2\|Q_rEQ_{\geq r_0}\|_2}{|\mu_r-\mu_{r_0}|}+\sum_{t< r_0:t\neq r}\frac{\|\hat{Q}_rQ_t\|\|Q_tEQ_{\geq r_0}\|_2}{|\mu_r-\mu_{r_0}|}+\frac{\|\hat{Q}_rQ_{\geq r_0}\|\|Q_{\geq r_0}EQ_{\geq r_0}\|_2}{|\mu_r-\mu_{r_0}|}.
\end{align*}
Hence, by Proposition \ref{PropProjConcMult}, Lemma \ref{LemBasIneq}, and \eqref{EqBoundRelCoeff}, we get
\begin{align*}
&\frac{\|(\hat{Q}_r-Q_r)EQ_s\|_2}{|\mu_r-\mu_s|}\\
&\leq Cx^2\rr_r(\Sigma)\frac{\sqrt{m_r\mu_rm_s\mu_s}}{|\mu_r-\mu_s|}+Cx^2\sum_{t< r_0:t\neq r}\frac{\sqrt{m_r\mu_rm_t\mu_t}}{|\mu_r-\mu_t|}\frac{\sqrt{m_t\mu_tm_s\mu_s}}{|\mu_r-\mu_s|}\\
&+Cx^2\frac{\sqrt{m_r\mu_r\operatorname{tr}_{\geq r_0}(\Sigma)}}{|\mu_r-\mu_{r_0}|}\frac{\sqrt{\operatorname{tr}_{\geq r_0}(\Sigma)m_s\mu_s}}{|\mu_r-\mu_s|}\\
&\leq Cx^2\frac{\sqrt{m_r\mu_rm_s\mu_s}}{|\mu_r-\mu_s|}\Big(\rr_r(\Sigma)+\sum_{\substack{t<r_0:\\t\neq r}}\frac{m_t\mu_t}{|\mu_r-\mu_t|}+\frac{\operatorname{tr}_{\geq r_0}(\Sigma)}{|\mu_r-\mu_{r_0}|}\Big)\\
&\leq Cx^2\rr_r(\Sigma)\frac{\sqrt{m_r\mu_rm_s\mu_s}}{|\mu_r-\mu_s|}
\end{align*}
and similarly
\begin{align*}
\frac{\|(\hat{Q}_r-Q_r)EQ_{\geq r_0}\|_2}{|\mu_r-\mu_{r_0}|}&\leq Cx^2\rr_r(\Sigma)\frac{\sqrt{m_r\mu_r\operatorname{tr}_{\geq r_0}(\Sigma)}}{|\mu_r-\mu_{r_0}|}\leq Cx^2\rr_r(\Sigma)\sqrt{\sum_{s\geq r_0}\frac{m_r\mu_rm_s\mu_s}{(\mu_r-\mu_s)^2}}.
\end{align*}
Thus
\begin{equation*}
\|(\hat{Q}_r-Q_r)ER_r\|_2\leq Cx^2\rr_r(\Sigma)\sqrt{\sum_{s\neq r}\frac{m_r\mu_rm_s\mu_s}{(\mu_r-\mu_s)^2}}.
\end{equation*}
Next, the third term on the right-hand side of \eqref{EqDecExpSP} is bounded as follows
\[
\|Q_r(\hat Q_r-Q_r)Q_r\|_2\leq \|Q_r(\hat Q_r-Q_r)Q_r\|_1=\operatorname{tr}(Q_r-Q_r\hat Q_r Q_r)=\|\hat Q_r-Q_r\|_2^2/2,
\]
where we used the fact that $Q_r-Q_r\hat Q_r Q_r$ is self-adjoint and positive. Thus, by  Proposition \ref{PropProjConcMult} and \eqref{EqDefR0},
\[
\|Q_r(\hat Q_r-Q_r)Q_r\|_2\leq Cx^2\sum_{s\neq r}\frac{m_r\mu_rm_s\mu_s}{(\mu_r-\mu_s)^2}\leq Cx^2\rr_r(\Sigma)\sqrt{\sum_{s\neq r}\frac{m_r\mu_rm_s\mu_s}{(\mu_r-\mu_s)^2}}.
\]
Similarly we have
\[
\|(I-Q_r)\hat Q_r(I-Q_r)\|_2=\|(I-Q_r)\hat Q_r\hat Q_r(I-Q_r)\|_2\leq \|\hat Q_r(I-Q_r)\|_2^2=\|\hat Q_r-Q_r\|_2^2/2
\]
and thus
\[
\|(I-Q_r)\hat Q_r(I-Q_r)\|_2\leq  Cx^2\rr_r(\Sigma) \sqrt{\sum_{s\neq r}\frac{m_r\mu_rm_s\mu_s}{(\mu_r-\mu_s)^2}}.
\]
This completes the proof.
\end{proof}

\section{Proofs of the inconsistency results}\label{section:proofs:inconsistency}
The purpose of this section is to prove the inconsistency results from Section~\ref{SecApp}. We first recall Setting \ref{setting:inconsistency:model} and introduce some further notation. Let $\Sigma =\sum_{j\ge 1}\lambda_j (u_j\otimes u_j)$ be a self-adjoint, positive trace class operator on $\mathcal{H}$. Let $ r\geq 1$ and $F=\sum_{j\leq r}\sqrt{\lambda_j}u_j$. Let $\epsilon$ be a Gaussian random variable in $\mathcal{H}$ with expectation $0$ and covariance operator $\Sigma-1/(2r)(F\otimes F)$, and let $f$ be a real random variable defined by $\P(f=0)=1-1/(2r^2)$ and $\P(f=\pm\sqrt{r})=1/(4r^2)$, independent of $\epsilon$. Let $X=f\cdot F+\epsilon$ with covariance operator $\Sigma$ and let $X_1,\dots,X_n$ be independent copies of $X$. Let
\begin{align*}
    \hat\Sigma=\frac{1}{n}\sum_{i=1}^nX_i\otimes X_i,\qquad
    \hat\Sigma_\epsilon=\frac{1}{n}\sum_{i=1}^n\epsilon_i\otimes\epsilon_i,\qquad \Sigma_\epsilon=\Sigma-\frac{1}{2r}F\otimes F,
\end{align*}
\begin{align*}
\tilde{\Sigma}=\Sigma+\tilde x(F\otimes F),\qquad \tilde x=\frac{1}{n}\sum_{i=1}^n\Big(f_i^2-\frac{1}{2r}\Big).
\end{align*}
For $0<z_1<z_2<\infty$, we define the events
\begin{align*}
    \tilde{\mathcal{E}}_{z_1}&=\Big\{\tilde x\geq \frac{z_1}{\sqrt{n}}\Big\},\qquad\tilde{\mathcal{E}}_{z_1,z_2}=\Big\{\frac{z_1}{\sqrt{n}}\leq \tilde x\leq \frac{z_2}{\sqrt{n}}\Big\},\qquad \tilde{\mathcal{E}}=\Big\{\frac{1}{n}\sum_{i=1}^nf_i^2\leq \frac{1}{r}\Big\}.
\end{align*}

\subsection{Inconsistency result for a rank-one perturbation}

In this section, we analyse the eigenstructure of the deterministic rank-one perturbation $\tilde\Sigma=\Sigma+\tilde x (F\otimes F)$
with $\tilde x>0$ fixed (corresponds to $f_1,\dots,f_n$ fixed) and $F=\sum_{j=1}^r\sqrt{\lambda}_ju_j$. Since eigenvalues and eigenvectors of $\tilde\Sigma$ coincide with those of $\Sigma$ for indices $j>r$, we may restrict ourselves without loss of generality to
\begin{align}\label{eq:relative:pert:inconsistency}
    \tilde\Sigma=\sum_{j=1}^r\lambda_j (u_j\otimes u_j)+\tilde x (F\otimes F)\qquad \text{with }\tilde x>0\text{ and } F=\sum_{j=1}^r\sqrt{\lambda}_ju_j
\end{align}
throughout this section. Let $\tilde\lambda_1\geq\dots\geq\tilde\lambda_r>0$ be the eigenvalues of $\tilde\Sigma$, and $\tilde u_1,\dots,\tilde u_r$ be the corresponding eigenvectors of $\tilde\Sigma$. By the interlacing theorem for rank-one symmetric matrices (see e.g.~\cite[Corollary 4.3.9]{MR2978290}) we have
\begin{align}\label{eq:interlacing:rank-one}
    &\lambda_1\leq \tilde\lambda_1 \qquad \text{and}\qquad\lambda_j\leq \tilde\lambda_j\leq \lambda_{j-1}\text{ for }j=2,\dots,r.
\end{align}
Moreover, for $j=2,\dots,r$ such that $\lambda_j<\lambda_{j-1}$, we have $\lambda_j< \tilde\lambda_j< \lambda_{j-1}$ (resp. $\tilde\lambda_1>\lambda_1$ for $j=1$), and $\tilde\lambda_j$, $j =1 \ldots,r$, satisfies the secular equation (see e.g. \cite[Equation (5.14)]{MR1463942})
\begin{align}\label{eq:secular:equation}
\tilde x \sum_{k=1}^r \frac{\lambda_k}{\tilde\lambda_j-\lambda_k}=1.
\end{align}
Using these facts we provide inconsistency results for the eigenvalues $\tilde\lambda_j$ and the corresponding spectral projectors $\tilde P_j=\tilde u_j\otimes \tilde u_j$, $j=1,\dots,r$, provided that a relative rank condition is violated. We consider separately the cases $j=1$ and $j\geq 2$.

\begin{lemma}\label{lemma:eigenvalue:1:inconsistency:perfect:model}
Consider the rank-one perturbation in \eqref{eq:relative:pert:inconsistency}. Suppose that there is a constant $L>0$, such that
\begin{align}\label{eq:rrc:violated:j=1}
    \tilde x\Big(\sum_{k=2}^r \frac{\lambda_k}{\lambda_1-\lambda_k}+\frac{\lambda_1}{\lambda_1-\lambda_2}\Big)\geq L+1.
\end{align}
Then we have
\begin{align}\label{eq:eigenvalue:inconsistency:perfect:model:1}
    \frac{\tilde{\lambda}_1-\lambda_1}{\lambda_1-\lambda_2}\geq L
\end{align}
and
\begin{align}\label{eq:eigenvector:1:inconsistency:perfect:model:2}
    \|\tilde P_1-P_1\|_2^2\geq 2\frac{\lambda_2}{\lambda_1+\lambda_2}\Big(\frac{L}{L+1}\Big)^2\geq 2\frac{\lambda_2}{\lambda_1+\lambda_2}\Big(1-\frac{2}{L}\Big),
\end{align}
where the last inequality holds for $L>2$.
\end{lemma}

\begin{proof}
Write $\tilde \lambda_1=\lambda_1+A(\lambda_1-\lambda_2)$, with $A>0$ by \eqref{eq:interlacing:rank-one}. Then, by \eqref{eq:secular:equation}, we have
\begin{align*}
    1= \tilde x \sum_{k=1}^r \frac{\lambda_k}{\tilde\lambda_1-\lambda_k}&=\frac{\tilde x}{A} \frac{\lambda_1}{\lambda_1-\lambda_2}+\tilde x\sum_{k=2}^r \frac{\lambda_k}{\lambda_1-\lambda_k+A(\lambda_1-\lambda_2)}\\
    &\geq \frac{\tilde x}{1+A} \Big(\frac{\lambda_1}{\lambda_1-\lambda_2}+\sum_{k=2}^r \frac{\lambda_k}{\lambda_1-\lambda_k}\Big).
\end{align*}
Employing $\eqref{eq:rrc:violated:j=1}$, we obtain $A\geq L$ and \eqref{eq:eigenvalue:inconsistency:perfect:model:1} follows. Moreover, by \eqref{EqEvEq} with $j=1$, we have, for all $k=2,\dots,d$,
\begin{align}\label{eq:eigenangle:identity}
\sqrt{\lambda_k}\langle \tilde u_1,u_k\rangle=\tilde x\frac{\lambda_k}{\tilde \lambda_1-\lambda_k}\big(\sqrt{\lambda_1}\langle\tilde u_1,u_1\rangle+\sum_{l>1}\sqrt{\lambda_l}\langle\tilde u_1,u_l\rangle\big),
\end{align}
where we used that all relative coefficients off the diagonal of $\tilde \Sigma$ are equal to $\tilde x$. In particular, summing over $k>1$, we get
\begin{align*}
   \tilde x\frac{\lambda_1}{\tilde\lambda_1-\lambda_1} \sum_{l>1}\sqrt{\lambda_l}\langle \tilde u_1,u_l\rangle=\Big(1-\tilde x\frac{\lambda_1}{\tilde\lambda_1-\lambda_1}\Big)\sqrt{\lambda_1}\langle\tilde u_1,u_1\rangle,
\end{align*}
where we used again \eqref{eq:secular:equation}.
Inserting this into \eqref{eq:eigenangle:identity} with $k=2$, we get
\begin{align*}
    \langle \tilde u_1,u_2\rangle&=\sqrt{\frac{\lambda_2}{\lambda_1}}\frac{\tilde\lambda_1-\lambda_1}{\tilde\lambda_1-\lambda_2}\langle\tilde u_1,u_1\rangle.
\end{align*}
Inserting
\begin{align*}
    \frac{\tilde\lambda_{1}-\lambda_1}{\tilde\lambda_{1}-\lambda_2}&= \frac{A}{A+1}\geq \frac{L}{L+1},
\end{align*}
which follow from \eqref{eq:eigenvalue:inconsistency:perfect:model:1}, we conclude that
\begin{align*}
    |\langle \tilde u_1,u_2\rangle|\geq \sqrt{\frac{\lambda_2}{\lambda_1}}\frac{L}{L+1}|\langle\tilde u_1,u_1\rangle|.
\end{align*}
Hence,
\begin{align*}
    1\geq \langle \tilde u_1,u_1\rangle^2+\langle\tilde u_1,u_2\rangle^2\geq \Big(1+\frac{\lambda_2}{\lambda_1}\Big(\frac{L}{L+1}\Big)^2\Big)\langle \tilde u_1,u_1\rangle^2,
\end{align*}
which, combined with the identity $\|\tilde P_1-P_1\|_2^2=2-2\langle \tilde u_1,u_1\rangle^2$, leads to
\begin{align*}
    \|\tilde P_1-P_1\|_2^2\geq 2\frac{\frac{\lambda_2}{\lambda_1}(\frac{L}{L+1})^2}{1+\frac{\lambda_2}{\lambda_1}(\frac{L}{L+1})^2}.
\end{align*}
This can be simplified to \eqref{eq:eigenvector:1:inconsistency:perfect:model:2}.
\end{proof}

\begin{lemma}\label{lemma:eigenvalue:j:inconsistency:perfect:model}
Consider the rank-one perturbation in \eqref{eq:relative:pert:inconsistency}. Let $2\leq j\leq r$ be such that $\lambda_j<\lambda_{j-1}$. Suppose that there are constants $L>\ell>0$, such that
\begin{align}\label{eq:eigenvalue:j:inconsistency:perfect:model:cond}
    \tilde x\sum_{k=1,k\neq j-1}^r\frac{\lambda_k}{\lambda_{j-1}-\lambda_k}\geq L+1\qquad\text{and}\qquad \tilde x\frac{\lambda_{j-1}}{\lambda_{j-1}-\lambda_j}\leq \ell.
\end{align}
Then we have
\begin{align}\label{eq:eigenvalue:inconsistency:perfect:model:j}
0<\frac{\lambda_{j-1}-\tilde\lambda_j}{\lambda_{j-1}-\lambda_j}\leq \frac{\ell}{L}
\end{align}
and
\begin{align}\label{eq:eigenvector:inconsistency:perfect:model:2}
    \|\tilde P_j-P_j\|_2^2\geq 2\Big(1-\frac{\ell^2}{\ell^2+(L-\ell)^2}\Big)\geq  2\Big(1-\frac{4\ell^2}{L^2}\Big),
\end{align}
where the last inequality holds for $L>2l$.
\end{lemma}


\begin{proof}
By fact that $\lambda_j<\lambda_{j-1}$ and the discussion following \eqref{eq:interlacing:rank-one}, we can write
\begin{align}\label{eq:eigenvalue:inconsistency:perfect:model:j:version}
\tilde\lambda_j=\lambda_{j-1}-A(\lambda_{j-1}-\lambda_j)\quad\text{with}\quad 0<A<1.
\end{align}
Then, using again \eqref{eq:secular:equation}, we have
\begin{align*}
    1&=\tilde x\sum_{k=1}^r\frac{\lambda_k}{\tilde\lambda_j-\lambda_k}=-\frac{\tilde x}{A}\frac{\lambda_{j-1}}{\lambda_{j-1}-\lambda_j}+\tilde x\sum_{k\neq j-1}\frac{\lambda_k}{\lambda_{j-1}-\lambda_k-A(\lambda_{j-1}-\lambda_j)}\\
    &\geq -\frac{\tilde x}{A}\frac{\lambda_{j-1}}{\lambda_{j-1}-\lambda_j}+\tilde x\sum_{k\neq j-1}\frac{\lambda_k}{\lambda_{j-1}-\lambda_k}.
\end{align*}
Using \eqref{eq:eigenvalue:j:inconsistency:perfect:model:cond}, we arrive at
\begin{align*}
    1\geq -\frac{\ell}{A}+L+1\quad\text{and thus}\quad A\leq \frac{\ell}{L}.
\end{align*}
Inserting this into \eqref{eq:eigenvalue:inconsistency:perfect:model:j:version}, we get \eqref{eq:eigenvalue:inconsistency:perfect:model:j}.
Moreover, similarly as in the proof of Lemma \ref{lemma:eigenvalue:1:inconsistency:perfect:model}, we have, for all $k=1,\dots,r$, $k\neq j$,
\begin{align}\label{eq:eigenangle:identity:relative:j}
\sqrt{\lambda_k}\langle \tilde u_j,u_k\rangle= \tilde x\frac{\lambda_k}{\tilde \lambda_j-\lambda_k}\big(\sqrt{\lambda_j}\langle\tilde u_j,u_j\rangle+\sum_{l\neq j}\sqrt{\lambda_l}\langle\tilde u_j,u_l\rangle\big),
\end{align}
from which one can deduce (by the same line of arguments as above) that for $k=j-1$, we have
\begin{align*}
    \langle \tilde u_j,u_{j-1}\rangle=\sqrt{\frac{\lambda_{j-1}}{\lambda_j}}\frac{\tilde\lambda_j-\lambda_j}{\tilde\lambda_{j}-\lambda_{j-1}}\langle\tilde u_j,u_j\rangle.
\end{align*}
Inserting
\begin{align*}
    \frac{\tilde\lambda_{j}-\lambda_j}{\lambda_{j-1}-\tilde\lambda_j}= \frac{1-A}{A}\geq \frac{L-\ell}{\ell},
\end{align*}
which follows from \eqref{eq:eigenvalue:inconsistency:perfect:model:j:version} with $0<A<\ell/L$, we conclude
\begin{align*}
    |\langle \tilde u_j,u_{j-1}\rangle|&\geq\sqrt{\frac{\lambda_{j-1}}{\lambda_j}}\frac{L-\ell}{\ell}|\langle\tilde u_j,u_j\rangle|\geq \frac{L-\ell}{\ell}|\langle\tilde u_j,u_j\rangle|.
\end{align*}
Hence,
\begin{align*}
    1\geq \langle \tilde u_j,u_j\rangle^2+\langle\tilde u_j,u_{j-1}\rangle^2\geq \Big(1+\Big(\frac{L-\ell}{\ell}\Big)^2\Big)\langle \tilde u_j,u_j\rangle^2,
\end{align*}
which in turn yields \eqref{eq:eigenvector:inconsistency:perfect:model:2}, additionally using the identity $\|\tilde P_j-P_j\|_2^2=2-2\langle \tilde u_j,u_j\rangle^2$.
\end{proof}

\subsection{Invoking the relative Weyl and Davis-Kahan inequalities}

We turn to the problem of transferring the inconsistency results for $\tilde\lambda_j,\tilde P_j$ to the case of $\hat\lambda_j,\hat P_j$. A first possibility would be to use the classical Weyl inequality $|\tilde\lambda_j-\hat\lambda_j|\leq \|\hat \Sigma-\tilde \Sigma\|_\infty$ and the Davis-Kahan $\sin \Theta$ inequality (see e.g.~\cite{MR3379106})
\begin{align*}
    \|\hat P_j-\tilde P_j\|_2\leq 2\tilde g_j^{-1}\|\hat \Sigma-\tilde \Sigma\|_\infty,\qquad \tilde g_j=\min(\tilde \lambda_{j-1}-\tilde\lambda_j,\tilde\lambda_j-\tilde\lambda_{j+1}),
\end{align*}
in combination with concentration inequalities. Since this approach leads to an effective rank condition, we instead apply the relative versions from Section \ref{SecRelWeylDK} (by considering $\hat\Sigma$ as a perturbed version of $\tilde \Sigma$).
For this purpose, let
\begin{align*}
    \tilde T_{\leq j}(y)=\sum_{k\leq j}(\tilde\lambda_k+y-\tilde\lambda_j)^{-1/2}\tilde P_k\quad\text{and}\quad \tilde T_j=|\tilde R_j|^{1/2}+\tilde g_j^{-1/2}\tilde P_j,
\end{align*}
where $|\tilde R_j|^{1/2}=\sum_{k\neq j}|\tilde \lambda_k-\tilde \lambda_j|^{-1/2}\tilde P_k$ is the square-root of the reduced resolvent of $\tilde \Sigma$ at $\tilde \lambda_j$. Replacing $\Sigma$ by $\tilde \Sigma$, Lemmas \ref{LemEvRC} and \ref{lemma:relative:DK} yield the following:

\begin{corollary}\label{Cor:relative:Weyl:DK:tilde}
Let $j\geq 1$. Then we have:
\begin{itemize}
    \item[(i)] If $\Vert\tilde T_{\leq j}(y)(\hat\Sigma-\tilde\Sigma)\tilde T_{\leq j}(y)\Vert_\infty\leq 1$, then $\hat\lambda_j\geq \tilde\lambda_j-y$.
    \item[(ii)] If $\tilde g_j>0$, then $\|\hat P_j-\tilde P_j\|_2\leq C_1\Vert\tilde T_j(\hat\Sigma-\tilde\Sigma)\tilde T_j\Vert_\infty$ for some absolute constant $C_1>1$.
\end{itemize}
\end{corollary}

\subsection{A concentration bound}

In this section we proof the following concentration inequality for the quantities appearing in Corollary \ref{Cor:relative:Weyl:DK:tilde}.

\begin{lemma}\label{lemma:transition:tilde:hat}
In Setting \ref{setting:inconsistency:model}, there are absolute constants $c_1,c_2\in(0,1)$ such that the following holds. Consider $f_1,\dots,f_n$ as fixed and suppose that $\tilde x>0$ and $\tilde{\mathcal{E}}$ holds. Let $j\geq 1$. Then, with probability at least $1-e^{-t}$, $1\leq t\leq n$,
\begin{align}\label{eq:transition:tilde:hat:result:Proj}
    c_1\cdot\|\tilde T_j(\hat\Sigma-\tilde{\Sigma}) \tilde T_j\|_{\infty}
    &\leq \sqrt{\frac{1}{n}\frac{\tilde\lambda_j}{\tilde g_j}\rr_j(\tilde\Sigma)}\bigvee\frac{\tilde\lambda_j}{\tilde g_j}\sqrt{\frac{t}{n}},
\end{align}
provided that the right-hand side is smaller than or equal to $c_2$. Moreover, with probability at least $1-e^{-t}$, $1\leq t\leq n$,
\begin{align}\label{eq:transition:tilde:hat:result:Ev}
    c_1\cdot\|\tilde T_{\leq j}(y)(\hat\Sigma-\tilde{\Sigma}) \tilde T_{\leq j}(y)\|_{\infty}
    &\leq \sqrt{\frac{1}{n}\frac{\tilde\lambda_j}{y}\sum_{k\leq j}\frac{\tilde\lambda_k}{\tilde\lambda_k+y-\tilde\lambda_j}}\bigvee\frac{\tilde\lambda_j}{y}\sqrt{\frac{t}{n}},
\end{align}
provided that the right-hand side is smaller than or equal to $c_2$.
\end{lemma}

\begin{proof}
We work on the event that $\tilde x>0$ and $\tilde{\mathcal{E}}$ holds, assuming that $f_1,\dots,f_n$ are fixed. We decompose
\begin{align*}
\hat\Sigma-\tilde{\Sigma}=\hat\Sigma_\epsilon-\Sigma_\epsilon+F\otimes\Big(\frac{1}{n}\sum_{i=1}^nf_i\epsilon_i\Big)+\Big(\frac{1}{n}\sum_{i=1}^nf_i\epsilon_i\Big)\otimes F.
\end{align*}
Using the triangle inequality and the Cauchy-Schwarz inequality, we thus obtain
\begin{align}\label{eq:tilde:delta:1}
\|\tilde T_j(\hat\Sigma-\tilde{\Sigma}) \tilde T_j\|_{\infty}\leq \|\tilde T_j(\hat{\Sigma}_\epsilon-\Sigma_\epsilon)\tilde T_j\|_{\infty}+2\|\tilde T_j F\|\Big\|\frac{1}{n}\sum_{i=1}^nf_i\tilde T_j\epsilon_i\Big\|.
\end{align}
We now apply concentration inequalities to the two terms on the right-hand side. First, applying \cite[Corollary 2]{KL14} to $\epsilon_i'=\tilde T_j\epsilon_i$, $1\leq i\leq n$, having covariance $\tilde T_j\Sigma_\epsilon \tilde T_j$, we get that, with probability at least $1-e^{-t}$, $1\leq t\leq n$,
\begin{align}\label{eq:apply:KL}
    &\|\tilde T_j(\hat\Sigma_\epsilon-\Sigma_\epsilon)\tilde T_j\|_{\infty}\nonumber\\
    &\leq C\sqrt{\frac{1}{n}\|\tilde T_j\Sigma_\epsilon \tilde T_j\|_{\infty}\operatorname{tr}(\tilde T_j\Sigma_\epsilon \tilde T_j)}\bigvee \frac{\operatorname{tr}(\tilde T_j\Sigma_\epsilon \tilde T_j)}{n}\bigvee \|\tilde T_j\Sigma_\epsilon \tilde T_j\|_{\infty}\sqrt{\frac{t}{n}}.
\end{align}
Second, conditional on $f_1,\dots,f_n$, the random variable $\frac{1}{n}\sum_{i=1}^nf_i\tilde T_j\epsilon_i$ is Gaussian with expectation $0$ and covariance $(\frac{1}{n^2}\sum_{i=1}^nf_i^2)\tilde T_j\Sigma_\epsilon \tilde T_j$. Using Gaussian concentration inequalities (see, e.g.,~\cite[Theorem 5.6]{MR3185193}), we get, with probability at least $1-e^{-t}$, $t>0$,
\begin{align*}
    \Big\|\frac{1}{n}\sum_{i=1}^nf_i\tilde T_j\epsilon_i\Big\|\leq \E_{\epsilon}\Big\|\frac{1}{n}\sum_{i=1}^nf_i\tilde T_j\epsilon_i\Big\|+\sqrt{2t\Big(\frac{1}{n^2}\sum_{i=1}^nf_i^2\Big)\|\tilde T_j\Sigma_\epsilon \tilde T_j\|_\infty}.
\end{align*}
Now, on the event $\tilde{\mathcal{E}}$, we have $n^{-1}\sum_{i=1}^nf_i^2\leq 1/r$. Using the Cauchy-Schwarz inequality, we get, with probability at least $1-e^{-t}$, $t>0$,
\begin{align*}
    \Big\|\frac{1}{n}\sum_{i=1}^nf_i\tilde T_j\epsilon_i\Big\|&\leq \sqrt{\Big(\frac{1}{n^2}\sum_{i=1}^nf_i^2\Big)\operatorname{tr}(\tilde T_j\Sigma_\epsilon \tilde T_j)}+\sqrt{2t\Big(\frac{1}{n^2}\sum_{i=1}^nf_i^2\Big)\|\tilde T_j\Sigma_\epsilon \tilde T_j\|_{\infty}}\\
    &\leq \sqrt{\frac{1}{r}}\sqrt{\frac{1}{n}\operatorname{tr}(\tilde T_j\Sigma_\epsilon \tilde T_j)}+\sqrt{\frac{1}{r}}\sqrt{\frac{t}{n}\|\tilde T_j\Sigma_\epsilon \tilde T_j\|_{\infty}}.
\end{align*}
Combining this inequality with
\begin{align}\label{eq:max:ev:Tj}
    \|\tilde T_j F\|\leq  \sqrt{\max\Big(\frac{\tilde\lambda_j}{\tilde g_j},\frac{\tilde\lambda_{j-1}}{\tilde\lambda_{j-1}-\tilde\lambda_j}\Big)}\sqrt{r}\leq \sqrt{\frac{\tilde\lambda_j}{\tilde g_j}}\sqrt{2r},
\end{align}
we obtain with probability at least $1-e^{-t}$, $t>0$,
\begin{align}\label{eq:apply:GC}
    \|\tilde T_j F\|\Big\|\frac{1}{n}\sum_{i=1}^nf_i\tilde T_j\epsilon_i\Big\|
    &\leq C\sqrt{\frac{1}{n}\frac{\tilde\lambda_j}{\tilde g_j}\operatorname{tr}(\tilde T_j\Sigma_\epsilon \tilde T_j)}\bigvee\sqrt{\frac{t}{n}\frac{\tilde\lambda_j}{\tilde g_j}\|\tilde T_j\Sigma_\epsilon \tilde T_j\|_{\infty}}.
\end{align}
In order to simplify \eqref{eq:apply:KL} and \eqref{eq:apply:GC}, we now show that 
\begin{align}\label{eq:compare:tildeSigma:epsSigma}
    \operatorname{tr}(\tilde T_j\Sigma_\epsilon \tilde T_j)\leq  \rr_j(\tilde\Sigma)=\frac{\tilde\lambda_{j}}{\tilde g_{j}}+\sum_{k\neq j}\frac{\tilde\lambda_k}{|\tilde\lambda_{j}-\tilde\lambda_k|}\quad\text{and}\quad
    \|\tilde T_j\Sigma_\epsilon \tilde T_j\|_{\infty}\leq 2\frac{\tilde\lambda_j}{\tilde g_j}.
\end{align}
To see this, first note that 
\begin{align*}
    \Sigma_\epsilon=\Sigma-\frac{1}{2r} F\otimes F\leq\Sigma-\frac{1}{2r} F\otimes F+\Big(\frac{1}{n}\sum_{i=1}^nf_i^2\Big) F\otimes F=\tilde \Sigma,
\end{align*}
meaning that $\tilde\Sigma-\Sigma_\epsilon$ is positive. Hence, we get $\tilde T_j\Sigma_\epsilon \tilde T_j\leq \tilde T_j\tilde\Sigma \tilde T_j$, leading to $\operatorname{tr}(\tilde T_j\Sigma_\epsilon \tilde T_j)\leq  \operatorname{tr}(\tilde T_j\tilde\Sigma \tilde T_j)$ and $\|\tilde T_j\Sigma_\epsilon \tilde T_j\|_{\infty}\leq \|\tilde T_j\tilde\Sigma \tilde T_j\|_{\infty}$ and \eqref{eq:compare:tildeSigma:epsSigma} follows from inserting the definition of $\tilde T_j$ (compare also to \eqref{eq:max:ev:Tj}). Combining \eqref{eq:apply:KL} and \eqref{eq:apply:GC} with \eqref{eq:compare:tildeSigma:epsSigma}, we get, with probability at least $1-e^{-t}$, $1\leq t\leq n$,
\begin{align*}
   \|\tilde T_j(\hat\Sigma-\tilde{\Sigma}) \tilde T_j\|_{\infty}&\leq C\sqrt{\frac{1}{n}\frac{\tilde\lambda_j}{\tilde g_j}\rr_j(\tilde\Sigma)}\bigvee \frac{1}{n}\rr_j(\tilde\Sigma)\bigvee\frac{\tilde\lambda_j}{\tilde g_j}\sqrt{\frac{t}{n}}.
\end{align*}
From this \eqref{eq:transition:tilde:hat:result:Proj} follows. The proof of \eqref{eq:transition:tilde:hat:result:Ev} follows the same line of arguments. We omit the details.
\end{proof}

\subsection{General inconsistency results}\label{sec:General:inconsistency:results}

In this section, we present some general inconsistency results. We separate the cases of eigenvalues and spectral projectors and the cases $j=1$ and $j\geq 2$.

\begin{proposition}\label{ThmLowBoundFM} For a self-adjoint, positive trace class operator $\Sigma$ on $\mathcal{H}$, consider the  model from Setting \ref{setting:inconsistency:model}. Let $z>0$ and $y>0$ be real numbers such that
\begin{equation}\label{EqFixPoint}
\frac{z}{\sqrt{n}}\sum_{k\geq 1}\frac{\lambda_k}{\lambda_1+y-\lambda_k}= 1,\qquad \frac{z}{\sqrt{n}}\sum_{k>r}\frac{\lambda_k}{\lambda_1-\lambda_k}<1/2.
\end{equation}
Then we have
\begin{align*}
\P(\hat{\lambda}_1 - {\lambda}_1>y) \geq  \frac{1-\Phi(4z)}{2} - C\frac{r}{\sqrt{n}}
\end{align*}
with an absolute constant $C>0$ and $\Phi(x)=(1/\sqrt{2\pi})\int_{-\infty}^xe^{-t^2/2}\,dt$.
\end{proposition}

\begin{proof}
Set  $S=(\lambda_1+y-\Sigma)^{-1/2}(\hat \Sigma-\Sigma)(\lambda_1+y-\Sigma)^{-1/2}$. By \cite[Lemma 3.11]{RW17}, we have the implication $\lambda_1(S)> 1\Rightarrow\hat \lambda_1(\Sigma)> \lambda_1(\Sigma)+y$. Using the inequality $\lambda_1(S)\geq \langle v,Sv\rangle/\langle v,v\rangle$, which holds for all $v\in\mathcal{H}$, the implication $\langle v,Sv\rangle/\langle v,v\rangle> 1\Rightarrow\hat{\lambda}_1-\lambda_1> y$ follows. We consider
\begin{align}\label{eq:choice:v:max:ev}
v=(\lambda_1+y-\Sigma)^{-1/2}F\quad \text{such that }\quad \frac{z}{\sqrt{n}}\|v\|^2=\frac{z}{\sqrt{n}}\sum_{k = 1}^r\frac{\lambda_k}{\lambda_1+y-\lambda_k}> \frac{1}{2},
\end{align}
where the last inequality follows from \eqref{EqFixPoint}, and let
\[
w=(\lambda_1+y-\Sigma)^{-1/2}v=(\lambda_1+y-\Sigma)^{-1}F.
\]
Using $\hat\Sigma-\Sigma=\tilde\Sigma-\Sigma+\hat\Sigma-\tilde\Sigma$ and \eqref{eq:choice:v:max:ev}, we get that
\begin{align*}
    \frac{\langle v,Sv\rangle}{\langle v,v\rangle}=\tilde x \langle v,v\rangle+\frac{\langle w,(\hat\Sigma-\tilde\Sigma) w\rangle}{\langle v,v\rangle}>\frac{\tilde x}{2}\frac{\sqrt{n}}{z}+\frac{\langle w,(\hat\Sigma-\tilde\Sigma) w\rangle}{\langle v,v\rangle}.
\end{align*}
Hence, on the event
\begin{align*}
  \Big\{\frac{\langle w,(\hat\Sigma-\tilde\Sigma) w\rangle}{\langle v,v\rangle}\geq 0\Big\}\cap\Big\{\tilde x\geq \frac{2z}{\sqrt{n}}\Big\}=\Big\{\frac{\langle w,(\hat\Sigma-\tilde\Sigma) w\rangle}{\langle v,v\rangle^2}\geq 0\Big\}\cap\tilde{\mathcal{E}}_{2z},
\end{align*}
we have $\hat\lambda_1-\lambda_1>y$. To this event we apply the Berry-Esseen theorem.

\begin{lemma}\label{LemBE1} For $r\geq 1$, we have
\[
\sup_{z\in\mathbb{R}}\Big|\P\Big(\tilde x\leq \frac{z}{\sqrt{n}}\sqrt{1/2-1/(4r^2)} \Big)-\Phi(z)\Big|\leq C\frac{r}{\sqrt{n}}.
\]
for some absolute constant $C>0$.
\end{lemma}

\begin{proof}[Proof of Lemma \ref{LemBE1}]
Recall that $\tilde x=n^{-1}\sum_{i=1}^nf_i^2-1/(2r)$. We have $\operatorname{Var}(f_i^2)=1/2-1/(4r^2)\geq 1/4$ and $\E f_i^6=r/2$. Hence the claim follows from the Berry-Esseen theorem.
\end{proof}

\begin{lemma}\label{LemBE2} For $r\geq 1$, we have
\[
\sup_{z\in\mathbb{R}}\Big|\P\Big(\Big\{\frac{\langle w,(\hat\Sigma-\tilde\Sigma) w\rangle}{\langle v,v\rangle^2}\geq 0\Big\}\cap\tilde{\mathcal{E}}_{2z}\Big) - \frac{1}{2}\P(\tilde{\mathcal{E}}_{2z})\Big|\leq C\frac{r}{\sqrt{n}}.
\]
\end{lemma}
\begin{proof}[Proof of Lemma \ref{LemBE2}]
Set
\begin{align*}
    A=\frac{\langle w,(\hat\Sigma-\tilde\Sigma) w\rangle}{\langle v,v\rangle^2}=\frac{1}{n}\sum_{i=1}^n\bigg(\frac{\langle\epsilon_i,w\rangle^2}{\langle v,v\rangle^2}-\E\frac{\langle\epsilon_i,w\rangle^2}{\langle v,v\rangle^2}+2f_i\frac{\langle \epsilon_i,w\rangle}{\langle v,v\rangle}\bigg).
\end{align*}
By conditioning on $f_1,\dots,f_n$, we have
\begin{align*}
&\sup_{x \in \R}|\P(A \geq 0, \tilde x \geq 2zn^{-1/2}) - (1/2)\P(\tilde x \geq 2zn^{-1/2})| \leq \E|\P(A \geq 0 | f_i,i\leq n ) - 1/2|.
\end{align*}
The random variable $\langle\epsilon,w\rangle/\langle v,v\rangle$ is Gaussian with expectation zero and variance
\begin{align}
\sigma^2&=\frac{\langle w,(\Sigma-1/(2r)(F\otimes F)) w\rangle}{\langle v,v\rangle^2}\nonumber\\&=\bigg(\sum_{j=1}^r\frac{\lambda_j^2}{(\lambda_1+y-\lambda_j)^2}-\frac{1}{2r}\bigg(\sum_{j= 1}^r\frac{\lambda_j}{\lambda_1+y-\lambda_j}\bigg)^2\bigg)\bigg(\sum_{j= 1}^r\frac{\lambda_j}{\lambda_1+y-\lambda_j}\bigg)^{-2}\geq \frac{1}{2r}\label{EqKeAhn},
\end{align}
as can be seen from applying the Cauchy-Schwarz inequality twice. Moreover, we have
\[
\operatorname{Var}\Big(\frac{\langle\epsilon_i,w\rangle^2}{\langle v,v\rangle^2}-\E\frac{\langle\epsilon_i,w\rangle^2}{\langle v,v\rangle^2}+2f_i\frac{\langle \epsilon_i,w\rangle}{\langle v,v\rangle}\Big|f_i,i\leq n\Big)\geq 2\sigma^4\]
and
\[
\E\Big(\Big|\frac{\langle\epsilon_i,w\rangle^2}{\langle v,v\rangle^2}-\E\frac{\langle\epsilon_i,w\rangle^2}{\langle v,v\rangle^2}+2f_i\frac{\langle \epsilon_i,w\rangle}{\langle v,v\rangle}\Big|^3\Big|f_i,i\leq n\Big)\leq C(\sigma^6+|f_i|^3\sigma^3).
\]
Thus the Berry-Esseen theorem gives
\begin{align*}
|\P(A_2 \geq 0 | f_i,i\leq n ) - 1/2| \leq C\sum_{i=1}^n\frac{\sigma^6+|f_i|^3\sigma^3}{n^{3/2}\sigma^6}.
\end{align*}
Taking expectation with respect to the $f_i$ and using $\E f_i^3=1/\sqrt{r}$ and \eqref{EqKeAhn}, we conclude that
\[
\E|\P(A \geq 0 | f_i,i\leq n ) - 1/2|\leq C\frac{r}{\sqrt{n}}
\]
which completes the proof.
\end{proof}

Applying Lemmas \ref{LemBE1} and \ref{LemBE2}, we conclude that
\begin{align*}
\P(\hat\lambda_1-\lambda_1>y)&\geq \P\Big(\Big\{\frac{\langle w,(\hat\Sigma-\tilde\Sigma) w\rangle}{\langle v,v\rangle^2}\geq 0\Big\}\cap\tilde{\mathcal{E}}_{2z}\Big)\\
& \geq \frac{1}{2}\P(\tilde{\mathcal{E}}_{2z})-C\frac{r}{\sqrt{n}}\geq \frac{1-\Phi(4z)}{2}-C\frac{r}{\sqrt{n}},
\end{align*}
and the claim follows from taking complements.
\end{proof}

\begin{proposition}
\label{thm:inconsistency:j=1}
For a self-adjoint, positive trace class operator $\Sigma$ on $\mathcal{H}$, consider the  model from Setting \ref{setting:inconsistency:model}. There are absolute constants $c,C>0$ such that the following holds. Suppose that $z,t>0$, $L>0$ and $0<\delta\leq 1$ are positive real numbers such that
\begin{itemize}
    \item[(i)] $\frac{z}{\sqrt{n}}\sum_{k=2}^r\frac{\lambda_k}{\lambda_1-\lambda_k}\geq L+1$.
    \item[(ii)] $\max\Big(\frac{1}{L},\sqrt{\frac{1}{\sqrt{n}}\Big(1+\frac{1}{L}\frac{\lambda_1}{\lambda_1-\lambda_2}\Big)\Big(\frac{2}{z}+\frac{1}{\sqrt{n}}+\frac{1}{\sqrt{n}}\sum_{k>r}\frac{\lambda_k}{\lambda_1-\lambda_k}\Big)},\Big(1+\frac{1}{L}\frac{\lambda_1}{\lambda_1-\lambda_2}\Big)\sqrt{\frac{t}{n}}\Big)\leq c\delta$.
\end{itemize}
Then we have
\begin{align*}
    \P\Big(\|\hat P_1-P_1\|_2^2\geq \frac{\lambda_2}{\lambda_1+\lambda_2}(2-\delta)\Big)\geq 1-\Phi(2z)-e^{-t}-C\frac{r}{\sqrt{n}}.
\end{align*}
\end{proposition}

\begin{lemma}\label{lemma:tilde:eigenvalue:bounds}
Under (i) of Proposition \ref{thm:inconsistency:j=1}, we have on the event $\tilde{\mathcal{E}}_{z}$,
\begin{align*}
    \frac{\tilde\lambda_1}{\tilde\lambda_1-\tilde\lambda_2}&\leq 1+\frac{1}{L}\frac{\lambda_1}{\lambda_1-\lambda_2},\qquad
    \frac{1}{\sqrt{n}}\rr_1(\tilde\Sigma) \leq \frac{2}{z}+\frac{1}{\sqrt{n}}+\frac{1}{\sqrt{n}}\sum_{k>r}\frac{\lambda_k}{\lambda_1-\lambda_k}.
\end{align*}
\end{lemma}

\begin{proof}[Proof of Lemma \ref{lemma:tilde:eigenvalue:bounds}]
By \eqref{eq:interlacing:rank-one} and the fact that $x\mapsto x/(\tilde\lambda_1-x)$ is increasing for $x<\tilde\lambda_1$, we get
\begin{align*}
   \sum_{k=2}^r\frac{\tilde\lambda_k}{\tilde\lambda_1-\tilde\lambda_k}+\frac{\tilde\lambda_1}{\tilde\lambda_1-\tilde\lambda_2}&\leq 1+2\sum_{k=1}^r\frac{\lambda_k}{\tilde\lambda_1-\lambda_k},\qquad
    \frac{\tilde\lambda_1}{\tilde\lambda_1-\tilde\lambda_2}\leq 1+\frac{\lambda_1}{\tilde\lambda_1-\lambda_1}.
\end{align*}
Hence, the first claim follows from inserting \eqref{eq:eigenvalue:inconsistency:perfect:model:1} into the second inequality. On the event $\tilde{\mathcal{E}}_{z}$, we have $\tilde x \geq zn^{-1/2}$.
Hence, by \eqref{eq:secular:equation}, we get
\begin{align*}
    &\frac{1}{\sqrt{n}}\Big(\sum_{k=2}^r\frac{\tilde\lambda_k}{\tilde\lambda_1-\tilde\lambda_k}+\frac{\tilde\lambda_1}{\tilde\lambda_1-\tilde\lambda_2}\Big)\leq \frac{2\tilde x}{z}\sum_{k=1}^r\frac{\lambda_k}{\tilde\lambda_1-\lambda_k}+\frac{1}{\sqrt{n}}=\frac{2}{z}+\frac{1}{\sqrt{n}}.
\end{align*}
Moreover, using that $\tilde\lambda_k=\lambda_k$ for $k>r$ and \eqref{eq:interlacing:rank-one}, we have
\begin{align}\label{eq:tilde:eigenvalue:bounds:j=1}
    \frac{1}{\sqrt{n}}\sum_{k>r}\frac{\tilde\lambda_k}{\tilde\lambda_1-\tilde\lambda_k}=\frac{1}{\sqrt{n}}\sum_{k>r}\frac{\lambda_k}{\tilde\lambda_1-\lambda_k}\leq \frac{1}{\sqrt{n}}\sum_{k>r}\frac{\lambda_k}{\lambda_1-\lambda_k}
\end{align}
and the second claim follows.
\end{proof}

\begin{proof}[Proof of Proposition~\ref{thm:inconsistency:j=1}]
Assume that the event $\tilde{\mathcal{E}}_{z}\cap \tilde{\mathcal{E}}$ holds. By Lemma \ref{lemma:eigenvalue:1:inconsistency:perfect:model} and (i) with $L\geq 4/\delta$, we have
\begin{align*}
\|\tilde P_1- P_1\|_2\geq \frac{\lambda_1}{\lambda_1+\lambda_2}(2-\delta).    \end{align*}
On the other hand, combining Corollary \ref{Cor:relative:Weyl:DK:tilde} with Lemmas~\ref{lemma:transition:tilde:hat} and \ref{lemma:tilde:eigenvalue:bounds}, we get that, conditional on $f_1,\dots,f_n$ such that the event $\tilde{\mathcal{E}}_{z}\cap \tilde{\mathcal{E}}$ holds, with probability at least $1-e^{-t}$,
\begin{align*}
\|\hat P_1-\tilde P_1\|_2\leq \frac{\delta}{2}\leq  \frac{\lambda_1}{\lambda_1+\lambda_2}\delta,
\end{align*}
provided that the left-hand side in (ii) is bounded by $\min(c_1C_1^{-1}\delta/2, c_2)$. This is satisfied if (ii) holds with $c$ small enough. Combining these bounds, we conclude that
\begin{align*}
    &\P\Big(\|\hat P_1- P_1\|_2>\frac{2\lambda_1}{\lambda_1+\lambda_2}(1-\delta)\Big)\\
    &\geq \P\Big(\Big\{\|\tilde P_1- P_1\|_2>\frac{2\lambda_1}{\lambda_1+\lambda_2}(1-\delta)+\|\tilde P_1- \hat P_1\|_2\Big\}\cap \tilde{\mathcal{E}}_{z}\cap \tilde{\mathcal{E}}\Big)\\
    & \geq \P\Big(\Big\{\|\tilde P_1- P_1\|_2>\frac{\lambda_1}{\lambda_1+\lambda_2}(2-\delta)\Big\}\cap \tilde{\mathcal{E}}_{z}\cap \tilde{\mathcal{E}}\Big)-e^{-t}\geq1-e^{-t}-\P(\tilde{\mathcal{E}}_{z}^c)-\P(\tilde{\mathcal{E}}^c).
\end{align*}
By Lemma \ref{LemBE1}, we have $\P(\tilde{\mathcal{E}}_{z}^c)\leq \Phi(2z)+Cr/\sqrt{n}$. The claim now follows from the following lemma.

\begin{lemma}\label{lem:chernoff}
We have
\begin{align*}
    \P(\tilde{\mathcal{E}}^c)=\P\Big(\frac{1}{n}\sum_{i=1}^nf_i^2>\frac{1}{r}\Big)\leq e^{-\frac{3n}{16r^2}}.
\end{align*}
\end{lemma}

It remains to proof Lemma \ref{lem:chernoff}. By construction, the random variables $f_i^2/r$, $1\leq i\leq n$, are independent Bernoulli random variables with parameter $1/(2r^2)$.
An application of Bernstein's inequality (see e.g.~\cite{MR3185193}) yields
\begin{align}\label{eq:chernoff}
\P\Big(\frac{1}{n}\sum_{i = 1}^n f_i^2 - \frac{1}{2r} \geq t\Big) \leq \exp\Big(-\frac{nt^2}{1+\frac{2rt}{3}}\Big).
\end{align}
Setting $t=1/(2r)$, the claim follows.
\end{proof}

\begin{proposition}\label{thm:inconsistency:j>=2:Ev}
For a self-adjoint, positive trace class operator $\Sigma$ on $\mathcal{H}$, consider the  model from Setting \ref{setting:inconsistency:model}. There are absolute constants $c,C>0$ such that the following holds. Let $2\leq j<r$ and suppose that $0< z_1< z_2\leq \sqrt{n}/(2r)$, $t>0$, $0<\delta<1$ and $L>\ell>0$ are real numbers such that
\begin{itemize}
    \item[(i)] $\frac{z_1}{\sqrt{n}}\sum\limits_{k=1,k\neq j-1}^r\frac{\lambda_k}{\lambda_{j-1}-\lambda_k}\geq L+1\quad\text{and}\quad\frac{z_2}{\sqrt{n}}\frac{\lambda_{j-1}}{\lambda_{j-1}-\lambda_j}\leq \ell$.
    \item[(ii)] $\max(\frac{\ell}{L},\sqrt{\frac{1}{n}\frac{\lambda_{j-1}}{\lambda_{j-1}-\lambda_j}\sum_{k\leq j-1}\frac{\lambda_k}{\lambda_k-\lambda_j}},\frac{\lambda_{j-1}}{\lambda_{j-1}-\lambda_j}\sqrt{\frac{t}{n}})\leq c\delta$.
\end{itemize}
Then we have
\begin{align*}
     \P\Big(\frac{\lambda_{j-1}-\hat\lambda_j}{\lambda_{j-1}-\lambda_{j}}\leq \delta\Big)\geq \Phi(\sqrt{2}z_2)-\Phi(2z_1)-e^{-t}-C\frac{r}{\sqrt{n}}.
\end{align*}
\end{proposition}

\begin{proof}
Assume that the event $\tilde{\mathcal{E}}_{z_1,z_2}$ holds, implying that $\tilde{\mathcal{E}}$ also holds since $z_2\leq \sqrt{n}/(2r)$. By Lemma \ref{lemma:eigenvalue:j:inconsistency:perfect:model} and (i) with $\ell/L\leq \delta/2$, we have
\begin{align}\label{proof:thm:inconsistency:j>=1:tilde:Ev}
    \frac{\lambda_{j-1}-\tilde\lambda_j}{\lambda_{j-1}-\lambda_j}\leq \delta/2.
\end{align}
On the other hand, for $y=\delta(\lambda_{j-1}-\lambda_j)/2$, we get that
\begin{align*}
    \sum_{k\leq j-1}\frac{\tilde\lambda_k}{\tilde\lambda_k+y-\tilde\lambda_j}\leq \sum_{k\leq j-1}\frac{\lambda_k}{\lambda_k+y-\lambda_{j-1}}\leq \frac{2}{\delta}\sum_{k\leq j-1}\frac{\lambda_k}{\lambda_k-\lambda_j},
\end{align*}
as can be seen from \eqref{eq:interlacing:rank-one}, and thus
\begin{align*}
    \frac{\tilde\lambda_j}{y}\leq \frac{2}{\delta}\frac{\lambda_{j-1}}{\lambda_{j-1}-\lambda_j},\qquad\sum_{k\leq j}\frac{\tilde\lambda_k}{\tilde\lambda_k+y-\tilde\lambda_j}\leq \frac{4}{\delta}\sum_{k\leq j-1}\frac{\lambda_{k}}{\lambda_{k}-\lambda_j}.
\end{align*}
Combining Corollary \ref{Cor:relative:Weyl:DK:tilde} (applied with $y=\delta(\lambda_{j-1}-\lambda_j)/2$) with Lemma \ref{lemma:transition:tilde:hat} and the above, we get 
\begin{align}\label{eq:transition:tilde:hat:j>=2:Ev}
    \P\Big(\Big\{\frac{\hat\lambda_j-\tilde\lambda_j}{\lambda_{j-1}-\lambda_j}\geq -\frac{\delta}{2}\Big\}\cap \tilde{\mathcal{E}}_{z_1,z_2}\Big)\geq 1-e^{-t},
\end{align}
provided that (ii) holds with $c=c_1/\sqrt{8}$, where $c_1$ is the constant from \eqref{eq:transition:tilde:hat:result:Ev}. From \eqref{proof:thm:inconsistency:j>=1:tilde:Ev} and \eqref{eq:transition:tilde:hat:j>=2:Ev}, we conclude that
\begin{align*}
    \P\Big(\frac{\lambda_{j-1}-\hat\lambda_j}{\lambda_{j-1}-\lambda_j}\leq \delta\Big)
    &\geq \P\Big(\Big\{\frac{\lambda_{j-1}-\tilde\lambda_j}{\lambda_{j-1}-\lambda_j}+\frac{\tilde\lambda_{j}-\hat\lambda_j}{\lambda_{j-1}-\lambda_j}\leq \delta\Big\}\cap \tilde{\mathcal{E}}_{z_1,z_2}\Big)\\
    &\geq \P\Big(\Big\{\frac{\lambda_{j-1}-\tilde\lambda_j}{\lambda_{j-1}-\lambda_j}\leq \frac{\delta}{2}\Big\}\cap \tilde{\mathcal{E}}_{z_1,z_2}\Big)-e^{-t}=1-e^{-t}-\P(\tilde{\mathcal{E}}_{z_1,z_2}^c).
\end{align*}
By Lemma \ref{LemBE1} applied twice, we have
\begin{align}\label{eq:deviation:calE:z1:z2}
    \P(\tilde{\mathcal{E}}_{z_1,z_2}^c)
    \leq 1-\Phi(\sqrt{2}z_2)+\Phi(2z_1)+C\frac{r}{\sqrt{n}}
\end{align}
and the claim follows.
\end{proof}

\begin{proposition}\label{thm:inconsistency:j>=2}
For a self-adjoint, positive trace class operator $\Sigma$ on $\mathcal{H}$, consider the  model from Setting \ref{setting:inconsistency:model}. There are absolute constants $c,C>0$ such that the following holds. Let $3\leq j< r$ and suppose that $0< z_1< z_2\leq \sqrt{n}/(2r)$, $t>0$, $0<\delta<1$ and $L>\ell>0$ are real numbers such that
\begin{itemize}
    \item[(i)] $\frac{z_1}{\sqrt{n}}\sum_{k=1,k\neq j^*}^r\limits\frac{\lambda_k}{\lambda_{j^*}-\lambda_k}\geq L+1$ for $j^*\in\{j-1,j-2\}\quad\text{and}\quad\frac{z_2}{\sqrt{n}}\frac{\lambda_{j-1}}{g_{j-1}}\leq \ell$,
    \item[(ii)] $\max(\frac{\ell}{L},\sqrt{\frac{1}{n}\frac{\lambda_{j-1}}{g_{j-1}}\rr_{j-1}(\Sigma)},\frac{\lambda_{j-1}}{g_{j-1}}\sqrt{\frac{t}{n}})\leq c\delta$.
\end{itemize}
Then we have
\begin{align*}
     \P\big(\|\hat P_j-P_j\|_2^2\geq 2(1-\delta)\big)\geq \Phi(\sqrt{2}z_2)-\Phi(2z_1)-e^{-t}-C\frac{r}{\sqrt{n}}.
\end{align*}
\end{proposition}

\begin{lemma}\label{lem:compare:tildeSigma:epsSigma:j}
Under (i) of Proposition \ref{thm:inconsistency:j>=2} with $L\geq 4\ell$, we have on the event $\tilde{\mathcal{E}}_{z_1,z_2}$,
\begin{align*}
    \frac{\tilde\lambda_j}{\tilde g_j}&\leq 2\frac{\lambda_{j-1}}{g_{j-1}}
\qquad\text{and}\qquad
    \rr_j(\tilde\Sigma)\leq 3\rr_{j-1}(\Sigma).
    \end{align*}
\end{lemma}
\begin{remark}
Under (i) and (ii) the eigenvalues $\tilde\lambda_{j-1}$ and $\tilde\lambda_{j}$ are pushed towards $\lambda_{j-2}$ and $\lambda_{j-1}$, respectively. This explains why the relative rank at $j-1$ appears in the upper bound.
\end{remark}

\begin{remark}
The case $j=2$ works as well. In this case (i) is only needed for $j^*=j-1$. We omit the details. 
\end{remark}

\begin{proof}[Proof of Lemma \ref{lem:compare:tildeSigma:epsSigma:j}]
Arguing as \eqref{eq:tilde:eigenvalue:bounds:j=1}, it suffices to bound the sum over $k\leq r$, $k\neq j$. First, under (i) of Proposition \ref{thm:inconsistency:j>=2}, the assumptions of Lemma \ref{lemma:eigenvalue:j:inconsistency:perfect:model} are satisfied (with $L$ and $\ell$, because $z_1n^{-1/2}\leq \tilde x\leq z_2n^{-1/2}$ on the event $\tilde{\mathcal{E}}_{z_1,z_2}$). In particular, we have $\tilde\lambda_j=\lambda_{j-1}-A(\lambda_{j-1}-\lambda_j)$ with $A\in(0,1/2]$, provided that $\ell/L\leq  1/2$. Using this, again the monotonicity of the map $x\mapsto x/(\tilde\lambda_k-x)$ and \eqref{eq:interlacing:rank-one}, we get
\begin{align*}
    \sum_{k>j}\frac{\tilde\lambda_k}{\tilde\lambda_j-\tilde\lambda_k}&\leq \sum_{k>j}\frac{\lambda_{k-1}}{\tilde\lambda_j-\lambda_{k-1}}=\sum_{k>j}\frac{\lambda_{k-1}}{\lambda_{j-1}-\lambda_{k-1}-A(\lambda_{j-1}-\lambda_j)} \\
    &\leq \frac{1}{1-A}\sum_{k>j}\frac{\lambda_{k-1}}{\lambda_{j-1}-\lambda_{k-1}}\leq 2\sum_{k>j}\frac{\lambda_{k-1}}{\lambda_{j-1}-\lambda_{k-1}},
\end{align*}
\begin{align*}
    \frac{\tilde\lambda_j}{\tilde\lambda_j-\tilde\lambda_{j+1}}=1+ \frac{\tilde\lambda_{j+1}}{\tilde\lambda_j-\tilde\lambda_{j+1}}\leq 1+2\frac{\lambda_j}{\lambda_{j-1}-\lambda_{j}}\leq 2\frac{\lambda_{j-1}}{\lambda_{j-1}-\lambda_{j}},
\end{align*}
and
\begin{align*}
    \sum_{k<j-1}\frac{\tilde\lambda_k}{\tilde\lambda_k-\tilde\lambda_j}\leq \sum_{k<j-1}\frac{\lambda_{k}}{\lambda_{k}-\lambda_{j-1}}.
\end{align*}
Moreover, the assumptions of Lemma \ref{lemma:eigenvalue:j:inconsistency:perfect:model} are also satisfied with $j$ replaced by $j-1$ and with constants $L$ and $2\ell$, as can be seen from combining (i) with the inequality
\begin{align*}
    \tilde x\frac{\lambda_{j-2}}{\lambda_{j-2}-\lambda_{j-1}}=\tilde x\Big(1+\frac{\lambda_{j-1}}{\lambda_{j-2}-\lambda_{j-1}}\Big)\leq 2\tilde{x}\frac{\lambda_{j-1}}{g_{j-1}}\leq 2z_2\frac{\lambda_{j-1}}{g_{j-1}}.
\end{align*}
In particular, we have $\tilde\lambda_{j-1}=\lambda_{j-2}-A(\lambda_{j-2}-\lambda_{j-1})$ with $A\in(0,1/2]$, provided that $L\geq 4\ell$, and we obtain
\begin{align*}
    \frac{\tilde\lambda_{j}}{\tilde\lambda_{j-1}-\tilde\lambda_j}\leq \frac{\lambda_{j-1}}{\tilde\lambda_{j-1}-\lambda_{j-1}}=\frac{1}{1-A}\frac{\lambda_{j-1}}{\lambda_{j-2}-\lambda_{j-1}}\leq 2\frac{\lambda_{j-1}}{\lambda_{j-2}-\lambda_{j-1}},
\end{align*}
and also
\begin{align*}
    \frac{\tilde\lambda_{j-1}}{\tilde\lambda_{j-1}-\tilde\lambda_j}=1+\frac{\tilde\lambda_{j}}{\tilde\lambda_{j-1}-\tilde\lambda_j}\leq 1+2\frac{\lambda_{j-1}}{\lambda_{j-2}-\lambda_{j-1}}\leq 2\frac{\lambda_{j-2}}{\lambda_{j-2}-\lambda_{j-1}}.
\end{align*}
Combining these estimates, we conclude that
\begin{align*}
    \frac{\tilde\lambda_j}{\tilde g_j}\leq 2\frac{\lambda_{j-1}}{g_{j-1}}\quad\text{and}\quad
   \sum_{k\neq j}\frac{\tilde\lambda_k}{|\tilde\lambda_j-\tilde\lambda_k|} \leq 3\sum_{k\neq j-1}\frac{\lambda_k}{|\lambda_{j-1}-\lambda_k|}.
\end{align*}
From this, the claim follows.
\end{proof}

\begin{proof}[Proof of Proposition \ref{thm:inconsistency:j>=2}]
Assume that the event $\tilde{\mathcal{E}}_{z_1,z_2}$ holds. By Lemma \ref{lemma:eigenvalue:j:inconsistency:perfect:model} and (i) with $\ell^2/L^2\leq \delta/8$, we have $\|\tilde P_j- P_j\|_2\geq 2-\delta$. On the other hand, combining Corollary \ref{Cor:relative:Weyl:DK:tilde}, Lemma \ref{lemma:transition:tilde:hat} and Lemma~\ref{lem:compare:tildeSigma:epsSigma:j}, we get that $\P(\{\|\hat P_j-\tilde P_j\|_2\leq \delta\}\cap \tilde{\mathcal{E}}_{z_1,z_2})\geq 1-e^{-t}$, provided the left-hand side in (ii) is bounded by $\min(c_1C_1^{-1}\delta,c_2)$, where $c_1,c_2$ are the constants from Lemma~\ref{lemma:transition:tilde:hat} and $C_1$ is the constant from Corollary \ref{Cor:relative:Weyl:DK:tilde}. In particular, this is satisfied if (ii) holds with $c$ small enough. Combining these bounds, we conclude that
\begin{align*}
    \P\Big(\|\hat P_j- P_j\|_2>2-2\delta\Big)
    &\geq \P\Big(\Big\{\|\tilde P_j- P_j\|_2>2-2\delta+\|\tilde P_j- \hat P_j\|_2\Big\}\cap \tilde{\mathcal{E}}_{z_1,z_2}\Big)\\
    & \geq \P\Big(\Big\{\|\tilde P_1- P_1\|_2>2-\delta\Big\}\cap \tilde{\mathcal{E}}_{z_1,z_2}\Big)-e^{-t}= 1-e^{-t}-\P(\tilde{\mathcal{E}}_{z_1,z_2}^c).
\end{align*}
Inserting \eqref{eq:deviation:calE:z1:z2}, the claim follows.
\end{proof}

\subsection{Proofs of the inconsistency results from Section~\ref{SecApp}}

\begin{proof}[Proof of Theorem \ref{thm:inconsistency:j=1:CLT}]
The implication from (a) to (b) is immediate from Theorem \ref{CorRootNCons}, since the models from Setting \ref{setting:inconsistency:model} satisfy \eqref{EqMomentAssIID} with $p=4$ and a constant $C_\eta$ which does not depend on $n$. Moreover, the implication from (b) to (c) follows from standard properties of the stochastic Landau symbols.

Now suppose that (c) holds and let $\epsilon >0$. Then we have $\P(|\hat\lambda_1^{(n)}-\lambda_1^{(n)}|> \epsilon(\lambda_1^{(n)}-\lambda_2^{(n)}))\rightarrow 0$ as $n\rightarrow\infty$. Now, for $y=y_n=\epsilon(\lambda_1^{(n)}-\lambda_2^{(n)})$ let $z=z_n>0$ be the unique solution of Equation \eqref{EqFixPoint}. From Proposition \ref{ThmLowBoundFM} it follows that $z_n\rightarrow \infty$ as $n\rightarrow \infty$. Thus
\begin{align*}
&\frac{1}{1+\epsilon}\frac{1}{\sqrt{n}}\sum_{k>  1}\frac{\lambda_k^{(n)}}{\lambda_1^{(n)}-\lambda_k^{(n)}}\leq
\frac{1}{\sqrt{n}}\sum_{k\geq 1}\frac{\lambda_k^{(n)}}{\lambda_1^{(n)}+\epsilon(\lambda_1^{(n)}-\lambda_2^{(n)})-\lambda_k^{(n)}}\rightarrow 0\quad\text{as}\quad n\rightarrow \infty.
\end{align*}
Moreover, using that \eqref{EqCCondAsymp} is equivalent to
\begin{equation*}
\frac{1}{\sqrt{n}}\sum_{k> 1}\frac{\lambda_k^{(n)}}{\lambda_1^{(n)}-\lambda_k^{(n)}}\rightarrow 0\quad\text{as}\quad n\rightarrow \infty,
\end{equation*}
as can be seen from the inequality
\begin{align}
\frac{\lambda_1^{(n)}}{\lambda_1^{(n)}-\lambda_{2}^{(n)}}
&=1+\frac{\lambda_{2}^{(n)}}{\lambda_1^{(n)}-\lambda_{2}^{(n)}}\leq 1+\sum_{k> 1}\frac{\lambda_k^{(n)}}{\lambda_1^{(n)}-\lambda_k^{(n)}},\label{EqRREquiv}
\end{align}
the claim follows.
\end{proof}

\begin{proof}[Proof of Theorem \ref{thm:inconsistency:j=1:projectors}]
The implication from (a) to (b) is immediate from Theorem \ref{ThmTightProjector}, since the models from Setting \ref{setting:inconsistency:model} satisfy \eqref{EqMomentAssIID} with $p=4$ and a constant $C_\eta$ which does not depend on $n$.

Now, suppose that (a) does not hold. By restricting to a subsequence, we can assume that there is some $\epsilon>0$ such that
\begin{align*}
    \sqrt{\frac{1}{n}}\sum_{k>1}\frac{\lambda_k^{(n)}}{\lambda_1^{(n)}-\lambda_k^{(n)}}\geq \epsilon\qquad\forall n\geq 1,
\end{align*}
where we also used \eqref{EqRREquiv}. We choose $z=z_n=(c+2)c^{-1}\epsilon^{-1}$ with constant $c$ from (ii) in Proposition \ref{thm:inconsistency:j=1} such that (i) in Proposition \ref{thm:inconsistency:j=1} holds with $L=2/c$. Moreover, we choose $t=t_n=Cz^2$ with $C$ large enough such that  $e^{-t}\leq (1-\Phi(2z))/2$. With these choices (ii) of Proposition \ref{thm:inconsistency:j=1} is satisfied for $\delta=1/2$ and for all $n$ large enough, as can be seen by using that $n^{-1/2}\lambda_1^{(n)}/(\lambda_1^{(n)}-\lambda_2^{(n)})\rightarrow 0$ as $n\rightarrow \infty$. Hence, using also that $r_n/\sqrt{n}\rightarrow 0$, Proposition \ref{thm:inconsistency:j=1} yields
\begin{align*}
   \liminf_{n\rightarrow\infty}\P\Big(\|\hat P_1^{(n)}-P_1^{(n)}\|_2^2\geq \frac{3}{2}\frac{\lambda_2^{(n)}}{\lambda_1^{(n)}+\lambda_2^{(n)}}\Big)>0.
\end{align*}
From (i) of Proposition \ref{thm:inconsistency:j=1}, it follows that $1\leq r_nn^{-1/2}\lambda_1^{(n)}/(\lambda_1^{(n)}-\lambda_2^{(n)})$, meaning that $\lambda_2^{(n)}\geq (1-r_nn^{-1/2})\lambda_1^{(n)}$. Hence, (b) does not hold.
\end{proof}

\begin{proof}[Proof of Theorem \ref{thm:inconsistency:j>=2:eigenvalue}]
By restricting to a subsequence, the assumptions in \eqref{eq:thm:inconsistency:j>=2:Ev} imply that there is some $\epsilon>0$ such that
\begin{align*}
    \frac{1}{\sqrt{n}}\sum_{k=1:k\neq j-1}^{r_n}\frac{\lambda_k^{(n)}}{\lambda_{j-1}^{(n)}-\lambda_k^{(n)}}\geq \epsilon\qquad\forall n\geq 1.
\end{align*}
For $\delta>0$, we choose $z_1=z_{1,n}=2\epsilon^{-1}$ such that (i) in Proposition \ref{thm:inconsistency:j>=2:Ev} holds with $L=1$. Moreover we choose $z_2=z_{2,n}=Cz_1$ and  $t=t_n=Cz_1^2$ with $C$ large enough such that $\Phi(2z_{1})\leq \Phi(\sqrt{2}z_{2})/2$ and $e^{-t}\leq \Phi(\sqrt{2}z_{2})/4$. For these choices (i) of Proposition \ref{thm:inconsistency:j>=2:Ev} hold for all $n$ large enough, as can be seen from inserting \eqref{eq:thm:inconsistency:j>=2:Ev}. Hence Proposition \ref{thm:inconsistency:j>=2:Ev} yields
\begin{align}\label{eq:corollary:inconsistency:j>=2:Ev}
   \lim_{\delta\rightarrow 0} \liminf_{n\rightarrow\infty}\P\Big(\frac{\lambda_{j-1}^{(n)}-\hat\lambda_j^{(n)}}{\lambda_{j-1}^{(n)}-\lambda_j^{(n)}}\leq\delta\Big)>0.
\end{align}
Inserting
\begin{align*}
    |\hat\lambda_j^{(n)}-\lambda_j^{(n)}|/g_j^{(n)}\geq \frac{\hat\lambda_j^{(n)}-\lambda_j^{(n)}}{\lambda_{j-1}^{(n)}-\lambda_j^{(n)}}= 1-\frac{\lambda_{j-1}^{(n)}-\hat\lambda_j^{(n)}}{\lambda_{j-1}^{(n)}-\lambda_j^{(n)}},
\end{align*}
the claim follows.
\end{proof}

\begin{proof}[Proof of Theorem \ref{thm:inconsistency:j>=2:projector}]
By restricting to a subsequence, the first condition in \eqref{eq:thm:inconsistency:j>=2:Proj} implies that there is some $\epsilon>0$ such that for $j^*\in\{j-1,j-2\}$,
\begin{align*}
    \frac{1}{\sqrt{n}}\sum_{k=1:k\neq j^*}^{r_n}\frac{\lambda_k^{(n)}}{\lambda_{j^*}^{(n)}-\lambda_k^{(n)}}\geq \epsilon\qquad\forall n\geq 1.
\end{align*}
For $\delta>0$, we choose $z_1=z_{1,n}=2\epsilon^{-1}$ such that (i) in Proposition \ref{thm:inconsistency:j>=2} holds with $L=1$. Moreover we choose $z_2=z_{2,n}=Cz_1$ and  $t=t_n=Cz_1^2$ with $C$ large enough such that $\Phi(2z_{1})\leq \Phi(\sqrt{2}z_{2})/2$ and $e^{-t}\leq \Phi(\sqrt{2}z_{2})/4$. For these choices  (i) of Proposition~\ref{thm:inconsistency:j>=2} hold for all $n$ large enough, as can be seen from inserting  \eqref{eq:thm:inconsistency:j>=2:Proj}. Hence, Proposition \ref{thm:inconsistency:j>=2}  yields the claim.
\end{proof}

\begin{proof}[Proof of Theorem \ref{thm:inconsistency:j=1:concentration}]
First, assume that (a) does not hold. Then there is a constant $C>1$ such that
\begin{align*}
    \frac{1}{C}\sqrt{\frac{\log n}{n}}\rr_1(\Sigma^{(n)})\leq 1,\qquad \forall n\geq 1.
\end{align*}
We apply Lemma \ref{LemEvRC}. For this purpose, write
\begin{align*}
    T^{(n)}=T^{(n)}_{\geq 1}(y_n)=\sum_{k\geq 1}\frac{1}{\sqrt{\lambda_1^{(n)}+y_n-\lambda_k^{(n)}}}P_k^{(n)},\qquad y_n=\sqrt{\frac{\log n}{n}}\lambda_1^{(n)}.
\end{align*}
Using the notation from Section \ref{section:proofs:inconsistency} with additional superscript $^{(n)}$, we have
\begin{align*}
    \|T^{(n)}(\hat\Sigma^{(n)}-\Sigma^{(n)})T^{(n)}\|_\infty\leq \|T^{(n)}(\tilde\Sigma^{(n)}-\Sigma^{(n)})T^{(n)}\|_\infty+\|T^{(n)}(\hat\Sigma^{(n)}-\tilde\Sigma^{(n)})T^{(n)}\|_\infty.
\end{align*}
First, since $\tilde\Sigma^{(n)}-\Sigma^{(n)}=\tilde x_n\cdot F^{(n)}\otimes F^{(n)}$ is rank-one, we have
\begin{align*}
    \|T^{(n)}(\tilde\Sigma^{(n)}-\Sigma^{(n)})T^{(n)}\|_\infty&=\tilde x_n\sum_{k\geq 1}\frac{\lambda_k^{(n)}}{\lambda_1^{(n)}+y_n-\lambda_k^{(n)}}\leq \frac{\tilde x_n \lambda_1^{(n)}}{y_n}+\tilde x_n \rr_1(\Sigma^{(n)}),
\end{align*}
and the right-hand side is smaller than $1/2$ if
\begin{align*}
    \tilde x_n\leq \frac{1}{4C}\sqrt{\frac{\log n}{n}}.
\end{align*}
By \eqref{eq:chernoff}, the latter holds with probability at least $1-n^{-c}$ with constant $c>0$ depending only on $C$. Second, following the arguments in Lemma \ref{lemma:transition:tilde:hat}, we have, with probability at least $1-e^{-t_n}$, $t_n\geq 1$,
\begin{align*}
    \|T^{(n)}(\hat\Sigma^{(n)}-\tilde\Sigma^{(n)})T^{(n)}\|_\infty\leq C_1\max\Big(\sqrt{\frac{1}{n}\frac{\lambda_1^{(n)}}{y_n}\sum_{k\geq 1}\frac{\lambda_k^{(n)}}{\lambda_1^{(n)}+y_n-\lambda_k^{(n)}}},\frac{\lambda_1^{(n)}}{y_n}\sqrt{\frac{t_n}{n}}\Big).
\end{align*}
The first term in the maximum is smaller than
\begin{align*}
    C_1\Big(\frac{1}{\log n}+\frac{1}{\sqrt{n\log n}}\rr_1(\Sigma^{(n)})\Big).
\end{align*}
Letting $n\rightarrow \infty$, this term tends to zero by assumption. Hence, for the choice
\begin{align*}
    t_n=\frac{1}{C_1^2}\frac{\log n}{4},
\end{align*}
the maximum is smaller than $1/2$ for all $n$ sufficiently large. By Lemma \ref{LemEvRC}, we conclude that
\begin{align*}
    \P(\sqrt{n}(\hat\lambda_1^{(n)}-\lambda_1^{(n)})/\lambda_1^{(n)}\geq  \sqrt{\log n})\leq n^{-c}+n^{-\frac{1}{4C_1^2}}
\end{align*}
for all sufficiently large $n$. Taking (minus of) the logarithm, we conclude that (b) does not hold.

It remains to prove the implication (a)$\Rightarrow$(b). By restricting to a subsequence, there is a sequence $(a_n)$ of positive real numbers such that $a_n\rightarrow \infty$, $(\log n)/a_n\rightarrow\infty$ and
\begin{align*}
    \sqrt{\frac{\log n}{a_nn}}\sum_{k>1}\frac{\lambda_k^{(n)}}{\lambda_1^{(n)}-\lambda_k^{(n)}}\geq 1,\qquad\forall n\geq 1,
\end{align*}
where we also used \eqref{EqRREquiv}. For $y=y_n=\lambda_1^{(n)}\sqrt{(\log n)/n}$ let $z=z_n>0$ be the unique solution of Equation \eqref{EqFixPoint}. By assumption, there is a constant $C>0$ such that $y_n\leq C(\lambda_1^{(n)}-\lambda_2^{(n)})$ for all $n\geq 1$. Hence, we get that
\begin{align*}
    1&\geq \frac{z_n}{\sqrt{n}}\sum_{k>1}\frac{\lambda_k^{(n)}}{\lambda_1^{(n)}+y_n-\lambda_k^{(n)}}\geq \frac{z_n}{\sqrt{n}}\frac{1}{1+C}\sum_{k>1}\frac{\lambda_k^{(n)}}{\lambda_1^{(n)}-\lambda_k^{(n)}}\geq \frac{z_n}{1+C}\sqrt{\frac{a_n}{\log n}}
\end{align*}
for all $n\geq 1$. We conclude that $z_n^2=o(\log n)$ and the claim follow from Proposition \ref{ThmLowBoundFM} and standard tail bounds for Gaussian random variables (see e.g.~\cite[Eq.~(2.23)]{MR3588285}).
\end{proof}

\section{Additional proofs and results}

\subsection{Bounding bad events by concentration}

For $x > 0$, denote with $\mathcal{E}_x$ the event given in \eqref{Def:event:simple}, and for $r_0 \geq 1$, denote with $\mathcal{E}_{x,r_0}$ the event given in \eqref{Def:event:general}. The following bounds are an immediate consequence of Burkholder's inequality in combination with Markov's inequality.

\begin{proposition}\label{prop_control_events1}
In Setting \ref{iidSetting}, there is a constant $C>0$ depending only on $p$ and $C_\eta$ such that the following holds.
\begin{itemize}
\item[(a)] For all $x > 0$, we have
\begin{align*}
\P\big(\mathcal{E}_{x}^c \big) \leq  Cd^2 (\sqrt{n}x)^{-p/2}.
\end{align*}
\item[(b)] For all $x > 0$ and $r_0\geq 1$, we have
\begin{align*}
\P\big(\mathcal{E}_{x,r_0}^c \big) \leq  Cr_0^2 (\sqrt{n}x)^{-p/2}.
\end{align*}
\end{itemize}
\end{proposition}

\begin{proof}
By Burkholder's inequality and Minkowski's inequality, we have
\[
\E^{4/p}|n\bar{\eta}_{kl}|^{p/2} \leq Cn\E^{4/p}|\eta_k\eta_l|^{p/2}\qquad\forall k,l\geq 1.
\]
Application of the Cauchy-Schwarz inequality then yields
\begin{equation}\label{eq:prop:control:events:1}
\E|\bar{\eta}_{kl}|^{p/2}\leq Cn^{-p/4}\qquad\forall k,l\geq 1.
\end{equation}
Thus, (a) follows from the union bound, Markov's inequality, and \eqref{eq:prop:control:events:1}.
Next, by Minkowski's inequality and \eqref{eq:prop:control:events:1}, we have
\begin{align*}
\E^{4/p}\|Q_s E Q_t \|_2^{p/2} \leq \mu_s\mu_t\sum_{j \in \mathcal{I}_s}\sum_{k \in \mathcal{I}_t} \E^{4/p}\big|\bar{\eta}_{jk}\big|^{p/2} \leq C n^{-1}m_s\mu_sm_t\mu_t,
\end{align*}
and thus
\[
\E\bigg(\frac{\|Q_sEQ_t\|_2}{\sqrt{m_s\mu_sm_t\mu_t}}\bigg)^{p/2}\leq Cn^{-p/4}\qquad\forall s,t<r_0.
\]
Similarly, for all $s<r_0$, we have
\[
\E\bigg(\frac{\|Q_sEQ_{\geq r_0}\|_2}{\sqrt{m_s\mu_s\operatorname{tr}_{\geq r_0}(\Sigma)}}\bigg)^{p/2}\leq Cn^{-p/4},\qquad\E\bigg(\frac{\|Q_{\geq r_0}EQ_{\geq r_0}\|_2}{\operatorname{tr}_{\geq r_0}(\Sigma)}\bigg)^{p/2}\leq Cn^{-p/4}.
\]
Hence, (b) follows from the union bound, Markov's inequality, and the above inequalities.
\end{proof}
If $p > 4$ is small, the bounds in Proposition \ref{prop_control_events1} can be improved.
\begin{proposition}\label{prop_control_events2}
In Setting \ref{iidSetting} with $p>4$, there are constants $C_1,C_2>0$, depending only on $C_\eta$ and $p$, such that the following holds.
\begin{itemize}
\item[(a)] For $x \geq C_1  \sqrt{(\log n)/n}$, we have
\begin{align*}
\P\big(\mathcal{E}_{x} ^c\big) \leq C_2d^2(\log n)^{-p/4} n^{1-p/4}.
\end{align*}
\item[(b)] For  $x \geq C_1  \sqrt{(\log n)/n}$, we have
\begin{align*}
\P\big(\mathcal{E}_{x,r_0} ^c\big) \leq C_2 r_0^2 (\log n)^{-p/4}n^{1-p/4}.
\end{align*}
\end{itemize}
\end{proposition}

\begin{proof}
For $C_1>0$ sufficiently large, an application of the Fuk-Nagaev inequality (cf.~\cite{nagaev1979}) yields
\begin{align}\label{EqFukNag}
\P(|\bar{\eta}_{kl}| > x  ) \leq C \frac{n}{(nx)^{p/2}}\leq (\log n)^{-p/4} n^{1-p/4}.
\end{align}
Combining this with the union bound gives (a). Similarly, (b) follows from Lemma 1 in \cite{JW20}, a Hilbert-space valued version of the Fuk-Nagaev inequality (see also \cite{MR2434306}), and again the union bound.
\end{proof}

\begin{remark}\label{rem:weakdependence}
Both Propositions only appeal to well-known concentration inequalities in the literature for a sequence of independent random variables. By replacing these with analogous results for weakly dependent sequences (e.g.~\cite{dedecker_prieur_2005,merleved_rio_bernstein_2011}) or other dependence structures (e.g. $m$-dependence), the above results can be transferred to a dependent framework.
\end{remark}

\subsection{Additional proofs}

\begin{proof}[Proof of Theorem \ref{CorRootNCons}]
For any $\epsilon>0$, let $R$ be sufficiently large such that $C(2R/3)^{-p/2}\leq \epsilon/2$, where $C$ is the constant from \eqref{eq:prop:control:events:1}, and let $y_n=(R/\sqrt{n})\lambda_j^{(n)}$. Then, by \eqref{EqCCondAsymp}, Condition~\eqref{EqCCond} holds with $x_n=(2/3)R/\sqrt{n}$ for all $n$ large enough. Hence,
\begin{align}\label{Eq:Thm:tight:eigenvalues:1}
x_n\sum_{k\geq j}\frac{\lambda_k^{(n)}}{\lambda_j^{(n)}+y_n-\lambda_k^{(n)}}\leq \frac{2}{3}+x_n\sum_{k>j}\frac{\lambda_k^{(n)}}{\lambda_j^{(n)}-\lambda_k^{(n)}}\leq 1.
\end{align}
Moreover, by Minkowski's inequality and \eqref{eq:prop:control:events:1}, we have
\begin{align*}
&\E^{4/p}\bigg(\sum_{k\geq j}\sum_{l\geq j}\frac{\lambda_k^{(n)}}{\lambda_j^{(n)}+y_n-\lambda_k^{(n)}}\frac{\lambda_l^{(n)}}{\lambda_j^{(n)}+y_n-\lambda_l^{(n)}}(\bar{\eta}_{kl}^{(n)})^2\bigg)^{p/4}\leq \frac{C}{n}\bigg(\sum_{k\geq j}\frac{\lambda_k^{(n)}}{\lambda_j^{(n)}+y_n-\lambda_k^{(n)}}\bigg)^2.
\end{align*}
From this, \eqref{Eq:Thm:tight:eigenvalues:1} and Markov's inequality, we get
\begin{align*}
&\P\bigg(\sum_{k\geq j}\sum_{l\geq j}\frac{\lambda_k^{(n)}}{\lambda_j^{(n)}+y_n-\lambda_k^{(n)}}\frac{\lambda_l^{(n)}}{\lambda_j^{(n)}+y_n-\lambda_l^{(n)}}(\bar{\eta}_{kl}^{(n)})^2> 1\bigg)\\
&\leq\P\bigg(\sum_{k\geq j}\sum_{l\geq j}\frac{\lambda_k^{(n)}}{\lambda_j^{(n)}+y_n-\lambda_k^{(n)}}\frac{\lambda_l^{(n)}}{\lambda_j^{(n)}+y_n-\lambda_l^{(n)}}(\bar{\eta}_{kl}^{(n)})^2> x_n^2\bigg(\sum_{k\geq j}\frac{\lambda_k^{(n)}}{\lambda_j^{(n)}+y_n-\lambda_k^{(n)}}\bigg)^2\bigg)\leq \frac{\epsilon}{2}.
\end{align*}
Combining this with Corollary \ref{PropEvRC} yields $\P(\sqrt{n}(\hat\lambda_j^{(n)}-\lambda_j^{(n)})/\lambda_j^{(n)}> R)\leq \epsilon/2$ for all sufficiently large $n$. The bound $\P(\sqrt{n}(\hat\lambda_j^{(n)}-\lambda_j^{(n)})/\lambda_j^{(n)}< -R)\leq \epsilon/2$ is established in the same manner. The claim now follows from standard arguments.
\end{proof}

\begin{proof}[Proof of Theorem \ref{ThmTightProjector}]
We first establish \eqref{EqSpectralProjCons}. Due to Lemma \ref{lemma:relative:DK}, it suffices to show that $\E \delta_j^{(n)} \to 0$ as $n$ increases, where we recall from \eqref{Def:delta:j} that $\delta_j^{(n)} = \| T_j^{(n)}(\hat\Sigma^{(n)}-\Sigma^{(n)}) T_j^{(n)}\Vert_\infty$. This follows from the inequality $\| T_j^{(n)}(\hat\Sigma^{(n)}-\Sigma^{(n)}) T_j^{(n)}\Vert_\infty \leq \| T_j^{(n)}(\hat\Sigma^{(n)}-\Sigma^{(n)}) T_j^{(n)}\Vert_2$ and similar computations as in the proof of Theorem \ref{CorRootNCons}, together with Assumption \eqref{EqCCondAsymp}.
It remains to prove the second claim. For any $R>0$, we choose $x_n=cR/\sqrt{n}$. Then, by \eqref{EqCCondAsymp}, Condition~\eqref{EqCCond} holds for all $n$ large enough. Hence, Corollary \ref{CorConcentrationMulti} and Proposition \ref{prop_control_events1} yield
\begin{align*}
    \limsup_{n\rightarrow\infty}\P\Big(\sqrt{n} \|\hat P_j^{(n)}-P_j^{(n)}\|_2\Big/\sqrt{\sum_{k\neq j}\frac{\lambda_j^{(n)}\lambda^{(n)}_k}{(\lambda^{(n)}_j-\lambda^{(n)}_k)^2}}>R\Big)\leq Cj_0^2R^{-p/2}.
\end{align*}
Letting $R\rightarrow\infty$, the right-hand side tends to zero, and the second claim follows.
\end{proof}

\begin{proof}[Proof of Theorem \ref{Cor:clt:eigenvalues:anderson}]
We first deal with eigenvalues. By the central limit theorem (Lyapunov), we have
\[
\sqrt{\frac{n}{\operatorname{Var}((\eta_j^{(n)})^2)}}\bar\eta_{jj}^{(n)}\xrightarrow{d}\mathcal{N}(0,1).
\]
By \eqref{EqCCondAsymp}, there exists $a_n \to \infty$ such that $a_n^2n^{-1/2}\rr_j(\Sigma^{(n)})\rightarrow 0$ as $n \to \infty$. Selecting $x_n = a_n n^{-1/2}$, Proposition \ref{prop_control_events1} (ii) (applied with $r_0$ such that $\mu_{r_0}=\lambda_{j_0}$) implies $\P(\mathcal{E}_{x_n,r_0}) \to 1$ as $n\rightarrow \infty$. Due to Theorem \ref{ThmExpEvMult} and the choice of the $a_n$, we then have $\sqrt{n}|\hat\lambda_j^{(n)}-\lambda_j^{(n)}-\bar\eta_{jj}^{(n)}|\xrightarrow{\P}0$ and the claim follows from Slutsky's lemma. 

Next, we deal with angles. We use $r_0$, $a_n$ and $x_n = a_n n^{-1/2}$ from above such that $\P(\mathcal{E}_{x_n,r_0}) \to 1$ and $\sqrt{n}x_n^2\rr_{j}(\Sigma^{(n)})\rightarrow 0$ as $n \to \infty$. First, from \eqref{EqEvEq}, we have
\begin{align*}
\frac{\hat{\lambda}_j^{(n)}-\lambda_k^{(n)}}{\sqrt{\lambda_j^{(n)}\lambda_k^{(n)}}}\langle \hat{u}_j^{(n)},u_k^{(n)}\rangle
&=\bar{\eta}_{kj}^{(n)}\langle\hat{u}_j^{(n)},u_j^{(n)}\rangle+\sum_{l\neq j}\bar{\eta}_{kl}^{(n)}\sqrt{\frac{\lambda_l^{(n)}}{\lambda_j^{(n)}}}\langle\hat{u}_j^{(n)},u_l^{(n)}\rangle.
\end{align*}
Now, Lemma \ref{LemBasIneq} and the Cauchy-Schwarz inequality yield on the event $\mathcal{E}_{x_n,r_0}$ that
\begin{align*}
    \sqrt{n}\sum_{l\neq j}|\bar{\eta}_{kl}^{(n)}|\sqrt{\frac{\lambda_l^{(n)}}{\lambda_j^{(n)}}}|\langle\hat{u}_j^{(n)},u_l^{(n)}\rangle|\leq \sqrt{n}x_n^2\rr_{j}(\Sigma^{(n)})\rightarrow 0\quad\text{as }n\rightarrow \infty.
\end{align*}
Hence, since $\P(\mathcal{E}_{x_n,r_0}) \to 1$ as $n \to \infty$, the left-hand side converges to zero in probability and we obtain that 
\begin{align*}
\frac{\hat{\lambda}_j^{(n)}-\lambda_k^{(n)}}{\sqrt{\lambda_j^{(n)}\lambda_k^{(n)}}} \sqrt{n}\langle \hat{u}_j^{(n)},u_k^{(n)}\rangle
-\sqrt{n}\bar{\eta}_{kj}^{(n)}\langle\hat{u}_j^{(n)},u_j^{(n)}\rangle\xrightarrow{\P} 0.
\end{align*}
By the central limit theorem (Lyapunov), we have
\begin{align*}
\frac{\sqrt{n}}{\sqrt{\operatorname{Var}(\eta_{k}^{(n)}\eta_j^{(n)})}} \bar{\eta}_{kj}^{(n)} \xrightarrow{d} \mathcal{N}\big(0,1\big)
\end{align*}
Moreover, by \eqref{EqGapWeightedCons} and \eqref{EqSpectralProjCons}, we have
\begin{align*}
\frac{\lambda_j^{(n)} - \lambda_k^{(n)}}{\hat\lambda_j^{(n)} - \lambda_k^{(n)}} \xrightarrow{\P} 1
\end{align*}
and 
\begin{align*}
    1-\langle \hat{u}_j^{(n)},u_j^{(n)} \rangle=1-|\langle \hat{u}_j^{(n)},u_j^{(n)}\rangle|\leq 1-\langle \hat{u}_j^{(n)},u_j^{(n)}\rangle^2=\frac{1}{2}\|\hat P_j^{(n)}-P_j^{(n)}\|_2^2\xrightarrow{\P} 0.
\end{align*}
The claim now follows again from Slutsky's lemma.
\end{proof}

\begin{proof}[Proof of Corollary \ref{CorIntro}]
By Proposition \ref{prop_control_events2} (i), we have that, with probability at least $1-C_2d^2(\log n)^{-p/4} n^{1-p/4}$, the bound $|\bar\eta_{kl}|\leq C_1\sqrt{(\log n)/n}$ holds for all $k,l\geq 1$. Hence, on this event, Corollary \ref{CorConcentrationMulti} implies that for all $j\geq 1$ such that \eqref{EqRelEffRank} is satisfied (with $c_1=1/(6C_1)$), we have $|\hat\lambda_j-\lambda_j|/\lambda_j\leq C_3  \sqrt{(\log n)/n}$. This gives the claim for the eigenvalues. The second claim follows similarly from Corollary~\ref{CorConcentrationMulti}. Alternatively, one may also use Proposition \ref{prop_control_events2} (ii), we omit the details.
\end{proof}

\begin{proof}[Proof of Proposition \ref{prop:factor:model}]
The fact that $\E\eta_j = 0$ is obvious. Since $\langle \Gamma u_j, u_j \rangle \geq \lambda_d(\Gamma)$, we have
\begin{align*}
\lambda_j(\Sigma) \geq \sum_{k = 1}^d \omega_k^2 \langle u_j, f_k \rangle^2 + \lambda_d(\Gamma).
\end{align*}
Minkowski's inequality now yields
\begin{align*}
\E^{1/p} |\langle X, u_j \rangle|^p \leq \E^{1/p}|\sum_{k \geq 1} \omega_k F_{k} \langle u_j, f_k \rangle |^p + \E^{1/p}|\langle \Gamma^{1/2}Y, u_j \rangle|^p.
\end{align*}
Treating the first part, Burkholder's inequality and Minkowski's inequality imply
\begin{align*}
\E^{2/p}|\sum_{k = 1}^d \omega_k F_{k} \langle u_j, f_k \rangle|^p &\leq C \sum_{k = 1}^d \omega_k^2 \E^{2/p}|F_{k}|^p \langle u_j, f_k \rangle^2\leq C \sum_{k = 1}^d \omega_k^2 \langle u_j, f_k \rangle^2.
\end{align*}
Similarly, by Burkholder's inequality, Minkowski's inequality, and the inequality $\| \Gamma^{1/2}u_j  \|^2 \leq  \lambda_1(\Gamma)$, we get
\begin{align*}
\E^{2/p}|\langle \Gamma^{1/2}Y, u_j\rangle|^p \leq  C \lambda_1(\Gamma).
\end{align*}
Using the inequality $(x + y)^2 \leq 2 x^2 + 2 y^2$ and the above, we arrive at
\begin{align*}
\E^{2/p} |\langle X, u_j \rangle|^p &\leq C \sum_{k = 1}^d \omega_k^2 \langle u_j, f_k \rangle^2 + C \lambda_1(\Gamma) \leq C \lambda_j(\Sigma)
\end{align*}
and the claim follows from the equation $\eta_j=\lambda_j(\Sigma)^{-1/2}\langle X, u_j \rangle$.
\end{proof}

\begin{proof}[Proof of Corollary \ref{CorrFunctionalDataProj}]
We start with the following more general result:
\begin{lemma}\label{cor:projectors:err}
Suppose that Setting \ref{iidSetting} 
holds with $p>4$. Let $r_0\geq 1$. Then we have
\begin{align*}
&\E \|\hat{Q}_r-Q_r\|_2^2 \leq \frac{2}{n}\sum_{s\neq r}\frac{\mu_s\mu_r}{(\mu_s-\mu_r)^2} \sum_{j \in \mathcal{I}_r}\sum_{k \in \mathcal{I}_s} n\E \bar{\eta}_{jk}^2 \\&+ C\frac{(\log n)^{3/2}}{n^{3/2}} \rr_r(\Sigma)\sum_{s\neq r}\frac{ \mu_r m_r\mu_s m_s}{(\mu_r-\mu_s)^2} + Cr_0^2 m_{r} (\log n)^{-p/4}n^{1-p/4}
\end{align*}
for all $r\geq 1$, such that $\mu_{r_0}\leq \mu_r/2$ and $\rr_r(\Sigma)\leq c\sqrt{n/\log n}$.
\end{lemma}

We can now deduce Corollary \ref{CorrFunctionalDataProj} from Lemma \ref{cor:projectors:err} as follows. We first note that the quantities $n \E \bar{\eta}_{jk}^2$ are uniformly bounded by \eqref{eq:prop:control:events:1}. Combining Lemma \ref{cor:projectors:err} (applied with $r_0 = \sqrt{n}$) with Proposition \ref{prop_control_events2} and \eqref{eq_convex_condi}, we have for $1 \leq j \leq \sqrt{n}(\log n)^{-5/2}$
\begin{align*}
\E\|\hat{P}_j-P_j\|_2^2 &\leq C\frac{j^2}{n} + C\frac{(\log n)^{5/2}}{n^{3/2}} j^3  +
C\frac{r_0^2 n^{1-p/4}}{(\log n)^{p/4}}.
%
\end{align*}
Setting 
Since $p \geq 16$, we obtain the desired bound, provided that we can show that $\lambda_{r_0} \leq \lambda_j/2$. Since $j/r_0 \leq (\log n)^{-5/2}$, this follows from the convexity (cf. \cite{cardot_mas_sarda_2007}): if $j$ is large enough, $k > j$ and \eqref{condi_convex} holds, then $j \lambda_j \geq k \lambda_k$. We conclude that $\lambda_{r_0} \leq \lambda_j/2$ for $n$ large enough, and the proof of Corollary \ref{CorrFunctionalDataProj} is complete.

It remains to prove Lemma \ref{cor:projectors:err}. First, note that from Theorem \ref{ThmExpEvMult}, it follows that, on the event $\mathcal{E}_{x,r_0}$, we have
\begin{equation*}
\|\hat{Q}_r-Q_r\|_2^2\leq 2\|R_rEQ_r\|_2^2+ Cx^3\rr_r(\Sigma) \sum_{s\neq r}\frac{m_r\mu_rm_s\mu_s}{(\mu_r-\mu_s)^2}.
\end{equation*}
Combining this with Proposition \ref{prop_control_events2} and using that for random variables $Y,Z\geq 0$ with $Y\leq C$ we have $\E Y \leq C \P\big(Y \geq Z+a \big) + a + \E Z$, $a>0$, the claim follows. 
\end{proof}

\subsection*{Acknowledgements}

We would like to thank the reviewer for valuable comments and suggestions,
which lead to considerable improvements, both in results and presentation.

\bibliography{revisionmsc}

\begin{thebibliography}{10}

\bibitem{MR3407216}
R.~Adamczak.
\newblock A note on the {H}anson-{W}right inequality for random vectors with
  dependencies.
\newblock {\em Electron. Commun. Probab.}, 20:no. 72, 13pp, 2015.

\bibitem{And}
T.~W. Anderson.
\newblock {\em An introduction to multivariate statistical analysis}.
\newblock Wiley, Hoboken, NJ, third edition, 2003.

\bibitem{baik:aop:2005}
J.~Baik, G.~Ben~Arous, and S.~P\'{e}ch\'{e}.
\newblock Phase transition of the largest eigenvalue for nonnull complex sample
  covariance matrices.
\newblock {\em Ann. Probab.}, 33(5):1643--1697, 2005.

\bibitem{MR3531673}
A.~S. Bandeira and R.~van Handel.
\newblock Sharp nonasymptotic bounds on the norm of random matrices with
  independent entries.
\newblock {\em Ann. Probab.}, 44(4):2479--2506, 2016.

\bibitem{bandeira2021matrix}
A.S. Bandeira, M.T. Boedihardjo, and R.~van Handel.
\newblock Matrix concentration inequalities and free probability, 2021.

\bibitem{MR878974}
H.~Baumg\"{a}rtel.
\newblock {\em Analytic perturbation theory for matrices and operators}.
\newblock Birkh\"{a}user Verlag, Basel, 1985.

\bibitem{Benaych-Georges:advances:2011}
F.~Benaych-Georges and R.~R. Nadakuditi.
\newblock The eigenvalues and eigenvectors of finite, low rank perturbations of
  large random matrices.
\newblock {\em Adv. Math.}, 227(1):494--521, 2011.

\bibitem{MR1477662}
R.~Bhatia.
\newblock {\em Matrix analysis}.
\newblock Springer-Verlag, New York, 1997.

\bibitem{B07}
R.~Bhatia.
\newblock {\em Perturbation bounds for matrix eigenvalues}.
\newblock Society for Industrial and Applied Mathematics (SIAM), Philadelphia,
  PA, 2007.
\newblock Reprint of the 1987 original.

\bibitem{blanchard_2007}
G.~Blanchard, O.~Bousquet, and L.~Zwald.
\newblock Statistical properties of kernel principal component analysis.
\newblock {\em Machine Learning}, 66(2-3):259--294, 2007.

\bibitem{MR3185193}
S.~Boucheron, G.~Lugosi, and P.~Massart.
\newblock {\em Concentration inequalities}.
\newblock Oxford University Press, Oxford, 2013.

\bibitem{brailovskaya2022universality}
T.~Brailovskaya and R.~van Handel.
\newblock Universality and sharp matrix concentration inequalities, 2022.

\bibitem{Cai2015_ptrf}
T.~Cai, Z.~Ma, and Y.~Wu.
\newblock Optimal estimation and rank detection for sparse spiked covariance
  matrices.
\newblock {\em Probab. Theory Related Fields}, 161(3-4):781--815, 2015.

\bibitem{cai:2020:aos:spiked:limit}
T.T. Cai, X.~Han, and G.~Pan.
\newblock {Limiting laws for divergent spiked eigenvalues and largest nonspiked
  eigenvalue of sample covariance matrices}.
\newblock {\em The Annals of Statistics}, 48(3):1255 -- 1280, 2020.

\bibitem{cai_yuan_2012}
T.T. Cai and M.~Yuan.
\newblock Minimax and adaptive prediction for functional linear regression.
\newblock {\em J. Amer. Statist. Assoc.}, 107(499):1201--1216, 2012.

\bibitem{cardot_mas_sarda_2007}
H.~Cardot, A.~Mas, and P.~Sarda.
\newblock C{LT} in functional linear regression models.
\newblock {\em Probab. Theory Related Fields}, 138(3-4):325--361, 2007.

\bibitem{C83}
F.~Chatelin.
\newblock {\em Spectral approximation of linear operators}.
\newblock Society for Industrial and Applied Mathematics (SIAM), Philadelphia,
  PA, 2011.
\newblock Reprint of the 1983 original.

\bibitem{Dauxois_1982}
J.~Dauxois, A.~Pousse, and Y.~Romain.
\newblock Asymptotic theory for the principal component analysis of a vector
  random function: some applications to statistical inference.
\newblock {\em J. Multivariate Anal.}, 12(1):136--154, 1982.

\bibitem{dedecker_prieur_2005}
J.~Dedecker and C.~Prieur.
\newblock New dependence coefficients. {E}xamples and applications to
  statistics.
\newblock {\em Probab. Theory Related Fields}, 132(2):203--236, 2005.

\bibitem{MR1463942}
J.~W. Demmel.
\newblock {\em Applied numerical linear algebra}.
\newblock Society for Industrial and Applied Mathematics (SIAM), Philadelphia,
  PA, 1997.

\bibitem{MR1009162}
N.~Dunford and J.~T. Schwartz.
\newblock {\em Linear operators. {P}art {I}}.
\newblock John Wiley \& Sons, Inc., New York, 1988.

\bibitem{MR2434306}
U.~Einmahl and D.~Li.
\newblock Characterization of {LIL} behavior in {B}anach space.
\newblock {\em Trans. Amer. Math. Soc.}, 360(12):6677--6693, 2008.

\bibitem{MR2308592}
N.~El~Karoui.
\newblock Tracy-{W}idom limit for the largest eigenvalue of a large class of
  complex sample covariance matrices.
\newblock {\em Ann. Probab.}, 35(2):663--714, 2007.

\bibitem{MR2738536}
N.~El~Karoui and A.~d'Aspremont.
\newblock Second order accurate distributed eigenvector computation for
  extremely large matrices.
\newblock {\em Electron. J. Stat.}, 4:1345--1385, 2010.

\bibitem{MR3588285}
E.~Gin\'{e} and R.~Nickl.
\newblock {\em Mathematical foundations of infinite-dimensional statistical
  models}.
\newblock Cambridge University Press, New York, 2016.

\bibitem{MR2102509}
E.~Gobet, M.~Hoffmann, and M.~Rei\ss.
\newblock Nonparametric estimation of scalar diffusions based on low frequency
  data.
\newblock {\em Ann. Statist.}, 32(5):2223--2253, 2004.

\bibitem{MR2332269}
P.~Hall and J.~L. Horowitz.
\newblock Methodology and convergence rates for functional linear regression.
\newblock {\em Ann. Statist.}, 35(1):70--91, 2007.

\bibitem{hall_hosseini_2009}
P.~Hall and M.~Hosseini-Nasab.
\newblock Theory for high-order bounds in functional principal components
  analysis.
\newblock {\em Math. Proc. Cambridge Philos. Soc.}, 146(1):225--256, 2009.

\bibitem{MR3099123}
N.~Hilgert, A.~Mas, and N.~Verzelen.
\newblock Minimax adaptive tests for the functional linear model.
\newblock {\em Ann. Statist.}, 41(2):838--869, 2013.

\bibitem{mas_hilgert_2013}
N.~Hilgert, A.~Mas, and N.~Verzelen.
\newblock Minimax adaptive tests for the functional linear model.
\newblock {\em Ann. Statist.}, 41(2):838--869, 2013.

\bibitem{hoermann_2010}
S.~H{\"o}rmann and P.~Kokoszka.
\newblock Weakly dependent functional data.
\newblock {\em Ann. Statist.}, 38(3):1845--1884, 2010.

\bibitem{MR2978290}
R.~A. Horn and C.~R. Johnson.
\newblock {\em Matrix analysis}.
\newblock Cambridge University Press, Cambridge, second edition, 2013.

\bibitem{MR3379106}
T.~Hsing and R.~Eubank.
\newblock {\em Theoretical foundations of functional data analysis, with an
  introduction to linear operators}.
\newblock John Wiley \& Sons, Ltd., Chichester, 2015.

\bibitem{I98}
I.~C.~F. Ipsen.
\newblock Relative perturbation results for matrix eigenvalues and singular
  values.
\newblock {\em Acta numerica}, 7:151--201, 1998.

\bibitem{I00}
I.~C.~F. Ipsen.
\newblock An overview of relative {$\sin\Theta$} theorems for invariant
  subspaces of complex matrices.
\newblock {\em J. Comput. Appl. Math.}, 123(1-2):131--153, 2000.

\bibitem{M16}
M.~Jirak.
\newblock Optimal eigen expansions and uniform bounds.
\newblock {\em Probab. Theory Related Fields}, 166(3-4):753--799, 2016.

\bibitem{JW20}
M.~Jirak and M.~Wahl.
\newblock Perturbation bounds for eigenspaces under a relative gap condition.
\newblock {\em Proc. Amer. Math. Soc.}, 148(2):479--494, 2020.

\bibitem{Johnstone_2000}
I.~M. Johnstone.
\newblock On the distribution of the largest eigenvalue in principal components
  analysis.
\newblock {\em Ann. Statist.}, 29(2):295--327, 2001.

\bibitem{jolliffe_book_2002}
I.T. Jolliffe.
\newblock {\em Principal component analysis}.
\newblock Springer-Verlag, New York, second edition, 2002.

\bibitem{K95}
T.~Kato.
\newblock {\em Perturbation theory for linear operators}.
\newblock Springer-Verlag, Berlin, 1995.
\newblock Reprint of the 1980 edition.

\bibitem{MR4210724}
V.~Koltchinskii.
\newblock Asymptotically efficient estimation of smooth functionals of
  covariance operators.
\newblock {\em J. Eur. Math. Soc. (JEMS)}, 23(3):765--843, 2021.

\bibitem{KG00}
V.~Koltchinskii and E.~Gin\'{e}.
\newblock Random matrix approximation of spectra of integral operators.
\newblock {\em Bernoulli}, 6(1):113--167, 2000.

\bibitem{KL16b}
V.~Koltchinskii and K.~Lounici.
\newblock Asymptotics and concentration bounds for bilinear forms of spectral
  projectors of sample covariance.
\newblock {\em Ann. Inst. Henri Poincar\'{e} Probab. Stat.}, 52(4):1976--2013,
  2016.

\bibitem{KL14}
V.~Koltchinskii and K.~Lounici.
\newblock Concentration inequalities and moment bounds for sample covariance
  operators.
\newblock {\em Bernoulli}, 23(1):110--133, 2017.

\bibitem{KL17b}
V.~Koltchinskii and K.~Lounici.
\newblock Normal approximation and concentration of spectral projectors of
  sample covariance.
\newblock {\em Ann. Statist.}, 45(1):121--157, 2017.

\bibitem{MR2111932}
R.~Latala.
\newblock Some estimates of norms of random matrices.
\newblock {\em Proc. Amer. Math. Soc.}, 133(5):1273--1282, 2005.

\bibitem{MR3878726}
R.~Latala, R.~van Handel, and P.~Youssef.
\newblock The dimension-free structure of nonhomogeneous random matrices.
\newblock {\em Invent. Math.}, 214(3):1031--1080, 2018.

\bibitem{M64}
A.~S. Markus.
\newblock Eigenvalues and singular values of the sum and product of linear
  operators.
\newblock {\em Uspehi Mat. Nauk}, 19:93--123, 1964.

\bibitem{MR2033885}
A.~Mas and L.~Menneteau.
\newblock Perturbation approach applied to the asymptotic study of random
  operators.
\newblock In {\em High dimensional probability, {III} ({S}andjberg, 2002)},
  pages 127--134. Birkh\"{a}user, Basel, 2003.

\bibitem{mas_complex_2014}
A.~Mas and F.~Ruymgaart.
\newblock High-dimensional principal projections.
\newblock {\em Complex Analysis and Operator Theory}, pages 1--29, 2014.

\bibitem{meister_2011}
A.~Meister.
\newblock Asymptotic equivalence of functional linear regression and a white
  noise inverse problem.
\newblock {\em Ann. Statist.}, 39(3):1471--1495, 2011.

\bibitem{merleved_rio_bernstein_2011}
F.~Merlev{\`e}de, M.~Peligrad, and E.~Rio.
\newblock A {B}ernstein type inequality and moderate deviations for weakly
  dependent sequences.
\newblock {\em Probab. Theory Related Fields}, 151(3-4):435--474, 2011.

\bibitem{boaz:aos:2008}
B.~Nadler.
\newblock Finite sample approximation results for principal component analysis:
  a matrix perturbation approach.
\newblock {\em Ann. Statist.}, 36(6):2791--2817, 2008.

\bibitem{nagaev1979}
S.V. Nagaev.
\newblock Large deviations of sums of independent random variables.
\newblock {\em Ann. Probab.}, 7(5):745--789, 10 1979.

\bibitem{MR3739989}
S.~O'Rourke, V.~Vu, and K.~Wang.
\newblock Random perturbation of low rank matrices: improving classical bounds.
\newblock {\em Linear Algebra Appl.}, 540:26--59, 2018.

\bibitem{debashis_2007}
D.~Paul.
\newblock Asymptotics of sample eigenstructure for a large dimensional spiked
  covariance model.
\newblock {\em Statist. Sinica}, 17(4):1617--1642, 2007.

\bibitem{ramsay_silverman_2005}
J.O. Ramsay and B.W. Silverman.
\newblock {\em Functional data analysis}.
\newblock Springer Series in Statistics. Springer, New York, second edition,
  2005.

\bibitem{RW17}
M.~Reiss and M.~Wahl.
\newblock Nonasymptotic upper bounds for the reconstruction error of {PCA}.
\newblock {\em Ann. Statist.}, 48(2):1098--1123, 2020.

\bibitem{MR4381414}
M.~Talagrand.
\newblock {\em Upper and lower bounds for stochastic processes}.
\newblock Springer, Cham, second edition, 2021.

\bibitem{T12}
T.~Tao.
\newblock {\em Topics in random matrix theory}.
\newblock American Mathematical Society, Providence, RI, 2012.

\bibitem{TV12}
T.~Tao and V.~Vu.
\newblock Random covariance matrices: universality of local statistics of
  eigenvalues.
\newblock {\em Ann. Probab.}, 40(3):1285--1315, 2012.

\bibitem{MR2946459}
J.~A. Tropp.
\newblock User-friendly tail bounds for sums of random matrices.
\newblock {\em Found. Comput. Math.}, 12(4):389--434, 2012.

\bibitem{V12}
R.~Vershynin.
\newblock Introduction to the non-asymptotic analysis of random matrices.
\newblock In {\em Compressed sensing}, pages 210--268. Cambridge Univ. Press,
  Cambridge, 2012.

\bibitem{VVU2011}
V.~Vu.
\newblock Singular vectors under random perturbation.
\newblock {\em Random Structures Algorithms}, 39(4):526--538, 2011.

\bibitem{wang2017}
W.~Wang and J.~Fan.
\newblock Asymptotics of empirical eigenstructure for high dimensional spiked
  covariance.
\newblock {\em Ann. Statist.}, 45(3):1342--1374, 06 2017.

\bibitem{MR353419}
F.~T. Wright.
\newblock A bound on tail probabilities for quadratic forms in independent
  random variables whose distributions are not necessarily symmetric.
\newblock {\em Ann. Probability}, 1(6):1068--1070, 1973.

\end{thebibliography}
\bibliographystyle{plain}

\end{document}